\documentclass[10pt]{amsart}

\usepackage{amssymb}
\usepackage{amsmath}
\usepackage{enumerate}
\usepackage{amsfonts}
\usepackage{color}
\usepackage{latexsym}
\usepackage{mathtools}
\usepackage{amstext,amssymb,amsopn,amsthm,mathrsfs,wasysym,dsfont,comment,tikz}

\usetikzlibrary{shapes.geometric}
\usetikzlibrary{patterns}

\usepackage{tikz}
\usetikzlibrary{arrows}

\allowdisplaybreaks

\newtheorem{thm}{Theorem}[section]
\newtheorem{coro}[thm]{Corollary}
\newtheorem{conj}[thm]{Conjecture}
\newtheorem{lema}[thm]{Lemma}
\newtheorem{propo}[thm]{Proposition}
\newtheorem{remark}[thm]{Remark}

\numberwithin{equation}{section}

\newtheorem{bigthm}{Theorem}

\newtheorem{bigcor}[bigthm]{Corollary}

\newtheorem{lem}[thm]{Lemma}

\newcommand*{\DMO}[1]{\expandafter\DeclareMathOperator\csname #1\endcsname {#1}}
\DeclarePairedDelimiter\abs{\lvert}{\rvert}
\newcommand{\ind}[1]{\mathds{1}_{{#1}}}

\DeclareMathOperator{\arcosh}{arcosh}
\DeclareMathOperator{\pv}{P.V.}

\renewcommand{\a}{\alpha}

\renewcommand{\b}{\beta}

\newcommand{\RR}{\mathbb{R}}

\newcommand{\NN}{\mathbb{N}}
\newcommand{\ZZ}{\mathbb{Z}}

\allowdisplaybreaks

\author[A.\ Nowak]{Adam Nowak}
\address[Adam Nowak]{Institute of Mathematics,
Polish Academy of Sciences,
\'Sniadeckich 8,
00-656 Warszawa, Poland}
\email{anowak@impan.pl}

\author[L.\ Roncal]{Luz Roncal}
\address[Luz Roncal]{
BCAM - Basque Center for Applied Mathematics,
48009 Bilbao, Spain
and
Ikerbasque, Basque Foundation for Science,
48011 Bilbao, Spain
and
Universidad del Pa\'is Vasco / Euskal Herriko Unibertsitatea,
48080 Bilbao, Spain}
\email{lroncal@bcamath.org}

\author[T.Z.\ Szarek]{Tomasz Z.\ Szarek}
\address[Tomasz Z. Szarek]{
BCAM - Basque Center for Applied Mathematics,
48009 Bilbao, Spain 
and
Instytut Matematyczny,
Uniwersytet Wroc{\l}awski,
Plac Grunwaldzki 2,
50-384 Wroc{\l}aw,
Poland}
\email{tzszarek@bcamath.org}



\begin{document}

\title[Endpoint estimates and optimality]{Endpoint estimates and optimality 
 \\  for the generalized spherical maximal operator\\  on radial functions}

\begin{abstract}
We find sharp conditions for the maximal operator associated with generalized spherical mean Radon transform on
radial functions $M^{\a,\b}_t$ to be bounded on power weighted Lebesgue spaces.
Moreover, we also obtain the corresponding endpoint results in terms of optimal
power weighted weak and restricted weak type estimates.
\end{abstract}

\maketitle

\section{Introduction and main results}

Let $\a > -1$ and $\a+\b > -1/2$. Denote $d\mu_{\a}(x) = x^{2\a+1}dx$.
We interpret the generalized spherical mean operator on radial functions as the one-dimensional two-parameter integral operator
$$
M_{t}^{\a,\b}f(x) = \int_0^{\infty} K_t^{\a,\b}(x,z) f(z)\, d\mu_{\a}(z), \qquad x > 0,
$$
where $t > 0$ and the kernel is given by 
$$
K_t^{\a,\b}(x,z) = \frac{2^{\a+\b}\Gamma(\a+\b+1)}{t^{\a+\b}(xz)^{\a}} \int_0^{\infty} J_{\a+\b}(ty)
	J_{\a}(xy) J_{\a}(zy) y^{1-\a-\b}\, dy,
$$
with $J_{\nu}$ standing for the Bessel function of the first kind and order $\nu$.
This kernel is well defined for $t,x,z>0$ such that, in general, $t \neq |x-z|$ and $t \neq x+z$.
Note that the triple Bessel function integral here converges absolutely when $\a+\b > 1/2$,
but for $-1/2 < \a + \b \le 1/2$ the convergence at infinity is only conditional, in the Riemann sense.
In the limiting case $\a+\b=-1/2$, not considered in this paper, the integral diverges, nevertheless $K_t^{\a,\b}(x,z)$ and
$M_t^{\a,\b}$ can still be defined in a suitable way; see Appendix.

The principal focus of this paper is on the maximal operator
$$
M^{\a,\b}_{*}f = \sup_{t > 0} \big| M_t^{\a,\b}f \big|
$$
with continuous ranges of the parameters $\alpha > -1$, $\beta > -1/2 -\alpha$.
Our aim is to prove complete sharp descriptions of strong and weak/restricted weak type
boundedness of $M^{\a,\b}_*$ with respect to power weights on $\mathbb{R}_+ := (0,\infty)$.
For background, motivations and importance of $M_t^{\a,\b}$ and $M_{*}^{\a,\b}$, in particular crucial connections to
ordinary and generalized spherical means, and to classical PDE problems,
we refer to \cite{CNR,CNR2} and references given there.

The main results of this paper are the following (here and elsewhere, by weakening a strict inequality we mean
replacing ``$<$'' by ``$\le$''; analogously we understand strictening a weak inequality).
\begin{bigthm} \label{thm:strong}
Assume that $\a > -1$ and $\a+\b>-1/2$. Let $1 \le p < \infty$ and $\delta \in \mathbb{R}$.
Then the maximal operator $M_{*}^{\a,\b}$ is bounded on $L^p(\mathbb{R}_+,x^{\delta}dx)$ if and only if $p>1$ and
\begin{equation} \label{cnn}
\frac{1}p < \a + \b + \frac{1}2
\end{equation}
and
\begin{equation*}
\left\{
\begin{array}{c}
-(\b p \wedge 1) < \delta < \big(2\a + (\b \wedge 1) + 1\big)p -1 \\
\textrm{\footnotesize{with the first inequality weakened when $\b p < 1$}}
\end{array}
\right\}.
\end{equation*}
\end{bigthm}

\begin{bigthm} \label{thm:weak}
Assume that $\a > -1$ and $\a+\b>-1/2$. Let $1 \le p < \infty$ and $\delta \in \mathbb{R}$.
\begin{itemize}
\item[(a)]
The maximal operator $M^{\a,\b}_*$ is of weak type $(p,p)$ with respect to the measure space $(\mathbb{R}_+,x^{\delta}dx)$
if and only if
\begin{equation} \label{cn5n}
\left\{
\begin{array}{c}
\frac{1}p < \a + \b + \frac{1}2 \\
\textrm{\footnotesize{with the inequality weakened when $\a + \b = 1/2$ and $\b \in \mathbb{Z}$}}
\end{array}
\right\}
\end{equation}
and
\begin{equation*}
\left\{
\begin{array}{c}
-(\b p \wedge 1) < \delta < \big(2\a + (\b \wedge 1) + 1\big)p -1 \\
\textrm{\footnotesize{with the first inequality weakened when $\b p < 1$}} \\
\textrm{\footnotesize{and the second inequality weakened when $p=1$}}
\end{array}
\right\}.
\end{equation*}
\item[(b)]
The maximal operator $M^{\a,\b}_*$ is of restricted weak type $(p,p)$ with respect to the measure space $(\mathbb{R}_+,x^{\delta}dx)$
if and only if
\begin{equation} \label{cnnn}
\left\{
\begin{array}{c}
\frac{1}p \le \a + \b + \frac{1}2 \\
\textrm{\footnotesize{with the inequality strictened when $\a + \b = 1/2$ and $\b \notin \mathbb{Z}$}}
\end{array}
\right\}
\end{equation}
and
\begin{equation*}
\left\{
\begin{array}{c}
-(\b p \wedge 1) < \delta \le \big(2\a + (\b \wedge 1) + 1\big)p -1 \\
\textrm{\footnotesize{with the first inequality weakened when $\b p < 1$}} 
\end{array}
\right\}.
\end{equation*}
\end{itemize}
\end{bigthm}

An instant consequence of Theorems \ref{thm:strong} and \ref{thm:weak} is the following corollary extracting
endpoint results related to Theorem \ref{thm:strong} that are contained in Theorem \ref{thm:weak}.
\begin{bigcor} \label{cor:incr}
Assume that $\a > -1$ and $\a+\b > -1/2$. Let $1 \le p < \infty$ and $\delta \in \mathbb{R}$.
\begin{itemize}
\item[(a)]
The maximal operator $M^{\a,\b}_*$ is of weak type $(p,p)$ with respect to the measure space $(\mathbb{R}_+,x^{\delta}dx)$,
but is not bounded on $L^p(\mathbb{R}_+,x^{\delta}dx)$, if and only if $p=1$ and
\begin{equation*}
\left\{
\begin{array}{c}
\a + \b > \frac{1}2 \\
\textrm{\footnotesize{with the inequality weakened when $\b \in \mathbb{Z}$}}
\end{array}
\right\}
\end{equation*}
and
\begin{equation} \label{c0corc}
\left\{
\begin{array}{c}
-(\b \wedge 1) < \delta \le 2\a + (\b \wedge 1) \\
\textrm{\footnotesize{with the first inequality weakened when $\b < 1$}}
\end{array}
\right\}.
\end{equation}

\item[(b)]
The maximal operator $M^{\a,\b}_*$ is of restricted weak type $(p,p)$, but not of weak type $(p,p)$,
with respect to the measure space $(\mathbb{R}_+,x^{\delta}dx)$, if and only if $p>1$ and either
\begin{equation} \label{c1corc}
\frac{1}p < \a + \b + \frac{1}2 \quad \textrm{and} \quad
\left\{
\begin{array}{c}
\delta = \big( 2\a + (\b \wedge 1) + 1 \big)p -1 > -(\b p \wedge 1) \\
\textrm{\footnotesize{with the inequality weakened when $\b p < 1$}}
\end{array}
\right\}
\end{equation}
or
\begin{equation} \label{c2corc}
\frac{1}p = \a + \b + \frac{1}{2} \quad \textrm{and} \quad
\left\{
\begin{array}{c}
-(\b p \wedge 1) < \delta \le \big(2\a + (\b \wedge 1) + 1\big)p -1 \\
\textrm{\footnotesize{with the first inequality weakened when $\b p < 1$}}
\end{array}
\right\}.
\end{equation}
\end{itemize}
\end{bigcor}

Both $M_t^{\a,\b}$ and $M^{\a,\b}_{*}$ have been extensively investigated recently by
Ciaurri and two of the authors \cite{CNR,CNR2}.
Boundedness of $M^{\a,\b}_{*}$ on power weighted $L^p$ spaces for $1 < p < \infty$ was studied in \cite{CNR2}.
Sufficient conditions for the boundedness were obtained in \cite[Theorem 1.5]{CNR2}, while necessary conditions can be found in
\cite[Proposition 1.6]{CNR2}. Theorem \ref{thm:strong} provides a sharp refinement of those results, as well as an
extension to $p=1$. In particular, it shows that \cite[Theorem 1.5]{CNR2} is not yet optimal for some $\a$ and $\b$
(contrary to what the authors of \cite{CNR2} presumed), whereas \cite[Proposition 1.6]{CNR2}
turns out to be already optimal. The discrepancy between the two results from \cite{CNR2},
and thus the contribution of the present paper to power weighted $L^p$-boundedness of $M^{\a,\b}_*$,
is described in detail in \cite[p.\,1600]{CNR2}, see also \cite[Remark~4.3]{CNR2}.
For the readers' convenience, we note that comparing to the results in \cite{CNR2}, the novelty of Theorem~\ref{thm:strong}
is the $L^p(x^{\delta}dx)$-boundedness of $M^{\a,\b}_*$ when, assuming that $\a > -1$, $\a+\b > -1/2$ and $1<p<\infty$,
either
$\b \le 0$, $\a+\b \le 1/2$, $-\b \notin \mathbb{N}$ and $\delta = -\b p$,
or
$0 < \b < 1$, $\a > -1/2$ and $-(\b p \wedge 1) < \delta < -\b\slash [(\a+\b+1/2)\wedge 1]$
(with the lower inequality for $\delta$ weakened if $\b p < 1$).
Furthermore, Theorem \ref{thm:strong} contains the lack of $L^1(x^{\delta}dx)$-boundedness of $M^{\a,\b}_*$
for any $\delta \in \mathbb{R}$.
We remark that for some very special choices of $\a$, $\b$ and $\delta$ the $L^p(x^{\delta}dx)$-boundedness of $M^{\a,\b}_*$,
and also its sharpness in certain cases, was obtained before \cite{CNR2} in e.g.\ \cite{DMO,DMO2}, see \cite[Section 1]{CNR2}
and results and references invoked there.

On the other hand, Theorem \ref{thm:weak} delivers a complete characterization of power weighted weak and restricted
weak type estimates. This can be seen as an endpoint result related to Theorem A, see Corollary \ref{cor:incr}.
The results of Theorem \ref{thm:weak} are new up to several very special subcases which we now summarize in detail.
To this end, all the mapping properties are understood with respect to the measure space $(\mathbb{R}_+,x^{\delta}dx)$.
In case $\a+\b=1/2$, $\a \ge -1/2$ and $\delta = 2\a+1$ Colzani et al.\ \cite{CoCoSt} proved that $M^{\a,\b}_*$ is of restricted
weak type $(p,p)$ for $p =(4\a+4)/(2\a+3)$ and $\a > -1/2$ (this corresponds to Condition \eqref{c1corc} in Corollary \ref{cor:incr}),
while for $\a=-1/2$ they obtained weak type $(1,1)$ of $M^{\a,\b}_*$.
See \cite[Corollary 3.5]{CoCoSt} and \cite[Theorem 3.1]{CoCoSt}, respectively.
In case $\a+\b = 1/2$ and $-\b \in \mathbb{N}$, Duoandikoetxea et al.\ \cite{DMO2} proved that $M^{\a,\b}_*$ is of weak type $(1,1)$
provided that$^{\dag}$ $-\b \le \delta \le 2\a+\b$, see \cite[Theorem 3.2]{DMO2}.
\footnotetext{$\dag$ There is a misprint in the condition for weak type $(1,1)$ in \cite[Theorem 3.2]{DMO2}, the upper inequality
should be weak.}
We note that \cite[Theorem 3.2]{DMO2} contains also a restricted weak type result for $M^{\a,\b}_*$, which is covered by the
above mentioned earlier and more general result of Colzani et al.\ \cite[Corollary 3.5]{CoCoSt}.
In another paper \cite{DMO} Duoandikoetxea et al.\ considered the case when $\b = 0$ and $\a \ge 0$ is half-integer.
For $\b=0$ and half-integer $\a \ge 1/2$ they proved that $M^{\a,\b}_*$ is of weak type $(1,1)$ if and only if
$0 \le \delta \le 2\a$ (this corresponds to Condition \eqref{c0corc} in Corollary \ref{cor:incr}).
Also for $\b=0$ and half-integer $\a \ge 1/2$, they obtained restricted weak type $(p,p)$ of $M^{\a,\b}_*$
for $p>1$ and $\delta = (2\a+1)p-1$ (which corresponds to Condition \eqref{c1corc} in Corollary \ref{cor:incr}).
In the case $\a=\b=0$, the authors of \cite{DMO} proved that $M^{\a,\b}_*$ is of restricted weak type $(2,2)$ if and only
if $0 \le \delta \le 1$ (this corresponds to Condition \eqref{c2corc} in Corollary \ref{cor:incr}) and
of restricted weak type $(p,p)$ if $p>2$ and $\delta = p-1$ (which corresponds to Condition \eqref{c1corc} in Corollary \ref{cor:incr}).
For all the results from \cite{DMO} just mentioned, see \cite[Theorem 1.1]{DMO}.
We note also an earlier result of Leckband \cite{Le}, who proved that $M^{\a,\b}_*$ is of restricted weak type $(2,2)$
when $\a=\b=0$ and $\delta = 1$. There is also a more earlier results of Bourgain \cite{Bourgain0}, which implies
restricted weak type $(p,p)$, with $p=(2\a+2)/(2\a+1)$,
of $M^{\a,\b}_*$ when $\b=0$ and $\a \ge 1/2$ is half-integer, see the related comments in
\cite[Section 1]{DMO}. The same Bourgain's work \cite{Bourgain0} also implies the above mentioned restricted weak type $(p,p)$
result from \cite{CoCoSt} in case $\alpha$ is half-integer, see the related comment in \cite[p.\,28]{CoCoSt}.

The proofs of the sufficiency parts in Theorems \ref{thm:strong} and \ref{thm:weak} are based on precise absolute estimates
of the kernel $K_t^{\a,\b}(x,z)$ obtained in \cite[Theorem 3.3]{CNR} (see also \cite[Theorem 2.1]{CNR2})
and suitable decomposition of the maximal
operator $M^{\a,\b}_{*}$ into parts, which can be controlled by means of auxiliary `special' operators that are more
convenient to analyze. The decomposition and the special operators are much the same as in \cite{CNR2}, and have roots
in the earlier papers by Duoandikoetxea, Moyua and Oruetxebarria \cite{DMO,DMO2}.
Here, however, we had to deal with some substantial technical difficulties, one
of them being the problem of finding sharp refinements/extensions of the results on the special operators known so far.
This strategy works for most choices of $\a$ and $\b$, nevertheless in some cases we found it convenient to use
somewhat different method inspired by the techniques from Colzani, Cominardi and Stempak \cite{CoCoSt},
where slightly different splittings of the regions involved for the analysis of the corresponding kernel are used. 
The proofs of the necessity parts in Theorems \ref{thm:strong} and \ref{thm:weak} rely on constructing suitable
counterexamples. This is possible thanks to a comprehensive analysis of the kernel $K_t^{\a,\b}(x,z)$ done in \cite[Section 3]{CNR}.
We emphasize that the present paper does not contribute in fundamentally novel techniques comparing
to the previous works \cite{CNR2,CoCoSt,DMO,DMO2}. The significant difficulty it overcomes lies in delicate
refinements/adjustments of the existing methods and an extended meticulous analysis that
turned out to be quite challenging. Note also that, in general,
the proofs of Theorems \ref{thm:strong} and \ref{thm:weak} refer heavily to the notation and reasonings
from \cite{CNR2}, so in this sense the present paper is not self-contained.
Otherwise we would need to essentially repeat some definitions and estimates from \cite{CNR2},
which would increase the volume of this paper without adding anything new.

An interesting, but more difficult and technical, problem would be to study boundedness of $M^{\a,\b}_*$ in the context of
general Lorentz spaces, say from $L^{p,q_1}(\mathbb{R}_+,x^{\delta}dx)$ to $L^{p,q_2}(\mathbb{R}_+,x^{\delta}dx)$ with
$1 \le q_1 \le q_2 \le \infty$. In particular, this could lead to finer endpoint results related to Theorem \ref{thm:strong}
than those contained in Corollary \ref{cor:incr}. However, all this requires a deeper sophisticated analysis, which is beyond the scope
of this paper. Note that even for specific choices of $\a$ and $\b$ such results do not seem to be available in the literature.

We now point out some interesting consequences of Theorems \ref{thm:strong} and \ref{thm:weak}. 
Let $n \ge 2$ and recall the generalized spherical means transformation in $\mathbb{R}^n$
$$
M^{\b}f(x,t) = \mathcal{F}^{-1} \big( m_{\b}(t|\cdot|) \mathcal{F}f \big)(x),
$$
where $\mathcal{F}$ is the Fourier transform in $\mathbb{R}^n$ and the radial multiplier is given via
\begin{equation}
\label{eq:Fmultiplier}
m_{\b}(s) = 2^{\b+n/2-1} \Gamma(\b+n/2) \frac{J_{\b+n/2-1}(s)}{s^{\b+n/2-1}}, \qquad s > 0.
\end{equation}
The parameter $\b$ can, in general, be a complex number excluding $\b = -n/2, -n/2-1, -n/2-2,\ldots$.
For $\b=0$ one recovers the classical spherical means, i.e.\
$M^0f(x,t)$ returns the mean value of $f$ on the sphere centered at $x$ and of radius $t$.

For the maximal operator $M^{\b}_*f = \sup_{t >0} |M^{\b}f(\cdot,t)|$ Stein \cite{Stein} proved the following.
\begin{thm}[{\cite{Stein}}] \label{thm:Stein}
Let $n \ge 3$. Then $M^{\b}_*$ is bounded on $L^p(\mathbb{R}^n)$ provided that
\begin{equation} \label{cndthm11}
1 < p \le 2 \;\; \textrm{and} \;\; \b > 1-n+\frac{n}p, \quad \textrm{or} \quad p > 2 \;\; \textrm{and} \;\; \b > \frac{2-n}p.
\end{equation}
\end{thm}

This result was enhanced in the sense of admitted parameters and dimensions by subsequent authors:
Bourgain \cite{Bourgain}, Mockenhaupt, Seeger and Sogge \cite{MSS} and recently by Miao,
Yang and Zheng \cite{MYZ}$^{\ddag}$ 
{\footnote{$\ddag$ In \cite{MYZ} the authors use the name \textit{Stein's maximal spherical means}
for the maximal operator $M_*^{\b}f$.}}\!\!,
see the historical comments in \cite[p.\,4272]{MYZ}.
All these refinements pertain to $p > 2$ and can be stated altogether as follows, cf.\ \cite[Theorem 1.1]{MYZ}.
\begin{thm}[{\cite{MYZ}}] \label{thm:MYZ}
Let $n \ge 2$. Then $M^{\b}_*$ is bounded on $L^p(\mathbb{R}^n)$ provided that
\begin{equation} \label{cndthm12}
2 < p \le \frac{2n+2}{n-1} \;\; \textrm{and} \;\; \b > \frac{1-n}{4}+\frac{3-n}{2p}, \quad \textrm{or} \quad
	p > \frac{2n+2}{n-1} \;\; \textrm{and} \;\; \b > \frac{1-n}p.
\end{equation}
\end{thm}

Theorem \ref{thm:Stein} is known to be optimal for $1<p\le2$. The range of $\b$ in Theorem \ref{thm:MYZ}, for $p > 2$,
is strictly wider than in Theorem \ref{thm:Stein}.
However, according to our best knowledge, it is not known whether it is already optimal when $n\ge3$.
On the other hand, a recent \cite{GWZ} striking proof of the so-called \textit{local smoothing conjecture} in the case
$n=2$ allows to enlarge the range of $\beta$ in Theorem \ref{thm:MYZ}, see Remark \ref{rem:Sogge} below
for more details.
We remark that both Theorems \ref{thm:Stein} and \ref{thm:MYZ} were originally proved for complex $\b$, but
for our purposes it is enough to state them for real values of the parameter.
An important open problem is to extend, if possible, Theorem~\ref{thm:MYZ} to the optimal range of $\beta$ and $p>2$.

A restriction of $M^{\b}$ to radially symmetric functions is still of great interest and, moreover,
admits a more explicit finer analysis that potentially leads to more general or stronger theorems.
Actually, the maximal operator $M^{\a,\b}_{*}$ we study generalizes the restriction of $M^{\b}_*$ to radial functions since,
in a sense, it covers a continuous range of dimensions $n=2\a+2$, $\a > -1$. Formally, for a radial function
$f=f_0(|\cdot|)$ in $L^2(\mathbb{R}^n)$, $n \ge 2$, $M^{\b}f(x,t)$ is for each $t>0$ a radial function in $x \in \mathbb{R}^n$
whose profile is given by $M_t^{n/2-1,\b}f_0$; see \cite[Corollary 4.2]{CNR}. Clearly, the maximal operators $M^{\b}_*$
and $M^{\a,\b}_*$ are connected in the same way. Thus, $M^{\b}_{*}$ is bounded on
$L^p_{\textrm{rad}}(\mathbb{R}^n,|\cdot|^{\gamma}dx)$ if and only if $M^{n/2-1,\b}_{*}$ is bounded on
$L^p(\mathbb{R}_+,x^{\gamma}d\mu_{n/2-1})$; here and elsewhere the subscript ``rad'' indicates the subspace of radial functions.
Analogous relations hold for weak and restricted weak type boundedness.

The announced consequences of Theorems \ref{thm:strong} and \ref{thm:weak} pertain to the operator $M^{\b}_*$.
Before stating them, however, it is convenient to specify the two theorems to the natural weight $\delta = 2\a+1$.
The first corollary below is a sharp improvement of \cite[Corollary 1.7]{CNR2}, where
the enhancement relies on including $p=1$ for the negative part, and, under Condition \eqref{cnn},
on including parameters satisfying $-\b/(2\a+1) = 1/p$, $0 < -\b \notin \mathbb{N}$ and $\a+\b \le 1/2$ for the positive result.

\begin{coro} \label{cor:17sh}
Assume that $\a > -1$ and $\a+\b > -1/2$. Let $1 \le p < \infty$.
Then $M^{\a,\b}_*$ is bounded on $L^p(\mathbb{R}_+,d\mu_{\a})$
if and only if $p > 1$ and Condition \eqref{cnn} is satisfied and
\begin{equation*}
\left\{
\begin{array}{c}
\frac{-\b}{2\a+1} \le \frac{1}p < 1 - \frac{1-\b}{2\a+2} \\
\textrm{\footnotesize{with the lower bound suppressed when $\b \ge 0$}} \\
\textrm{\footnotesize{and with the upper bound suppressed when $\b \ge 1$}}
\end{array}
\right\}.
\end{equation*}
\end{coro}

\begin{coro} \label{cor:wrw}
Assume that $\a > -1$ and $\a+\b > -1/2$. Let $1 \le p < \infty$.
\begin{itemize}
\item[(a)]
The maximal operator $M^{\a,\b}_*$ is of weak type $(p,p)$ with respect to $(\mathbb{R}_+,d\mu_{\a})$
if and only if Condition \eqref{cn5n} is satisfied and
\begin{equation*}
\left\{
\begin{array}{c}
\frac{-\b}{2\a+1} \le \frac{1}p < 1 - \frac{1-\b}{2\a+2} \\
\textrm{\footnotesize{with the second inequality weakened when $p=1$}} \\
\textrm{\footnotesize{and with the lower bound suppressed when $\b \ge 0$}} \\
\textrm{\footnotesize{and with the upper bound suppressed when $\b \ge 1$}}
\end{array}
\right\}.
\end{equation*}
\item[(b)]
The maximal operator $M^{\a,\b}_*$ is of restricted weak type $(p,p)$ with respect to $(\mathbb{R}_+,d\mu_{\a})$
if and only if Condition \eqref{cnnn} is satisfied and
\begin{equation*}
\left\{
\begin{array}{c}
\frac{-\b}{2\a+1} \le \frac{1}p \le 1 - \frac{1-\b}{2\a+2} \\
\textrm{\footnotesize{with the lower bound suppressed when $\b \ge 0$}} \\
\textrm{\footnotesize{and with the upper bound suppressed when $\b \ge 1$}}
\end{array}
\right\}.
\end{equation*}
\end{itemize}
\end{coro}

With the above corollaries, we can state a sharp improvement of \cite[Corollary 1.8]{CNR2} and an analogous
result for weak and restricted weak type inequalities.
First, however, we must confess that the statement of \cite[Corollary 1.8]{CNR2} missed the global condition
$\b > 1/p+(1-n)/2$ coming from \cite[(1.3)]{CNR2}.
In consequence, \cite[Figure 1]{CNR2} and \cite[Conjecture~1.9]{CNR2} should also be revised,
which is done below, see Figure \ref{fig:sh} and Conjecture \ref{conj:19sh}.
\begin{coro} \label{cor:18sh}
Let $n \ge 2$, $\beta>(1-n)/2$ and $1 \le p < \infty$. Then $M^{\b}_*$ is bounded on $L^p_{\textrm{rad}}(\mathbb{R}^n)$
if and only if $p > 1$ and
\begin{align} \label{id:151}
\begin{cases}
\b > 1 -n + \frac{n}p & \textrm{when} \;\;  p < 2, \\
\b > \frac{1}p + \frac{1-n}2 & \textrm{when} \;\; 2 \le p \le \frac{2n}{n-1}, \\
\b \ge \frac{1-n}p & \textrm{when} \;\; p > \frac{2n}{n-1}.
\end{cases}
\end{align}
\end{coro}

\begin{figure}
\centering
\begin{tikzpicture}[scale=3.6]
\small{
\draw[arrows=-angle 60] (0,-1.54) -- (0,1.2);
\draw[arrows=-angle 60] (0,0) -- (2.33,0);
\node at (-0.1,0) {$O$};
\node at (2.05,1.05) {$P$};
\node at (-0.1,1.15) {$ \beta$};
\node at (2.33,-0.1) {$\tfrac{1}{p}$};
\node at (-0.1,1) {$1$};
\node at (2,-0.1) {$1$};
}
\fill[pattern=dots, pattern color=lightgray]
(0,1) %
--(0,0) 
-- (6/8,-9/8) 
-- (1,-1) 
-- (2,1) 
-- (0,1); 

\tiny{
\node at (-0.2,-9/8)  {$-\tfrac{(n-1)^2}{2n}$}; 
\node at (0.58,0.1) {$\tfrac{n-1}{2n+2}$}; 
\node at (1,0.1) {$\tfrac{1}{2}$}; 
\node at (-0.2,-0.85) {$-\tfrac{(n-1)^2}{2n+2}$}; 
\node at (-0.15,-1.01) {$\tfrac{2-n}{2}$}; 
\node at (0.77,0.1) {$\tfrac{n-1}{2n}$}; 
\node at (-0.15,-0.5) {$\tfrac{3-n}{2}$}; 
\node at (-0.15,-1.5) {$\tfrac{1-n}2$};
}
\small{
\node at (0.5,-0.85) {$A$};
\node at (6/8, -1.2) {$C$};
\node at (1.05,-1.05) {$B$};
\draw[very thin, dashed] (0,1) -- (2,1);
\draw[very thin, dashed] (2,0) -- (2,1);
\draw[very thin, dashed] (0,-9/8) -- (6/8,-9/8); 
\node[regular polygon,regular polygon sides=4,draw=black, fill=black, inner sep=0pt,minimum size=5pt] (b) at (6/8,-9/8) {}; 
\draw[very thin, dashed] (6/8,0) -- (6/8,-9/8); 
\draw[very thin, dashed] (0,-0.9) -- (0.6,-0.9); 
\draw[very thin, dashed] (0.6,0) -- (0.6,-0.9); 
\draw[very thin, dashed] (0,-1) -- (1,-1); 
\node[regular polygon,regular polygon sides=4,draw=black, fill=black, inner sep=0pt,minimum size=5pt] (b) at (1,-1) {}; 
\draw[very thin, dashed] (1,0) -- (1,-1); 
\draw[very thick] (0,0) -- (0.6,-0.9); 
\draw[very thick, dashed] (0.6,-0.9) -- (1,-1); 
\draw[very thick, dotted] (1,-1) -- (2,1); 
\node[regular polygon,regular polygon sides=3, draw=black, fill=black, inner sep=0pt,minimum size=6pt] (b) at (2,1) {}; 
\draw[very thick] (0.6,-0.9) -- (6/8,-9/8); 
\draw[very thick, dashed] (0,0) -- (1,-1); 
\draw[very thick, dotted] (6/8,-9/8) -- (1,-1); 
\draw[very thin, dashed] (0,-0.5) -- (0.02,-0.5);
\draw[very thin, dashed] (0,-1.5) -- (0.02,-1.5);
}
\end{tikzpicture} 

\caption{Regions visualizing differences between Stein's conditions \eqref{cndthm11} (region $OBP$),
the less restrictive conditions \eqref{cndthm12} due to Miao et al.\ (region $OABP$),
and the least restrictive conditions from our radial case in Corollary \ref{cor:18sh} and Conjecture \ref{conj:19sh}
for the general case (region $OCBP$), respectively. Observe that $A$ is
\textit{Bourgain-Demeter's point} and $C$ is \textit{Sogge's conjecture point}, see Remark \ref{rem:Sogge}.
The operator $M^{\b}_*$ considered on the subspace of radial functions is of strong type $(p,p)$ for $(1/p,\b)$
in the dotted region and on $OC$ excluding $C$, of weak but not strong type $(p,p)$ when $(1/p,\b) = P=(1,1)$, and
of restricted weak type $(p,p)$ but not weak $(p,p)$ for $(1/p,\b)$ on $CB$ including endpoints and on $BP$ excluding $P$.
Picture for $n=4$, different axes scaling.} \label{fig:sh}
\end{figure}
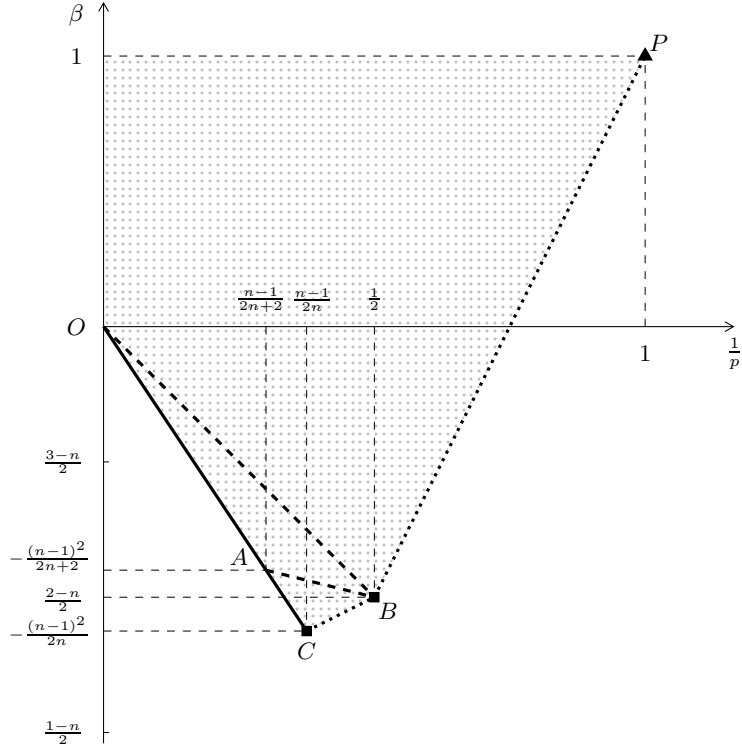

Observe that the assumption $\beta>(1-n)/2$ in Corollary \ref{cor:18sh}
(coming from the assumption $\a+\b > -1/2$ in Corollary \ref{cor:17sh})
is strictly weaker than Condition \eqref{id:151}.
Further, note that for $2 < p < \frac{2n+2}{n-1}$ the conditions of Corollary \ref{cor:18sh} are strictly less restrictive
than Condition \eqref{cndthm12} due to Miao et al. This can be seen in Figure~\ref{fig:sh}, where
the endpoint results related to Corollary \ref{cor:18sh} which are contained in Corollary \ref{cor:211} are also reflected.
We believe this indicates non-optimality of the yet known results on $L^p$-boundedness of $M^{\b}_*$
stated in Theorem \ref{thm:MYZ}; cf.\ \cite[Section 3, Problem (1)]{MYZ}.

\begin{coro} \label{cor:211}
Let $n \ge 2$, $\beta>(1-n)/2$ and $1 \le p < \infty$.
\begin{itemize}
\item[(a)]
The maximal operator $M^{\b}_*$
satisfies the weak type $(p,p)$ estimate on the subspace of radial functions on $\mathbb{R}^n$ 
if and only if
$$
\begin{cases}
\b \ge 1 & \textrm{when} \;\; p=1, \\
\b > 1 -n + \frac{n}p & \textrm{when} \;\; 1 < p < 2, \\
\b > \frac{1}p + \frac{1-n}2 & \textrm{when} \;\; 2 \le p \le \frac{2n}{n-1}, \\
\b \ge \frac{1-n}p & \textrm{when} \;\; p > \frac{2n}{n-1}.
\end{cases}
$$
\item[(b)]
The maximal operator $M^{\b}_*$
satisfies the restricted weak type $(p,p)$ estimate on the subspace of radial functions on $\mathbb{R}^n$ 
if and only if
$$
\begin{cases}
\b \ge 1 -n + \frac{n}p & \textrm{when} \;\; p < 2, \\
\b \ge \frac{1}p + \frac{1-n}2 & \textrm{when} \;\; 2 \le p \le \frac{2n}{n-1}, \\
\b \ge \frac{1-n}p & \textrm{when} \;\; p > \frac{2n}{n-1}.
\end{cases}
$$
\end{itemize}
\end{coro}
Comparing to \cite[Corollary 1.8]{CNR2} (after the relevant correction of the statement), Corollary~\ref{cor:18sh}
extends $L^p_{\textrm{rad}}(\mathbb{R}^n)$-boundedness of $M^{\b}_*$ to the case when $\b=(1-n)/p$ and $p > 2n/(n-1)$.
Concerning Corollary~\ref{cor:211}(a), its novelty, comparing to previously known results
and in view of Corollary~\ref{cor:18sh},
is the positive part in case $p=1$, and the negative part in case $\b\neq 0$
(the latter in case $\b=0$ follows from a counterexample of Stein \cite{Stein}).
In Corollary \ref{cor:211}(b), in view of item (a) and previously known results, the novel part is
the inclusion of the cases $\b = 1/p+(1-n)/2$, $2 \le p \le 2n/(n-1)$ and $\b=1-n+n/p$, $1<p<2$ for the positive part,
both except for the case when $\b=0$ where the result is due to Bourgain \cite{Bourgain} ($n\ge 3$) and Leckband \cite{Le} ($n=2$).
The negative part of Corollary \ref{cor:211}(b) is new.

The following statement corrects and strengthens \cite[Conjecture 1.9]{CNR2}.
The strengthening relies on extending the range of parameters for which $M^{\b}_*$ is conjectured to be bounded on
$L^p(\mathbb{R}^n)$ for $\b = (1-n)/p$ when $p > 2n/(n-1)$
(this corresponds to the segment $OC$ without endpoints on Figure \ref{fig:sh}).
\begin{conj} \label{conj:19sh}
Let $n \ge 2$ and $1 \le p < \infty$. The operator $M^{\b}_*$ is bounded on $L^p(\mathbb{R}^n)$ if and only if
$p > 1$ and Condition \eqref{id:151} is satisfied, i.e.
\begin{equation} \label{eq:nuestro}
\left\{
\begin{array}{c}
\b > \max\Big( 1-n + \frac{n}p, \frac{1}p + \frac{1-n}2, \frac{1-n}p\Big) \\
\textrm{\footnotesize{with the inequality weakened when $p > \frac{2n}{n-1}$}}
\end{array}
\right\}.
\end{equation}
\end{conj}

It is known that in general $M^{\b}_*$ has better weak/restricted weak type mapping properties when restricted to the subspace
of radial functions. For instance, $M^{0}_*$ fails to be of restricted weak type $(2,2)$ for $n=2$ (see \cite{STW}),
but it satisfies the restricted weak type $(2,2)$ inequality after restriction to radial functions, cf.\ \cite{Le}.
Thus the results of the present paper do not allow to state a plausible weak/restricted weak type counterpart
of Conjecture \ref{conj:19sh}.

\begin{remark}
\label{rem:Sogge}
There is an interesting link between 
our result in Corollary \ref{cor:18sh} and the celebrated Sogge's local smoothing conjecture for the wave equation.
The proof of Theorem \ref{thm:MYZ} (i.e.\ \cite[Theorem 1.1]{MYZ}) relies on recent results concerning this conjecture.
More precisely, if $u(x,t)$ is the solution to the Cauchy problem in $\mathbb{R}^n\times \mathbb{R}_+$
\begin{equation}
\label{eq:wave}
\Delta_xu-u_{tt}=0, \qquad u(x,0)=f(x), \qquad u_t(x,0)=g(x),
\end{equation}
it was conjectured in \cite{So} that for $n\ge 2$ and $p\ge \frac{2n}{n-1}$, one has
\begin{equation}
\label{eq:sogge}
\|u\|_{L^p(\mathbb{R}^n\times [1,2])}\le C_{\gamma} \Big(\|f\|_{W^{\gamma,p}(\mathbb{R}^n)}+\|g\|_{W^{\gamma-1,p}(\mathbb{R}^n)}\Big),
	\qquad \gamma>\frac{n-1}{2}-\frac{n}{p},
\end{equation}
where $W^{\gamma,p}$ is the standard inhomogeneous Sobolev space.
The best result towards the conjecture in dimensions $n\ge 3$
was obtained by Bourgain and Demeter \cite{BD} who proved that \eqref{eq:sogge}
is true for $p\ge \frac{2n+2}{n-1}$, $n\ge 2$.
Bourgain--Demeter's local smoothing estimate and an interpolation
argument yield (cf.\ \cite[p.\,4271]{MYZ}$^\ddag$
\footnote{$\ddag$
There is a misprint in \cite[p.\,4271]{MYZ}, the space $W^{\gamma-1,p}$ in \cite[(1.12)]{MYZ} should be replaced by $W^{\gamma,p}$.
}
\!\!\!), for $p > 2$,
\begin{equation}
\label{eq:BD}
\|u\|_{L^p(\mathbb{R}^n\times [1,2])}\le C_{\gamma} \|f\|_{W^{\gamma,p}(\mathbb{R}^n)},
\qquad \gamma>\max\bigg\{\frac{n-1}{2}\Big(\frac12-\frac1p\Big),\frac{n-1}{2}-\frac{n}{p}\bigg\},
\end{equation}
assuming that $g=0$. For $n=2$ Sogge's conjecture has been completely solved only recently by Guth et al.\ \cite{GWZ}.
The work by Miao et al.\ \cite{MYZ} was published before \cite{GWZ}.

It turns out that the solution to the wave equation \eqref{eq:wave} can be expressed in terms of the so-called half-wave propagator,
which in turn is closely related to the classical spherical maximal operator $M_*^{0}$;
these operators fall into the framework of Fourier integral operators, see for instance
\cite[Section 1.1, Examples 1 and 3]{BHS}. As observed in \cite{MYZ}, the Fourier multiplier \eqref{eq:Fmultiplier}
of the generalized spherical maximal operator can be written essentially as a half-wave propagator
with a slightly different symbol with worse decay. Thus, by virtue of the just sketched relations,
the estimate \eqref{eq:BD} is used in the proof of \cite[Theorem 1.1]{MYZ} to conclude the boundedness of $M_*^{\b}$
in the range of $\beta$ as given \eqref{cndthm12}.

According to this connection between the generalized spherical maximal operator and the local smoothing for the wave equation,
a better result in terms of local smoothing will yield a better boundedness result for $M_*^{\b}$.
Indeed, under the assumption of the conjectured sharp local smoothing estimate for the wave equation \eqref{eq:sogge},
by tracing the proof of \cite[Theorem 1.1]{MYZ} (see in particular \cite[pp.\,4277--4278]{MYZ}),
it is seen that the admissible range of $\b$ could be enlarged up to our range \eqref{eq:nuestro}
with strict inequality in the lower bound for $\b$.
Since under the (immaterial) assumption $\b > (1-n)/2$ Corollary \ref{cor:18sh} gives a complete, sharp description of the boundedness of
$M^{\beta}_*$ on $L^p_{\textrm{rad}}(\mathbb{R}^n)$ (one should compare with a corrected \cite[Corollary 1.8]{CNR2},
where the description was not as complete as in the current paper), the range of $\beta$ and $p$
in Conjecture \ref{conj:19sh} is essentially the best possible.
In particular, in virtue of the recent results in \cite{GWZ},
for $n=2$ the maximal operator $M^{\b}_*$ is bounded on $L^p(\mathbb{R}^2)$ if $p > 1$ and Condition \eqref{eq:nuestro}
is satisfied with strict inequality in the lower bound for $\b$.

On the other hand, the local smoothing conjecture for radial $f$ and $g \equiv 0$ was proved by M\"uller and Seeger \cite{MS}.
Hence, the reasoning described above leads to the sufficiency of conditions in Corollary \ref{cor:18sh},
except for $\b=(1-n)/p$ when $p>2n/(n-1)$, for the boundedness of $M_*^{\b}$.
In this regard, it is worth pointing out that in the present paper the question posed in \cite[Section 3, Problem (5)]{MYZ} is answered
in the radial case. Namely, it is possible to prove
the (sharp) result in Corollary \ref{cor:18sh} without appealing to the local smoothing estimate proved in \cite{MS}. 
\end{remark}

\medskip
\noindent{\textbf{Structure of the paper}.}
In Section \ref{sec:specaux} we gather definitions and mapping properties of several auxiliary operators
that are needed in the proofs of Theorem \ref{thm:strong} and Theorem \ref{thm:weak}.
Section \ref{sec:Mproof} and Section \ref{sec:Wproof} are devoted to the proofs of the sufficiency parts in Theorem
\ref{thm:strong} and Theorem \ref{thm:weak}, respectively. The corresponding necessity parts are proved in
Section \ref{sec:necpr}. Finally, Section \ref{sec:auxTS} contains proofs of the mapping properties stated
in Section \ref{sec:specaux} for the auxiliary maximal operators.
In Appendix we comment on the limiting case $\a+\b = -1/2$.

\medskip
\noindent{\textbf{Notation}.}
Throughout the paper we use a fairly standard notation.
The minimum and the maximum of two quantities is indicated by $\wedge$ and $\vee$, respectively.
We denote $\mathbb{R}_+=(0,\infty)$ and for the sake of brevity, we often omit $\mathbb{R}_+$ when dealing
with $L^p$ related to the measure space $(\mathbb{R}_+,x^{\delta}dx)$,
i.e., $L^p(x^{\delta}dx)=L^p(\mathbb{R}_+,x^{\delta}dx)$.
We write $L^p_{\textrm{rad}}(\ldots)$ for the subspace of $L^p(\ldots)$ consisting of radial functions.
As usual, for $1\le p\le \infty$, $p'$ denotes its conjugate exponent, $1/p+1/p'=1$.
By weakening a strict inequality ``$<$'' we mean replacing it by ``$\le$''.
Analogously, by strictening a weak inequality ``$\le$'' we mean replacing it by ``$<$''.
We write $X\lesssim Y$ to indicate that $X\le CY$ with a positive constant $C$ independent of significant quantities.
We shall write $X\simeq Y$ when simultaneously $X\lesssim Y$ and $Y\lesssim X$.

\medskip
\noindent{\textbf{Basic terminology}.}
Let $1 \le p < \infty$ and $\delta \in \mathbb{R}$.
An operator $T$ is said to be of strong type $(p,p)$ with respect to the measure space $(\mathbb{R}_+,x^{\delta}dx)$
when it is bounded on $L^p(\mathbb{R}_+,x^{\delta}dx)$.
Further, $T$ is said to be of weak type $(p,p)$ with respect to the measure space $(\mathbb{R}_+,x^{\delta}dx)$
if it satisfies the weak type $(p,p)$ estimate
$$
\lambda^p \int_{\{y > 0 : |Tf(y)|> \lambda\}} x^{\delta}\, dx \lesssim \int_0^{\infty} |f(x)|^p x^{\delta}\, dx,
	\qquad \lambda > 0,
$$
uniformly in $f \in L^p(\mathbb{R}_+,x^{\delta}dx)$. This is equivalent to boundedness
of $T$ from $L^p(\mathbb{R}_+,x^{\delta}dx)$ to the Lorentz space $L^{p,\infty}(\mathbb{R}_+,x^{\delta}dx)$.
The latter space is referred to as weak $L^p(\mathbb{R}_+,x^{\delta}dx)$.
Finally, $T$ is of restricted weak type $(p,p)$ with respect to the measure space $(\mathbb{R}_+,x^{\delta}dx)$
if it is bounded between the extreme Lorentz spaces,
from $L^{p,1}(\mathbb{R}_+,x^{\delta}dx)$ to $L^{p,\infty}(\mathbb{R}_+,x^{\delta}dx)$. (Recall that, on the second
index scale, the space $L^{p,1}$ is the smallest one, and $L^{p,\infty}$ is the biggest one among $L^{p,q}$, $1\le q \le \infty$.)
When $p>1$ this property is equivalent to the weak type $(p,p)$ estimate restricted to functions $f$ that are characteristic
functions of sets of finite $x^{\delta}dx$ measure. Note that for $p=1$ the notions of weak type and restricted weak type coincide.
 
\medskip
\noindent{\textbf{Acknowledgments}.}
The authors are grateful to David Beltran for useful discussions related to Remark~\ref{rem:Sogge}.

\smallskip
\noindent
{\footnotesize{
\textbf{A.\ Nowak} was partially supported by the National Science Centre of Poland within the research
project OPUS 2017/27/B/ST1/01623.
\textbf{L.\ Roncal} was supported by the Ministry of
Science and Innovation: BCAM Severo Ochoa accreditation
CEX2021-001142-S/MICIN/AEI/ 10.13039/501100011033, PID2020-113156GB-100 funded by MCIN/AEI/ 10.13039/501100011033
and RYC2018-025477-I funded by FSE ``invest in your future'',
and by the Basque Government through the BERC 2022-2025 program. She also acknowledges IKERBASQUE fundings.
\textbf{T.Z.\ Szarek} was partially supported by the National Science Centre of Poland within the grant Opus 2017/27/B/ST1/01623,
by the Ministry of
Science and Innovation: BCAM Severo Ochoa accreditation
CEX2021-001142-S/MICIN/AEI/10.13039/501100011033, Juan de la Cierva Incorporaci\'on 2019 IJC2019-039661-I,
and the project PID2020-113156GB-100 funded by MCIN/AEI/10.13039/ 501100011033,
and by the Basque Government through BERC 2022--2025.
}}

\section{Auxiliary operators} \label{sec:specaux}

In this section we gather mapping properties of several auxiliary operators that are needed in the proofs of
our main results. More precisely, the operators in question are the Hardy operators $H_{\eta}$ and $H_{\eta}^{\infty}$,
the (centered) local Hardy-Littlewood maximal operator $L$,
and special maximal operators $N_{\eta}$, $R_{\eta}$, $E_{k,\eta}$, $T_{\eta}$, $S_{\a,\b}$, $S_{\a,\b}^{\log}$
and $R_{\a,\b}^{\log}$.

The Hardy operators and the local Hardy-Littlewood maximal operator
are well studied and the results we need can simply be invoked from the literature.
The results for the remaining operators are either partially new or entirely new and thus require proofs.
Those proofs are technical and rather tedious. We postpone them to Section~\ref{sec:auxTS}
in order not to interfere the main line of the paper.

We remark that $R_{\eta}$ and $E_{k,\eta}$ for $\eta = 0$ appear in \cite{DMO}, and
$N_{\eta}$ for $\eta > 0$ and $T_{\eta}$ were introduced in \cite{DMO2}. See also \cite{CNR2}.
Some mapping properties of these operators were established in the papers just mentioned,
cf.\ \cite[Section 2]{DMO}, \cite[Section 2]{DMO2} and \cite[Section 2.3]{CNR2}.
Here we provide \emph{complete characterizations} of strong, weak and restricted weak type boundedness of
$N_{\eta}$, $R_{\eta}$, $E_{k,\eta}$ and $T_{\eta}$ with respect to the power weighted measure space $(\mathbb{R}_+, x^{\gamma}dx)$.

We also note that the results we prove for the auxiliary operators are sometimes stronger than we actually need for
our purpose. Nevertheless, we decided to state them as characterizations for the sake of completeness,
clarity, deeper understanding of the control they give and, finally, possible future applications.

\medskip
\noindent \underline{\textbf{Hardy operators $H_{\eta}$ and $H_{\eta}^{\infty}$.}} \,
For $\eta \in \mathbb{R}$, let
$$
H_{\eta}f(x) = \frac{1}{x^{\eta}} \int_0^{x} z^{\eta-1} f(z)\, dz, \qquad
H_{\eta}^{\infty}f(x) = x^{\eta} \int_x^{\infty} z^{-\eta-1} f(z) \, dz, \qquad x > 0.
$$
The following characterizations of the mapping properties of $H_{\eta}$ and $H_{\eta}^{\infty}$ are essentially contained
in \cite{AM}; see e.g.\ \cite[Lemmas 2.1 and 2.2]{LN} and references given there.

\begin{lema} \label{lem:H0}
Let $\eta,\gamma \in \mathbb{R}$ and $1 \le p < \infty$.
Consider $H_{\eta}$ on the measure space $(\mathbb{R}_+,x^{\gamma}dx)$.
\begin{itemize}
\item[(a)] $H_{\eta}$ is of strong type $(p,p)$ if and only if $\gamma < \eta p-1$.
\item[(b)] $H_{\eta}$ is of weak type $(p,p)$ if and only if $\gamma < \eta p-1$, with the inequality weakened in case
	$p=1$ and $\eta \neq 0$.
\item[(c)] $H_{\eta}$ is of restricted weak type $(p,p)$ if and only if 
	$\gamma \le \eta p-1$, with the inequality strictened in case $\eta = 0$.
\end{itemize}
\end{lema}

\begin{lema} \label{lem:Hinf}
Let $\eta,\gamma \in \mathbb{R}$ and $1 \le p < \infty$.
Consider $H_{\eta}^{\infty}$ on the measure space $(\mathbb{R}_+,x^{\gamma}dx)$.
\begin{itemize}
\item[(a)] $H_{\eta}^{\infty}$ is of strong type $(p,p)$ if and only if $-\eta p -1 < \gamma$.
\item[(b)] $H_{\eta}^{\infty}$ is of weak type $(p,p)$ if and only if $-\eta p -1 < \gamma$, with the inequality
	weakened in case $p=1$ and $\eta \neq 0$.
\item[(c)] $H_{\eta}^{\infty}$ is of restricted weak type $(p,p)$ if and only if $-\eta p -1 \le \gamma$, with
	the inequality strictened in case $\eta = 0$.
\end{itemize}
\end{lema}

\medskip
\noindent {\textbf{\underline{Local Hardy-Littlewood maximal operator $L$.}}}\,
Let
$$
Lf(x) = \sup_{t < x/2} \frac{1}{2t} \int_{x-t}^{x+t} |f(z)|\, dz, \qquad x > 0.
$$
The result below is well known, it is essentially due to Muckenhoupt \cite[Section 9]{M}.
See also \cite[Section 6]{NS} and the discussion on power weights succeeding \cite[(2.9)]{NS}.
\begin{lema} \label{lem:L}
Let $1 \le p < \infty$ and $\gamma \in \mathbb{R}$.
Consider $L$ on the measure space $(\mathbb{R}_+,x^{\gamma}dx)$.
\begin{itemize}
\item[(a)]
$L$ is of strong type $(p,p)$ if and only if $p > 1$.
\item[(b)]
$L$ is always of weak type $(p,p)$.
\end{itemize}
\end{lema}

\medskip
\noindent \textbf{\underline{Maximal operator $N_{\eta}$.}}\,
For $\eta \in \mathbb{R}$, define
$$
N_{\eta}f(x) = \sup_{t > x} \frac{1}{t^{\eta}} \int_0^t z^{\eta-1}|f(z)| \, dz, \qquad x > 0.
$$
\begin{lem} \label{lem:N2}
Let $\eta,\gamma \in \mathbb{R}$ and $1 \le p < \infty$.
Consider $N_{\eta}$ on the measure space $(\mathbb{R}_+,x^{\gamma}dx)$.
\begin{itemize}
\item[(a)] 
$N_{\eta}$ is of strong type $(p,p)$ if and only if 
$- 1 < \gamma < \eta p -1$.
\item[(b)] 
$N_{\eta}$ is of weak type $(p,p)$ if and only if 
$- 1 < \gamma < \eta p -1$, 
with the second inequality weakened if $p = 1$.
\item[(c)] 
$N_{\eta}$ is of restricted weak type $(p,p)$ if and only if 
$- 1 < \gamma \le \eta p -1$.
\end{itemize}
\end{lem}

\medskip
\noindent \textbf{\underline{Maximal operator $R_{\eta}$.}}\,
For $\eta \in \mathbb{R}$, define
$$
R_{\eta} f (x)
 = 
x^{\eta} \sup_{t > 2x} \frac{1}{2x} \int_{t-x}^{t+x} z^{-\eta} \abs{f(z)} \, dz, \qquad x > 0.
$$

\begin{lem} \label{lem:R2}
Let $\eta,\gamma \in \mathbb{R}$ and $1 \le p < \infty$.
Consider $R_{\eta}$ on the measure space $(\mathbb{R}_+,x^{\gamma}dx)$.
\begin{itemize}
\item[(a)] 
$R_{\eta}$ is of strong type $(p,p)$ if and only if 
$\gamma \ge -\eta p$, 
with the inequality strictened if $p = 1$.
\item[(b)] 
$R_{\eta}$ is of weak type $(p,p)$ if and only if 
$\gamma \ge -\eta p$, 
with the inequality strictened if $p = \eta = 1$.
\item[(c)] 
$R_{\eta}$ is of restricted weak type $(p,p)$ if and only if 
$\gamma \ge -\eta p$, 
with the inequality strictened if $p = \eta = 1$.
\end{itemize}
\end{lem}
Notice that the conditions for $R_\eta$ to be of weak type $(p,p)$ and of restricted weak type $(p,p)$ coincide.

\medskip
\noindent \textbf{\underline{Maximal operator $E_{k,\eta}$.}}\,
For $\eta \in \mathbb{R}$ and $k>0$, define
$$
E_{k, \eta} f (x)
=
x^{\eta} \sup_{0\le a <x< b} \frac{1}{b^k-a^k} \int_{a}^{b} z^{-\eta + k -1} \abs{f(z)} \, dz, \qquad x > 0.
$$

\begin{lem} \label{lem:E2}
Let $\eta,\gamma \in \mathbb{R}$ and $k>0$.
Consider $E_{k, \eta}$ on the measure space $(\mathbb{R}_+,x^{\gamma}dx)$.
\begin{itemize}
\item[(a)] 
$E_{k, \eta}$ is of strong type $(p,p)$ if and only if 
$p > 1$ and 
$- \eta p  - 1 < \gamma < kp - \eta p - 1$.
\item[(b)] 
$E_{k, \eta}$ is of weak type $(p,p)$ if and only if
$- \eta p  - 1 < \gamma < kp - \eta p - 1$,
with the first inequality weakened if $\eta \ne 0$ and with the second inequality weakened if $p=1$ and $\eta \ne k$.
\item[(c)] 
$E_{k, \eta}$ is of restricted weak type $(p,p)$ if and only if 
$- \eta p  - 1 < \gamma < kp - \eta p - 1$,
with the first inequality weakened if $\eta \ne 0$ and with the second inequality weakened if $\eta \ne k$.
\end{itemize}
\end{lem}

\medskip
\noindent \textbf{\underline{Maximal operator $T_{\eta}$.}}\,
For $\eta \in \mathbb{R}$, let
$$
T_{\eta}f(x) = \sup_{t > 2x} \int_{t/2}^t \frac{z^{\eta-1} |f(z)|}{(t-z+x)^{\eta}} \, dz, \qquad x > 0.
$$

\begin{lem} \label{lem:Tchar}
Let $\eta,\gamma \in \mathbb{R}$ and $1 \le p < \infty$.
Consider $T_{\eta}$ on the measure space $(\mathbb{R}_+,x^{\gamma}dx)$.
\begin{itemize}
\item[(a)] 
$T_{\eta}$ is of strong type $(p,p)$ if and only if 
$\gamma > (-1) \vee  [\,p(\eta -1)]$, 
with the inequality weakened if $p > 1$ and $\eta > 1/p'$.
\item[(b)] 
$T_{\eta}$ is of weak type $(p,p)$ if and only if 
$\gamma > (-1) \vee  [\,p(\eta -1)]$, 
with the inequality weakened if $\eta > 1/p'$.
\item[(c)] 
$T_{\eta}$ is of restricted weak type $(p,p)$ if and only if 
$\gamma > (-1) \vee  [\,p(\eta -1)]$, 
with the inequality weakened if $\eta > 1/p'$.
\end{itemize}
\end{lem}

Observe that the maximum occurring in Lemma \ref{lem:Tchar} is equal $-1$ when $\eta \le 0$ and $(\eta-1)p$ when $\eta \ge 1$.
When $\eta \in (0,1)$, any of the two expressions in the maximum can win, depending on $\eta$ and $p$.
Notice also that the conditions for $T_\eta$ to be of weak type $(p,p)$ and of restricted weak type $(p,p)$ coincide.

\medskip
\noindent \textbf{\underline{Maximal operator $S_{\a,\b}$.}}\,
For $\a,\b \in \mathbb{R}$, define
\begin{align*}
S_{\alpha, \beta} f (x)
=
\sup_{t>3x} t^{-\beta} \int_{\frac{t-x}2}^{t-x} (t + x - z)^{-\alpha - 1/2}
(t - x - z)^{\alpha +\beta - 1/2} \abs{f(z)} \, dz, \qquad x > 0.
\end{align*}

\begin{lem} \label{lem:5}
Assume that $\b \le 0$, $-1/2 < \a + \beta < 1/2$ and $1 < p < \infty$.
Then, $S_{\a,\b}$ is bounded on $L^p(\mathbb{R}_+, x^{-\beta p} dx)$ if and only if $1/p < \alpha + \beta + 1/2$. 
\end{lem}

\begin{lem} \label{lem:5a}
Assume that $0 < \b < 1$, $-1/2 < \a + \b < 1/2$, $\a > -1/2$ and $1 < p < \infty$.
Then, $S_{\a,\b}$ is bounded on $L^p(\mathbb{R}_+,x^{-\b p}dx)$ provided that
$\b < 1/p < \a + \b + 1/2$.
\end{lem} 

\medskip
\noindent \textbf{\underline{Maximal operator $S_{\a,\b}^{\log}$.}}\,
For $\a, \b \in \mathbb{R}$, $\a+\b=1/2$,
define a logarithmic variant of $S_{\a,\b}$
$$
S_{\alpha, \beta}^{\log} f (x)
=
\sup_{t>3x} t^{-\beta} \int_{(t-x)/2}^{t-x} (t + x - z)^{-\alpha - 1/2}
\log \bigg(2 + \frac{x}{t - x - z} \bigg)  \abs{f(z)} \, dz, \qquad x > 0.
$$

\begin{lem} \label{lem:5b}
Assume that $\alpha > -1$, $\a + \b = 1/2$ and $1 < p < \infty$.
Then, $S_{\alpha, \beta}^{\log}$ is bounded on $L^p(\mathbb{R}_+, x^{-\beta p} dx)$ provided that $\b < 1/p$.
\end{lem}

\medskip
\noindent \textbf{\underline{Maximal operator $R_{\a,\b}^{\log}$.}}\,
For $\a, \b \in \mathbb{R}$,
define
$$
R_{\a,\b}^{\log}f(x) = x^{\b} \sup_{t > 3x} \frac{1}{x} \int_{t-x}^{t+x} \log\bigg( \frac{4x}{z-(t-x)}\bigg) z^{-\b}|f(z)|\, dz,
\qquad x > 0.
$$
\begin{lem} \label{lem:Rlog}
Assume that $\a > -1$, $\a+\b = 1/2$ and $1 < p < \infty$. Then, $R_{\a,\b}^{\log}$ is bounded on $L^p(\mathbb{R}_+,x^{-\b p}dx)$.
\end{lem}

\section{Proof of Theorem \ref{thm:strong}, sufficiency part} \label{sec:Mproof}

In this section we prove the sufficiency part of Theorem \ref{thm:strong}.
To begin with, we rephrase it  in a more convenient form for the proof.
Assuming that $\a > -1$, $\a + \b > -1/2$, $1 \le p < \infty$ and $\delta \in \mathbb{R}$, Theorem \ref{thm:strong}
says that $M^{\a,\b}_*$ is bounded on $L^p(\mathbb{R}_+,x^{\delta}dx)$ if and only if $p > 1$ and
\begin{equation} \label{pcnd0}
\frac{1}{p} < \a + \b + \frac{1}2
\end{equation}
and
\begin{equation} \label{pcnd}
-1 < \delta, \qquad - \b p \le \delta, \qquad \delta < (2\a+2)p-1, \qquad \delta < (2\a+\b+1)p-1.
\end{equation}

Note that Condition \eqref{pcnd0} is meaningful only when $\a+\b < 1/2$.
In view of \cite[Theorem~1.5]{CNR2}, sufficiency of Conditions \eqref{pcnd0} and \eqref{pcnd}
is known for most choices of the parameters $\a$ and $\b$, see the comments succeeding \cite[Proposition 1.6]{CNR2}.
More precisely, the values of $\a$ and $\b$ that remain to be treated belong to the regions
\begin{align*}
\mathfrak{R}_1 & = \big\{ (\a,\b) : -1/2 < \a + \b \le 1/2,\; \b < 0,\; -\b \notin \mathbb{N} \big\},  \\
\mathfrak{R}_2 & = \big\{ (\a,\b) : \a > -1/2,\; 0 < \b < 1 \}.
\end{align*}

Taking into account \cite[Theorem 1.5]{CNR2}, it is enough to
show the following statements under the general assumptions $\a > -1$ and $\a + \b > -1/2$,
for $1 < p < \infty$, and under Condition \eqref{pcnd0}.
\begin{itemize}
\item[(i)]
If $(\a,\b) \in \mathfrak{R}_1$, then $M^{\a,\b}_*$ is bounded on $L^p(\mathbb{R}_+,x^{-\b p}dx)$.
\item[(ii)]
If $(\a,\b) \in \mathfrak{R}_2$, then $M^{\a,\b}_*$ is bounded on $L^p(\mathbb{R}_+,x^{\delta}dx)$ provided that
$$
- 1 < \delta, \qquad -\b p \le \delta, \qquad \delta < (2\a+\b+1)p-1.
$$
\end{itemize}
We may also assume throughout that $f \ge 0$, since the arguments we give are based on absolute estimates of the kernel.

In what follows we use some notation from \cite{CNR2} without much comment.
For the readers' convenience, Figure \ref{fig_reg} visualizes the relevant regions in $\mathbb{R}^3_+$ determining
the splitting of the kernel $K_t^{\a,\b}(x,z)$ introduced in \cite[Section 2.2]{CNR2}.
Recall that
\begin{align*}
E & = \big\{(t,x,z) \in \RR_+^3 : |x-z| < t < x+z\big\}, \\
F & = \big\{(t,x,z) \in \RR_+^3 : x+z < t \big\},
\end{align*}
and
$$
E = E_1 \cup E_2 \cup E_3, \qquad F = F_1 \cup F_2, \qquad F_2 = F'_2 \cup F''_2,
$$
all the sums being disjoint. For precise definitions of the regions appearing above, see \cite[Section 2.2]{CNR2}.
Note that up to boundaries those regions are specified by Figure \ref{fig_reg}, and the
boundaries do not really matter in our developments.

\begin{figure}
\centering
\begin{tikzpicture}[scale=1.6]
\draw[arrows=-angle 60] (0,0) -- (0,5.5);
\draw[arrows=-angle 60] (0,0) -- (5.5,0);
\node at (-0.15,-0.18) {$0$};
\node at (0.12,5.4) {$t$};
\node at (5.4,0.13) {$z$};
\draw[very thin] (0.5,-0.05) -- (0.5,0.05);
\node at (0.5,-0.2) {$x/2$};
\draw[very thin] (1,-0.05) -- (1,0.05);
\node at (1,-0.2) {$x$};
\draw[very thin] (1.5,-0.05) -- (1.5,0.05);
\node at (1.5,-0.2) {$3x/2$};
\draw[very thin] (2,-0.05) -- (2,0.05);
\node at (2,-0.2) {$2x$};
\draw[very thin] (3,-0.05) -- (3,0.05);
\node at (3,-0.2) {$3x$};
\draw[very thin] (4,-0.05) -- (4,0.05);
\node at (4,-0.2) {$4x$};
\draw[very thin] (5,-0.05) -- (5,0.05);
\node at (5,-0.2) {$5x$};
\draw[very thin] (-0.05,0.5) -- (0.05,0.5);
\node at (-0.2,0.5) {$\frac{x}2$};
\draw[very thin] (-0.05,1) -- (0.05,1);
\node at (-0.2,1) {$x$};
\draw[very thin] (-0.05,2) -- (0.05,2);
\node at (-0.2,2) {$2x$};
\draw[very thin] (-0.05,3) -- (0.05,3);
\node at (-0.2,3) {$3x$};
\draw[very thin] (-0.05,4) -- (0.05,4);
\node at (-0.2,4) {$4x$};
\draw[very thin] (-0.05,5) -- (0.05,5);
\node at (-0.2,5) {$5x$};
\fill[black!5!white] (0,1) -- (1,0) -- (5,4) -- (5,5) -- (4,5) -- cycle;
\fill[black!20!white] (0,1) -- (4,5) -- (0,5) -- cycle;
\draw[thick] (0,1) -- (4,5);
\draw[thick] (1,0) -- (5,4);
\draw[thick] (0,1) -- (1,0);
\draw[thick] (0,3) -- (4,3);
\draw[thick] (1,3) -- (2,5);
\draw[thick] (0.5,0.5) -- (1.5,0.5);
\node at (0.7,4) {$\mathbf{F_2'}$};
\node at (2.2,4) {$\mathbf{F_2''}$};
\node at (4.1,4) {$\mathbf{E_3}$};  
\node at (0.7,2.4) {$\mathbf{F_1}$}; 
\node at (1.7,1.6) {$\mathbf{E_2}$}; 
\node at (1.03,0.33) {$\mathbf{E_1}$};
\draw[very thin, dashed] (0,0.5) -- (0.5,0.5) -- (0.5,0);
\draw[very thin, dashed] (1,0) -- (1,3);
\draw[very thin, dashed] (1.5,0) -- (1.5,0.5);
\draw[very thin, dashed] (2,0) -- (2,3);
\draw[very thin, dashed] (4,0) -- (4,3);
\node[rotate=-45] at (0.4,0.4) {$z=x-t$};
\node[rotate=45] at (4.5,3.3) {$z=t+x$};
\node[rotate=45] at (3.5,4.3) {$z=t-x$};
\node[rotate=63.43] at (1.4,4.3) {$z=\frac{t-x}{2}$}; 
\draw[arrows=-angle 60] (0,0) -- (0,5.5);
\end{tikzpicture} 

\caption{Sections of regions $E_i$ and $F_i$ given $x>0$ fixed.} \label{fig_reg}
\end{figure}
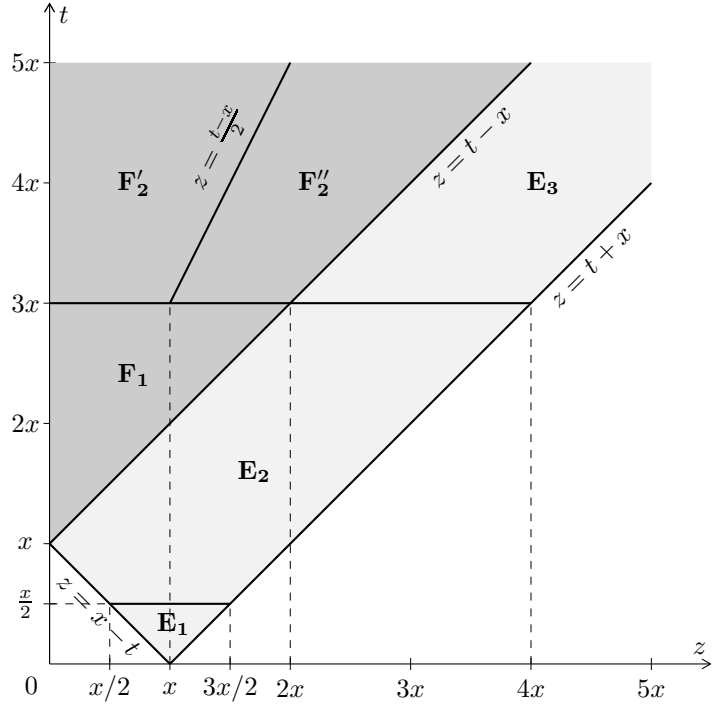

\subsection{Proof of (i)}
We let $(\a,\b) \in \mathfrak{R}_1$ and distinguish two cases.

\noindent \textbf{Case 1: $\a + \b < 1/2$.}
In view of the proof of \cite[Theorem 1.5]{CNR2}, see \cite[Case 2, pp.\,1613--1614]{CNR2}, it is enough to show
$L^p(\mathbb{R}_+,x^{-\b p}dx)$-boundedness of the part of $M^{\a,\b}_*$ related to the region $F_2''$
(see Figure \ref{fig_reg}).
This part is controlled by $\Psi_{*,F_2''}^{\a,\b}$ (cf.\ \cite[Theorem 2.2]{CNR2} and \cite[p.\,1609]{CNR2}),
$$
\Psi_{*,F_2''}^{\a,\b}f(x) := \sup_{t \ge 3x} \frac{1}{t^{2\a+2\b}} \int_{(t-x)/2}^{t-x}
	\big[ t^2 - (x-z)^2\big]^{-\a-1/2} \big[ t^2 - (x+z)^2\big]^{\a+\b-1/2} f(z) z^{2\a+1}\, dz,
$$
which is further estimated (see \cite[p.\,1612, Region $F_2''$]{CNR2})
$$
\Psi_{*,F_2''}^{\a,\b}f(x) \lesssim \sup_{t \ge 3x} \frac{1}{t^{2\a+\b+1}} \int_{\frac{t-x}2}^{t-x}
	(t+x-z)^{-\a-1/2} (t-x-z)^{\a+\b-1/2} f(z) z^{2\a+1}\, dz.
$$
Since $z \simeq t-x \simeq t$ on $F_2''$ (cf.\ \cite[(2.8)]{CNR2}), it follows that
$$
\Psi_{*,F_2''}^{\a,\b}f(x) \lesssim S_{\a,\b}f(x), \qquad x > 0,
$$
where $S_{\a,\b}$ is our auxiliary operator from Section \ref{sec:specaux}.
Thus the desired conclusion follows from Lemma \ref{lem:5}.

\noindent \textbf{Case 2: $\a + \b = 1/2$.}
In view of the proof of \cite[Theorem 1.5]{CNR2}, it is enough we prove $L^p(\mathbb{R}_+,x^{-\b p}dx)$-boundedness
of the parts of $M^{\a,\b}_{*}$ related to regions $E_3$ and $F_2''$. In \cite{CNR2} these parts were controlled
in terms of special operators $R$ (which is our $R_\eta$ from Section~\ref{sec:specaux} specified to $\eta = 0$) and $T_{\eta}$,
but we need a better control for our present purpose.
Recall, cf.\ \cite[Theorem 2.1(4)]{CNR2}, that for $\a$ and $\b$ under consideration
$$
|K_t^{\a,\b}(x,z)| \lesssim
		\begin{cases}
			 \frac{(xz)^{-\a-1/2}}{t}
				\log\frac{8xz}{(x+z)^2-t^2} & \textrm{in} \; E, \\
			\frac{1}{t} \big[t^2-(x-z)^2\big]^{-\a-1/2} \log\Big( 2 \frac{t^2-(x-z)^2}{t^2-(x+z)^2}\Big) & \textrm{in} \; F.
		\end{cases}
$$

The part of $M^{\a,\b}_*$ related to $E_3$ is controlled by the operator
\begin{align*}
f & \mapsto \sup_{t \ge 3x} \frac{1}t \int_{t-x}^{t+x} (xz)^{-\a-1/2} \log\bigg(
	\frac{8xz}{(x+z-t)(x+z+t)}\bigg) f(z) z^{2\a+1}\, dz \\
	& \qquad \simeq \sup_{t \ge 3x}  x^{-\a+1/2} \frac{1}x \int_{t-x}^{t+x} \log\bigg(\frac{4x}{x+z-t}\bigg) z^{\a-1/2}f(z)\, dz
	 = R_{\a,\b}^{\log}f(x),
\end{align*}
where we used the fact that $z \simeq t$ on $E_3$ (cf.\ \cite[(2.4)]{CNR2}).
By Lemma \ref{lem:Rlog} we infer the desired mapping property.

Passing to $F_2''$, this part of $M^{\a,\b}_*$ is controlled by
\begin{align*}
f & \mapsto \sup_{t \ge 3x} \frac{1}{t} \int_{\frac{t-x}2}^{t-x} \big[ t^2-(x-z)^2\big]^{-\a-1/2}
	\log\bigg( 2\frac{t^2-(x-z)^2}{t^2-(x+z)^2}\bigg) f(z) z^{2\a+1}\, dz \\
	& \qquad \simeq \sup_{t \ge 3x} t^{-\b}\int_{\frac{t-x}2}^{t-x} (t+x-z)^{-\a-1/2}
		\log\bigg( 2 + \frac{x}{t-x-z}\bigg) f(z)\, dz = S_{\a,\b}^{\log}f(x),
\end{align*}
where we used the fact that on $F_2''$ one has the relations (cf.\ \cite[(2.6),(2,8)]{CNR2})
$t^2-(x-z)^2 \simeq t(t+x-z)$, $t^2-(x+z)^2 \simeq t (t-x-z)$ and $z \simeq t$.
Now the $L^p(\mathbb{R}_+,x^{-\b p}dx)$-boundedness follows by Lemma \ref{lem:5b}.

The proof of (i) is completed.

\subsection{Proof of (ii)}
We let $(\a,\b) \in \mathfrak{R}_2$ and distinguish three cases.

\noindent \textbf{Case 1: $\a + \b > 1/2$.}
In view of the proof of \cite[Theorem 1.5, Case 1, p.\,1613]{CNR2}, it is enough we show that $T_{1-\b}$
is bounded on $L^p(\mathbb{R}_+,x^{\delta}dx)$ when both $\delta > -1$ and $\delta \ge -\b p$. This, however, is
contained in Lemma \ref{lem:Tchar}.

\noindent \textbf{Case 2: $\a + \b < 1/2$.}
Taking into account the proof of \cite[Theorem 1.5, Case 2, pp.\,1613--1614]{CNR2}, it suffices we focus on the part of
$M^{\a,\b}_*$ related to region $F_2''$. In \cite{CNR2} this part is controlled by the operator
\begin{equation} \label{Ttheta}
f \mapsto \big[ T_{(\a+1/2)(1+\theta)}\big(f^{1+\theta}\big)(x)\big]^{\frac{1}{1+\theta}}
\end{equation}
with any fixed $\theta > 0$ satisfying $\frac{1}{1+\theta} < \a+\b+1/2$.
Recall that in our present situation $-1/2 < \a < 1/2$ and $0 < \b < 1$.
We consider $\theta$ such that $\frac{1}p < \frac{1}{1+\theta}$ and $0 < \a + 1/2 < \frac{1}{1+\theta}< \a+\b+1/2$.
By Lemma \ref{lem:Tchar} the conditions for the $L^p(\mathbb{R}_+,x^{\delta}dx)$-boundedness of the operator in \eqref{Ttheta} are
$\delta \ge [(\a+1/2)(1+\theta)-1]\frac{p}{1+\theta}$ and $\delta > -1$. One can choose $\theta$ so that
$\frac{1}{1+\theta}$ is arbitrarily close to $\a + \b + 1/2$, which covers all $\delta > \max\{-\b p, -1\}$.

It remains to deal with $\delta = -\b p$ in case $\frac{1}p > \b$.
For this purpose a better control of the $F_2''$ part of $M^{\a,\b}_*$ is needed, and provided by,
see Case 1 in the proof of (i) above, $\Psi^{\a,\b}_{*,F_2''}f(x) \lesssim S_{\a,\b}f(x)$.
Now the conclusion follows from Lemma \ref{lem:5a}.

\noindent \textbf{Case 3: $\a + \b = 1/2$.}
In view of Case 2 above, repeating the argument from the proof of \cite[Theorem 1.5, Case 3, pp.\,1614--1615]{CNR2},
we get $L^p(\mathbb{R}_+,x^{\delta}dx)$-boundedness of $M^{\a,\b}_*$ for $\max\{-\b p,-1\} < \delta < (2\a+\b+1)p-1$.

It remains to consider $\delta = -\b p > -1$. Here the argument just invoked does not work for the parts of
$M^{\a,\b}_*$ related to regions $E_3$ and $F_2''$ (but it works for all the complementary parts).
Hence we must control those two parts in a better way, and this is done exactly as in the proof of (i), Case 2.
That is, the part associated with $E_3$ is controlled by $R^{\log}_{\a,\b}$, and that related
to $F_2''$ by $S_{\a,\b}^{\log}$. Therefore, the desired boundedness follows from Lemmas \ref{lem:Rlog} and \ref{lem:5b}, respectively.

The proof of (ii) is finished. This concludes the proof of the sufficiency part of Theorem~\ref{thm:strong}.

\section{Proof of Theorem \ref{thm:weak}, sufficiency parts} \label{sec:Wproof}

We split the proof of the sufficiency parts in Theorem \ref{thm:weak} into three parts contained in
the subsequent sections. The splitting is determined essentially by the location of $(\a,\b)$ with respect to the line
$\a + \b = 1/2$. In this way we consider the \emph{regular} case when $\a+\b > 1/2$ or $\a+\b =1/2$ and $\b$ is integer,
the \emph{logarithmic} case when $\a+\b=1/2$ and $\b$ is non-integer, and the \emph{singular} case when $\a + \b < 1/2$.
The terminology `regular', `logarithmic' and `singular' is suggested by the behavior of the kernel of $M_t^{\a,\b}$. 

We shall use the notation of \cite{CNR2}. In what follows we always assume that $f \ge 0$,
since our arguments are based on absolute estimates of the kernel $K_t^{\a,\b}(x,z)$.

\subsection{The regular case: $\boldsymbol{\a + \b > 1/2}$
	or $\boldsymbol{[\a+\b = 1/2 \;\textrm{and}\; \b \in \mathbb{Z}]}$} \, \medskip

It is convenient to consider several subcases.

\noindent \textbf{Subcase 1: $\a + \b > 1/2$ and $-\b \notin \mathbb{N}$.} 
As pointed out in \cite[Proof of Theorem 1.5, Case~1]{CNR2}, one has the control
\begin{align*}
M^{\a,\b}_{*}f(x) & \lesssim Lf(x) + H_{2\a+\b+1}f(4x) +  R_{\b}f(x)  \\
	& \qquad + H_{2\a+2}f(2x) + N_{2\a+2}f(x) + T_{1-\b}f(x). 
\end{align*}
From this the desired conclusion follows by Lemmas \ref{lem:H0}, \ref{lem:L}, \ref{lem:N2}, \ref{lem:R2} and \ref{lem:Tchar}.

\noindent \textbf{Subcase 2: $\a + \b \ge 1/2$ and $-\b \in \mathbb{N}$.}
In view of \cite[Proof of Theorem 1.5, Case 4]{CNR2}, one has the control
$$
M^{\a,\b}_{*}f(x) \lesssim Lf(x) + H_{2\a+\b+1}f(4x) + R_{\b}f(x).
$$
Now the conclusion follows by Lemmas \ref{lem:H0}, \ref{lem:L} and \ref{lem:R2}.

\noindent \textbf{Subcase 3: $(\a,\b) = (-1/2,1)$.}
In this case we have, see \cite[Theorem 2.1]{CNR2},
$$
K_t^{\a,\b}(x,z) \simeq \frac{1}t \qquad \textrm{in} \quad E \cup F.
$$
Therefore,
$$
M^{-1/2,1}_{*}f(x) \simeq \sup_{t>0} \frac{1}t \int_{\{z > 0 : |x-z| \le t\}} f(z)\, dz \lesssim E_{1,0}f(x).
$$
Using now Lemma \ref{lem:E2} we get the conclusion for $(\a,\b) = (-1/2,1)$.

It is perhaps interesting to observe that one actually has the identity $M^{-1/2,1}_{*}f(x) = \mathcal{M}f_{e}(x)$,
$x > 0$, where $f_e$ is an even extension of $f$ to $\mathbb{R}$ (the value at $0$ does not matter), and $\mathcal{M}$
is the classical one-dimensional centered Hardy-Littlewood maximal function.
This follows from the explicit formula for $K_t^{-1/2,1}(x,z)$, see \cite[p.\,4412]{CNR}.

\subsection{The logarithmic case: $\boldsymbol{\a + \b = 1/2}$ and $\boldsymbol{\b \notin \mathbb{Z}}$} \, \medskip

In this case we have, see \cite[Theorem 2.1]{CNR2},
\begin{equation} \label{logKbnd}
|K^{\a,\b}_t(x,z)| \lesssim
	\begin{cases}
		\frac{(xz)^{-\a-1/2}}t \log\frac{8xz}{(x+z)^2-t^2} & \textrm{in} \; E, \\
		\frac{1}t \big[ t^2-(x-z)^2\big]^{-\a-1/2}\log\Big( 2\frac{t^2-(x-z)^2}{t^2-(x+z)^2}\Big) & \textrm{in} \; F.
	\end{cases}
\end{equation}

To get the restricted weak type endpoint result that is needed,
we follow and extend Colzani et al.\ \cite{CoCoSt},
where for $\a \ge -1/2$ the special case of the natural weight $\delta= 2\a+1$ was done.
For the sake of clarity, we now make two statements which together with the already justified sufficiency part of
Theorem~\ref{thm:strong} give the sufficiency parts in Theorem \ref{thm:weak} for the considered $\a$ and $\b$. 
\begin{propo} \label{prop:5pre}
Assume that $-\b \notin \mathbb{N}$, $\a+\b=1/2$ and $\a > -1/2$ (and $\b < 1$).
Let $1 < p < \infty$ and $\delta = (2\a+\b+1)p-1$.
Then $M^{\a,\b}_*$ is of restricted weak type $(p,p)$ with respect to the measure space $(\mathbb{R}_+,x^{\delta}dx)$.
\end{propo}

\begin{propo} \label{prop:6pre}
Assume that $-1 < \a < -1/2$ and $\a+\b=1/2$. Let $1 < p < \infty$ and $\delta = (2\a+2)p-1$.
Then $M^{\a,\b}_*$ is of restricted weak type $(p,p)$ with respect to the measure space $(\mathbb{R}_+,x^{\delta}dx)$.
\end{propo}

\begin{proof}[{Proof of Proposition \ref{prop:5pre}}]
We follow the strategy of Colzani et al.\ \cite[pp.\,50--53]{CoCoSt}.
Let us write $E$ and $F$ in terms of bounds on $z$,
\begin{align*}
E & = \big\{ (t,x,z) \in \mathbb{R}^3_+ : |t-x| < z < t+x \big\}, \\
F & = \big\{ (t,x,z) \in \mathbb{R}^3_+ : z < t-x \big\}.
\end{align*}
In what follows we assume that $(t,x,z) \in E \cup F$ and split into the following three cases.
\begin{itemize}
\item[(i)] $t \le x$.
Then $(t,x,z) \in E$ and $x-t < z < x+t$, and we have, see \eqref{logKbnd},
$$
|K^{\a,\b}_t(x,z)| \lesssim \frac{(xz)^{-\a-1/2}}{t},
$$
since the argument of the logarithm is strictly between $2$ and $4$. Indeed, we have
$$
\frac{8xz}{(x+z)^2-t^2} \le \frac{8xz}{(x+z)^2-x^2} = \frac{8xz}{z(2x+z)} < \frac{8xz}{2xz} = 4.
$$
The lower bound by $2$ is easily verified using the constraint $|x-z|<t$.
\item[(ii)] $x < t$ and $t-x < z < t+x$.
Then $(t,x,z) \in E$ and one has the bound
$$
|K_t^{\a,\b}(x,z)| \lesssim \frac{(xz)^{-\a-1/2}}{t} \log \frac{8xz}{(x+z)^2-t^2}.
$$
\item[(iii)] $x < t$ and $z < t-x$.
Then $(t,x,z) \in F$ and, see \eqref{logKbnd},
$$
|K^{\a,\b}_t(x,z)| \lesssim \frac{(xz)^{-\a-1/2}}{t} \bigg[ \frac{xz}{t^2-(x-z)^2}\bigg]^{\a+1/2}
	\log\bigg( 2 \frac{t^2-(x-z)^2}{t^2-(x+z)^2}\bigg).
$$
\end{itemize}

Considering (i), we see that the part of $M^{\a,\b}_*$ emerging from restricting the supremum to $t \le x/2$ is controlled by 
$$
\sup_{t \le x/2} \frac{1}{t} \int_{x-t}^{x+t} (xz)^{-\a-1/2} f(z) z^{2\a+1}\, dz \simeq Lf(x).
$$
Furthermore, the part of $M^{\a,\b}_*$ coming from restriction to $x/2 < t \le x$ is controlled by
$$
\sup_{x/2 < t \le x} \frac{1}{t} x^{-\a-1/2} \int_{x-t}^{x+t} f(z) z^{\a+1/2}\, dz
	\simeq \frac{1}{x^{\a+3/2}} \int_0^{2x} f(z) z^{\a+1/2}\, dz \simeq H_{2\a+\b+1}(2x)
$$
(notice that $2\a+\b+1 = \a+3/2$, since $\a+\b=1/2$).
Both $L$ and $H_{2\a+\b+1}$ possess the desired mapping property, see Lemmas \ref{lem:H0} and \ref{lem:L}.

Next, we focus on the case $x < t$, see (ii) and (iii) above (then $0 < z < t+x$).
According to the function
$$
\varphi \colon z \mapsto \frac{4xz}{t^2-(x-z)^2},
$$
which is increasing for $z \in (0,t+x)$ and maps $(0,t+x)$ onto $(0,\infty)$, we split
$(E \cup F) \cap \{(t,x,z) \in \RR_+^3: x < t\}$ into the sets
\begin{align*}
\mathcal{D}_1 & = \big\{ (t,x,z) \in \mathbb{R}^3_+ : 0 < \varphi(z) < 1 - \varepsilon \big\} \cap \{ x < t\} \subset F, \\
\mathcal{D}_2 & = \big\{ (t,x,z) \in \mathbb{R}^3_+ : 1 - \varepsilon < \varphi(z) < 1 \big\} \cap \{ x < t\} \subset F, \\
\mathcal{D}_3 & = \big\{ (t,x,z) \in \mathbb{R}^3_+ : 1 < \varphi(z) < 1 + \varepsilon \big\} \cap \{ x < t\} \subset E, \\
\mathcal{D}_4 & = \big\{ (t,x,z) \in \mathbb{R}^3_+ : 1 + \varepsilon < \varphi(z) \big\} \cap \{ x < t\} \subset E;
\end{align*}
here $0 < \varepsilon < 1$ is fixed, to be specified in a moment.

On $\mathcal{D}_1$ and $\mathcal{D}_4$ the arguments of the logarithms in (iii) and (ii), respectively, are bounded from above, as can
easily be verified. Moreover, we have
$$
\bigg[ \frac{xz}{t^2-(x-z)^2}\bigg]^{\a+1/2} = \bigg( \frac{\varphi(z)}{4}\bigg)^{\a+1/2} < 1 \qquad \textrm{on}\;\; \mathcal{D}_1,
$$
since we assumed $\a > -1/2$. Thus
$$
|K^{\a,\b}(x,z)| \lesssim \frac{(xz)^{-\a-1/2}}{t} \qquad \textrm{on} \;\; \mathcal{D}_1 \cup \mathcal{D}_4. 
$$
It follows that the part of $M^{\a,\b}_*$ related to $\mathcal{D}_1 \cup \mathcal{D}_4$ is controlled by
$$
\frac{1}{x^{\a+3/2}} \int_0^{\infty} f(z) z^{\a+1/2}\, dz = H_{\a+3/2}f(x) + H_{-\a-3/2}^{\infty}f(x).
$$
By Lemmas \ref{lem:H0} and \ref{lem:Hinf}, $H_{\a+3/2}$ and $H_{-\a-3/2}^{\infty}$ have the desired mapping property
(recall again that $\a+3/2 = 2\a + \b +1$).

It remains to deal with the parts of $M^{\a,\b}_*$ corresponding to $\mathcal{D}_2$ and $\mathcal{D}_3$.
Choosing $\varepsilon$ sufficiently small ($\varepsilon = 1/4$ will do) we have the inclusions
\begin{align*}
\mathcal{D}_2 & \subset \big\{ (t-x)/2 < z < t-x \big\}, \\
\mathcal{D}_3 & \subset \big\{ (t-x) < z < t \wedge 2(t-x) \big\}
\end{align*}
(to verify them, use the monotonicity of $\varphi$; for instance, in case of $\mathcal{D}_2$ evaluate $\varphi$
at endpoints of $((t-x)/2,t-x)$ to see that $\varphi((t-x)/2) < 2/3$ and $\varphi(t-x)=1$; the case of
$\mathcal{D}_3$ is similar). Then
\begin{align*}
& \int \ind{\mathcal{D}_2}(t,x,z) |K_t^{\a,\b}(x,z)| f(z)\, d\mu_{\a}(z) \\
& \qquad  \lesssim
	\frac{1}{x^{\a+3/2}} \int_{\frac{t-x}2}^{t-x} z^{\a+1/2} \log\bigg( 2\frac{t^2-(x-z)^2}{t^2-(x+z)^2}\bigg) f(z)\, dz \\
& \qquad = \frac{1}{x^{\a+3/2}} \int_{\frac{t-x}2}^{t-x} z^{\a+1/2-\delta}
	\log\bigg( 2\frac{t^2-(x-z)^2}{t^2-(x+z)^2}\bigg) f(z) z^{\delta}\, dz.
\end{align*}
Here $\delta = (\a+3/2)p-1$, which implies $(\a+1/2-\delta)p' + \delta = -1$.
By H\"older's inequality, we can further estimate by
$$
\frac{1}{x^{\a+3/2}} \Bigg( \int_0^{\infty} f^p(z) z^{\delta}\, dz\Bigg)^{1/p}
\Bigg( \int_{\frac{t-x}2}^{t-x} \bigg[\log\bigg(2 \frac{t^2-(x-z)^2}{t^2-(x+z)^2}\bigg)\bigg]^{p'} \frac{dz}{z}\Bigg)^{1/p'},
$$
and it is straightforward to check that the second integral here is controlled by a constant independent of $t$ and $x$
(just estimate from above the argument of the logarithm by $2(t-x+z)/(t-x-z)$ and then change the variable $z\mapsto (t-x)z$
to arrive at a convergent integral independent of $t$ and $x$).

Analogously,
\begin{align*}
& \int \ind{\mathcal{D}_3}(t,x,z) |K_t^{\a,\b}(x,z)| f(z)\, d\mu_{\a}(z) \\
& \qquad \lesssim
\frac{1}{x^{\a+3/2}} \Bigg( \int_0^{\infty} f^p(z) z^{\delta}\, dz\Bigg)^{1/p}
\Bigg( \int_{t-x}^{t \wedge 2(t-x)} \bigg[\log\bigg( \frac{8xz}{(x+z)^2-t^2}\bigg)\bigg]^{p'} \frac{dz}{z}\Bigg)^{1/p'},
\end{align*}
and again it is not hard to see (just inspect the cases $x < t < 2x$ and $t > 2x$) that the integral involving
the logarithm is controlled by a constant independent of $t$ and $x$.

Altogether, this shows that the part of $M^{\a,\b}_*$ related to $\mathcal{D}_2 \cup \mathcal{D}_3$ is controlled by
the operator
$$
f \mapsto \frac{1}{x^{\a+3/2}} \|f\|_{L^p(x^{\delta}dx)},
$$
which is even weak type $(p,p)$ with respect to $(\mathbb{R}_+,x^{\delta}dx)$.
This finishes the proof.
\end{proof}

We now extend the arguments from the proof of Proposition \ref{prop:5pre} to cover the remaining case $-1 < \a < -1/2$
from Proposition \ref{prop:6pre}.

\begin{proof}[{Proof of Proposition \ref{prop:6pre}}]
We proceed as in the proof of Proposition \ref{prop:5pre}.
Considering (i) and that part of $M^{\a,\b}_*$, we get the same control in terms of $L$ and $H_{2\a+\b+1}$.
These operators possess the desired mapping property, see Lemmas \ref{lem:H0} and \ref{lem:L}
(notice that now $2\a+\b+1 > 2\a+2$).

We pass to (ii) and (iii). Thus we have $x<t$ and $0 < z < t+x < 2t$.
We first look at $\mathcal{D}_1$ and $\mathcal{D}_4$.
Recalling that $-1 < \a < -1/2$, on $\mathcal{D}_1$ we have
\begin{equation} \label{fct1}
|K_t^{\a,\b}(x,z)| \lesssim \frac{1}{t} \big[ t^2-(x-z)^2\big]^{-\a-1/2} \le \frac{1}{t^{2\a+2}} < \frac{1}{x^{2\a+2}},
\end{equation}
while on $\mathcal{D}_4$
\begin{equation} \label{fct2}
|K^{\a,\b}_t(x,z)| \lesssim \frac{(xz)^{-\a-1/2}}{t} \lesssim \frac{t^{-2\a-1}}t < \frac{1}{x^{2\a+2}}.
\end{equation}
Therefore, the part of $M^{\a,\b}_*$ related to $\mathcal{D}_1 \cup \mathcal{D}_4$ is controlled by
$$
\frac{1}{x^{2\a+2}} \int_0^{\infty} f(z) z^{2\a+1}\, dz = H_{2\a+2}f(x) + H_{-2\a-2}^{\infty}f(x).
$$
By Lemmas \ref{lem:H0} and \ref{lem:Hinf}, $H_{2\a+2}$ and $H_{-2\a-2}^{\infty}$ have the mapping property in question.

It remains to deal with $\mathcal{D}_2$ and $\mathcal{D}_3$.
Since the non-logarithmic factors in the kernel bound can be estimated as in \eqref{fct1} and \eqref{fct2},
the part of $M^{\a,\b}_*$ related to $\mathcal{D}_2 \cup \mathcal{D}_3$ is controlled by
\begin{align*}
& \frac{1}{x^{2\a+2}} \int_{\frac{t-x}2}^{t-x} z^{2\a+1-\delta} \log\bigg( 2 \frac{t^2-(x-z)^2}{t^2-(x+z)^2}\bigg)
	f(z) z^{\delta}\, dz \\
& \qquad + \frac{1}{x^{2\a+2}} \int_{t-x}^{t \wedge 2(t-x)} z^{2\a+1-\delta}
	\log\bigg( \frac{8xz}{(x+z)^2-t^2}\bigg) f(z) z^{\delta} \, dz.
\end{align*}
Here $\delta = (2\a+2)p-1$, and one has $(2\a+1-\delta)p' + \delta = -1$.
Applying H\"older's inequality twice as in the proof of Proposition \ref{prop:5pre}, we arrive at the control by the operator
$$
f \mapsto \frac{1}{x^{2\a+2}} \|f\|_{L^p(x^{\delta}dx)},
$$
which is even weak type $(p,p)$ with respect to $(\mathbb{R}_+,x^{\delta}dx)$.
The conclusion follows.
\end{proof}

\subsection{The singular case: $\boldsymbol{-1/2 < \a + \b < 1/2}$} \, \medskip

In this subsection we consider $\a > -1$ and $\b \in \mathbb{R}$ such that $-1/2 < \a + \b < 1/2$.
First, we need some technical preparation.
In consistence with the notation from \cite{CNR2}, let
\begin{align*} 
\Psi_{*,E}^{\alpha,\beta} f (x) &
= \sup_{t > 0}
\frac{x^{-2\alpha - \beta}}{ t^{2\alpha + 2\beta}} \int_{\abs{t-x}}^{t+x}
\Big( 
\big[(x+t)^2 - z^2 \big] \big[z^2 - (x-t)^2 \big]
\Big)^{\alpha + \beta - 1/2} z^{1-\beta} f(z) \, dz, \\
\Psi_{*,F}^{\alpha,\beta} f (x)
& = \sup_{t > x}
 \frac{1}{t^{2\alpha + 2\beta}} \int_0^{t-x} \big[ t^2-(x-z)^2\big]^{-\a-1/2}
	\big[t^2-(x+z)^2\big]^{\a+\b-1/2} z^{2\a+1} f(z) \, dz.
\end{align*}

We prove the following result, which is inspired by \cite[Theorem 3.1 (b)]{DMO}.
\begin{lem} \label{lem:DRgen}
Let $\alpha > -1$,
$- 1/2 < \alpha + \beta < 1/2$ and $2\alpha + \beta > -1$. Then
\begin{align*} 
\Psi_{*,E}^{\alpha,\beta} \ind{\mathcal{A}} (x) 
& \lesssim 
\big[E_{2,\beta/(\alpha + \beta + 1/2)}\ind{\mathcal{A}}(x)\big]^{\alpha + \beta + 1/2} 
+ \big[R_{\beta/(\alpha + \beta + 1/2)}\ind{\mathcal{A}}(x)\big]^{\alpha + \beta + 1/2}, \\
\Psi_{*,F}^{\alpha,\beta} \ind{\mathcal{A}} (x) 
& \lesssim 
N_{2\alpha + 2}  \ind{\mathcal{A}} (x) 
+
T_{1 - \beta}  \ind{\mathcal{A}} (x) 
+
\big[T_{(\alpha +1/2)/(\alpha + \beta + 1/2)} \ind{\mathcal{A}}(x)\big]^{\alpha + \beta + 1/2}  \\
& \quad + 
\big[H_{(\alpha +1/2 + (1-\beta)\wedge 0)/(\alpha + \beta + 1/2) + 1}\ind{\mathcal{A}}(2x)\big]^{\alpha + \beta + 1/2},
\end{align*}
uniformly in $x > 0$ and measurable subsets $\mathcal{A}$ of $\mathbb{R}_+$.
\end{lem}

In order to prove Lemma~\ref{lem:DRgen} we need the following technical result.

\begin{lem} \label{lem:71}
Let $\eta > 0$ and $\lambda \ge 1$ be fixed. Then
\begin{align} \label{R1.1}
\Big( \sum_{k\in \ZZ} 2^{k\eta} b_k \Big)^\lambda
\lesssim
\sum_{k\in \ZZ} 2^{k\eta \lambda} b_k,
\end{align}
uniformly in $0 \le b_k \le 1$, $k \in \ZZ$.
\end{lem}

\begin{proof}
We first prove \eqref{R1.1} for $\lambda \in \NN \setminus \{0\}$ and $\eta > 0$.
Observe that for $\lambda = 1$ there is nothing to prove. Further, using an induction argument it suffices to prove that 
\begin{align*} 
\sum_{k, l\in \ZZ} 2^{k\eta \lambda } b_k 2^{l\eta} b_l
\lesssim
\sum_{k\in \ZZ} 2^{k\eta (\lambda+1) } b_k, 
\qquad 0 \le b_k \le 1, \quad k \in \ZZ.
\end{align*}
This, however, follows thanks to the assumption that $\eta > 0$ and the estimates
\begin{align*} 
\sum_{\substack{ l \in \ZZ \\ l \le k}}
2^{l\eta} b_l
& \le 
\sum_{\substack{ l \in \ZZ \\ l \le k}}
2^{l\eta} 
\simeq 
2^{k\eta}, \qquad k \in \ZZ, \\
\sum_{\substack{ k \in \ZZ \\ k \le l}}
2^{k\eta \lambda } b_k
& \le 
\sum_{\substack{ k \in \ZZ \\ k \le l}}
2^{k\eta \lambda }
\simeq 
2^{l\eta \lambda }, \qquad l \in \ZZ.
\end{align*}

To finish the proof of Lemma~\ref{lem:71}, it suffices to show that if \eqref{R1.1} holds true for some
$\lambda > 1$ and any $\eta > 0$, then it holds true also for any $1< \widetilde{\lambda} < \lambda$  and any $\eta > 0$. 
To justify this, we define 
$\gamma := (\lambda - \widetilde{\lambda})/(\lambda - 1) \in (0,1)$ and
$p:= \widetilde{\lambda} \gamma^{-1} > 1$. Using H\"older's inequality we get
\begin{align*} 
\Big( \sum_{k\in \ZZ} 2^{k\eta} b_k \Big)^{\widetilde{\lambda}}
=
\Big( \sum_{k\in \ZZ} 2^{k\eta \gamma } 2^{k\eta (1 - \gamma) } b_k
\Big)^{\widetilde{\lambda}}
\le 
\Big( \sum_{k\in \ZZ} 2^{k\eta \widetilde{\lambda} } b_k \Big)^{\gamma}
\Big( \sum_{k\in \ZZ} 2^{k\eta (1 - \gamma) p'} b_k \Big)^{\widetilde{\lambda}/p'}.
\end{align*}
Therefore, it is enough to show that 
\begin{align*} 
\Big( \sum_{k\in \ZZ} 2^{k\eta (1 - \gamma) p'} b_k \Big)^{\widetilde{\lambda}/(1-\gamma)p'}
\lesssim
\sum_{k\in \ZZ} 2^{k\eta \widetilde{\lambda}} b_k.
\end{align*}
This, however, follows from our assumption, since $\widetilde{\lambda}/(1-\gamma)p' = \lambda$.
\end{proof}

\begin{proof}[{Proof of Lemma \ref{lem:DRgen}, the $\Psi_{*,E}^{\a,\b}$ part}]
Denote
$$
A= |x-t|, \qquad B = x+t
$$
and observe that
\begin{align} \label{ctr17gen}
\Psi_{*,E}^{\alpha,\beta} \ind{\mathcal{A}} (x) 
\simeq
x^\beta
\sup_{t>0} \frac{1}{(B^2-A^2)^{2\alpha + 2 \beta}} 
\int_A^B \Big[ \big(z^2-A^2\big)\big(B^2-z^2\big) \Big]^{\alpha + \beta - 1/2}
z^{1 - \beta} \ind{\mathcal{A}} (z) \, dz.
\end{align}
It is not hard to see that our task of estimating $\Psi_{*,E}^{\alpha,\beta}$ reduces to showing that
\begin{align} \nonumber
& \frac{1}{(B^2-A^2)^{2\alpha + 2 \beta}} 
\int_A^B \Big[ \big(z^2-A^2\big)\big(B^2-z^2\big) \Big]^{\alpha + \beta - 1/2}
z^{1 - \beta} \ind{\mathcal{A}} (z) \, dz \\ \label{R3.2}
& \qquad \qquad \lesssim
\bigg( \frac{1}{B^2-A^2} \int_A^{B} 
z^{1 - \beta/(\alpha + \beta + 1/2)} 
\ind{\mathcal{A}} (z) \, dz \bigg)^{\alpha + \beta + 1/2},
\qquad 0 \le A < B < \infty.
\end{align}
Indeed, having \eqref{R3.2} we proceed by splitting the supremum into $t \le 2x$ and $t > 2x$.
In the first case $A \le x$ and we get the control by 
$\big[E_{2,\beta/(\alpha + \beta + 1/2)}\ind{\mathcal{A}}(x)\big]^{\alpha + \beta + 1/2}$.
On the other hand, if $t > 2x$, then $z \simeq t$ and we get the control of this part of the maximal operator
by $\big[R_{\beta/(\alpha + \beta + 1/2)}\ind{\mathcal{A}}(x)\big]^{\alpha + \beta + 1/2}$. The conclusion follows.

It remains to prove \eqref{R3.2}.
Let us focus first on the subintegral over $I_1 = (A,(A+B)/2) \cap \mathcal{A}$.
For $z \in I_1$ one has $B^2-z^2 \simeq B(B-A) \simeq B^2-A^2$ and 
$z^2 - A^2 \simeq z(z-A)$.
Consequently, proving this counterpart of \eqref{R3.2} reduces to showing the estimate
\begin{align} \label{R4.2}
\int_{I_1} (z-A)^{\alpha + \beta - 1/2} z^{\alpha + 1/2} \, dz 
\lesssim
\bigg( \int_{I_1}
z^{(\alpha + 1/2)/(\alpha + \beta + 1/2)} \, dz \bigg)^{\alpha + \beta + 1/2},
\qquad 0 \le A < B < \infty.
\end{align}

To proceed, denote $l_1 = |I_1|$. We consider two cases.

\noindent \textbf{Case 1:} $A \ge B/100$.
In this case we have $A \simeq B$ and $z \simeq B$ for $z \in I_1$.
Therefore, showing \eqref{R4.2} reduces to proving
\begin{align} \label{R5.1}
\int_{I_1} (z-A)^{\alpha + \beta - 1/2} \, dz 
\lesssim
l_1^{\alpha + \beta + 1/2},
\qquad 0 \le A < B < \infty.
\end{align}
Since the function
$$
z \mapsto (z-A)^{\alpha + \beta - 1/2}, \qquad z > A,
$$
is decreasing, we see that
$$
\texttt{LHS}\eqref{R5.1} 
\le
\int_A^{A+l_1} (z-A)^{\alpha + \beta - 1/2} \, dz  
\simeq l_1^{\alpha + \beta + 1/2}
= 
\texttt{RHS}\eqref{R5.1}.
$$
This gives \eqref{R5.1} and finishes Case 1.

\noindent \textbf{Case 2:} $A < B/100$.
We aim at the same bound as in the other cases.
We split the first integral in \eqref{R4.2} into two parts, say $J_1$ and $J_2$, which correspond, respectively,
to integration over $\mathcal{A}_1 = (A, 10A) \cap I_1$ and $\mathcal{A}_2 = (10A,(A+B)/2)\cap I_1$.
Observe that the desired bound for $J_1$ follows from \eqref{R5.1} by taking there $B= 19A$. So it remains to deal with $J_2$.

Notice that $z-A \simeq z$ for $z \ge 10A$, thus we get
$$
J_2 \simeq \int_{\mathcal{A}_2} z^{2\alpha + \beta} \, dz.
$$
Then our task reduces to showing that
\begin{align} \label{id:72}
\int_{\mathcal{A}_2} z^{2\alpha + \beta} \, dz
\lesssim
\bigg( \int_{\mathcal{A}_2}
z^{(\alpha + 1/2)/(\alpha + \beta + 1/2)} \, dz \bigg)^{\alpha + \beta + 1/2},
\end{align}
uniformly in measurable subsets $\mathcal{A}_2$ of $\mathbb{R}_+$.
Denote $\mathcal{B}_k = \mathcal{A}_2 \cap [2^k,2^{k+1})$ and $b_k = \abs{\mathcal{B}_k} 2^{-k} \in [0,1]$, $k\in \ZZ$,
and notice that the above is equivalent to checking that
\begin{align*} 
\sum_{k\in \ZZ} 2^{k(2\alpha + \beta + 1)} b_k 
\lesssim
\Big( \sum_{k\in \ZZ} 2^{k (\alpha + 1/2)/(\alpha + \beta + 1/2) + k} b_k \Big)^{\alpha + \beta + 1/2}, 
\qquad 0 \le b_k \le 1.
\end{align*}
But this follows from Lemma~\ref{lem:71} specified to $\eta = 2\alpha + \beta + 1 > 0$ and $\lambda = 1/(\alpha + \beta + 1/2) > 1$. 
This concludes the analysis related to Case 2, hence also the analysis of the subintegral over $I_1$.

Next, consider the complementary subintegral over $I_2 = ((A+B)/2,B) \cap \mathcal{A}$.
For $z \in I_2$ one has $z^2-A^2 \simeq B(B-A)\simeq B^2-A^2$, $B^2 - z^2 \simeq B(B-z)$ and $z \simeq B$.
Denote $l_2 = |I_2|$ and observe that our task of proving this counterpart of \eqref{R3.2} is equivalent to showing that
\begin{align} \label{R4.1}
\int_{I_2} (B-z)^{\alpha + \beta - 1/2} \, dz
\lesssim
l_2^{\alpha + \beta + 1/2},
\end{align}
uniformly in $B > 0$ and any measurable subset $I_2$ of $(0,B)$.

Since the function
$$
z \mapsto (B-z)^{\alpha + \beta - 1/2}, \qquad z \in (0,B),
$$
is increasing, we get
\begin{align*}
\texttt{LHS}\eqref{R4.1} 
\le 
\int_{B - l_2}^B (B-z)^{\alpha + \beta - 1/2} \, dz
\simeq
l_2^{\alpha + \beta + 1/2}
=
\texttt{RHS}\eqref{R4.1}.
\end{align*}
This gives \eqref{R4.1}.

The proof of the bound for $\Psi_{*,E}^{\a,\b}\ind{\mathcal{A}}$ in Lemma~\ref{lem:DRgen} is finished.
\end{proof}

\begin{proof}[Proof of Lemma~\ref{lem:DRgen}, the $\Psi_{*,F}^{\a,\b}$ part]
In consistence with the notation from \cite{CNR2} (see Figure~\ref{fig_reg}), we have
\begin{align*}
\Psi_{*,F}^{\alpha,\beta} \ind{\mathcal{A}} (x) 
& \le 
\Psi_{*,F_1}^{\alpha,\beta} \ind{\mathcal{A}} (x)
+
\Psi_{*,F_2'}^{\alpha,\beta} \ind{\mathcal{A}} (x) 
+
\Psi_{*,F_2''}^{\alpha,\beta} \ind{\mathcal{A}} (x),
\end{align*}
where
\begin{align*}
\Psi_{*,F_1}^{\alpha,\beta} f (x) 
& =
\sup_{x < t < 3x} \frac{1}{t^{2\alpha + 2\beta}} \int_0^{t-x} \big[ t^2-(x-z)^2\big]^{-\a-\frac12}
	\big[t^2-(x+z)^2\big]^{\a+\b-\frac12} z^{2\a+1} f(z) \, dz, \\
\Psi_{*,F_2'}^{\alpha,\beta} f (x) 
& =
\sup_{t \ge 3x} \frac{1}{t^{2\alpha + 2\beta}} \int_0^{(t-x)/2} \big[ t^2-(x-z)^2\big]^{-\a-\frac12}
	\big[t^2-(x+z)^2\big]^{\a+\b-\frac12} z^{2\a+1} f(z) \, dz, \\
\Psi_{*,F_2''}^{\alpha,\beta} f (x) 
& =
\sup_{t \ge 3x}
\frac{1}{t^{2\alpha + 2\beta}} \int_{(t-x)/2}^{t-x} \big[ t^2-(x-z)^2\big]^{-\a-\frac12}
	\big[t^2-(x+z)^2\big]^{\a+\b-\frac12} z^{2\a+1} f(z) \, dz.
\end{align*}
It suffices to show that 
\begin{align} \label{T3}
\Psi_{*,F_2'}^{\alpha,\beta} \ind{\mathcal{A}} (x) 
& \lesssim
N_{2\alpha + 2}  \ind{\mathcal{A}} (x), \\ \label{T5.1}
\Psi_{*,F_2''}^{\alpha,\beta} \ind{\mathcal{A}} (x) 
& \lesssim
T_{1 - \beta}  \ind{\mathcal{A}} (x) 
+
\big[T_{(\alpha +1/2)/(\alpha + \beta + 1/2)} \ind{\mathcal{A}}(x)\big]^{\alpha + \beta + 1/2}, \\ \label{T8.1}
\Psi_{*,F_1}^{\alpha,\beta} \ind{\mathcal{A}} (x) 
& \lesssim
\big[H_{(\alpha +1/2 + (1-\beta)\wedge 0)/(\alpha + \beta + 1/2) + 1}\ind{\mathcal{A}}(2x)\big]^{\alpha + \beta + 1/2}.
\end{align}

Notice that 
\begin{align*} 
\Psi_{*,F_2'}^{\alpha,\beta} f (x) 
\simeq
\sup_{t \ge 3x} t^{-2\alpha - 2} \int_0^{(t-x)/2} 
z^{2\a+1} f(z) \, dz
\le
N_{2\alpha + 2}  f (x),
\end{align*}
and \eqref{T3} follows.

Next, we deal with $\Psi_{*,F_2''}^{\alpha,\beta}$.
We have 
\begin{align*} 
\Psi_{*,F_2''}^{\alpha,\beta} f (x) 
& \simeq
\sup_{t \ge 3x} t^{-\beta}  \int_{(t-x)/2}^{t-x} (t+x - z)^{-\alpha - 1/2}
(t - x - z)^{\alpha + \beta - 1/2} f(z) \, dz \\
& \simeq
\sup_{t \ge 2x} t^{-\beta} \int_{t/2}^{t} 
(t - z + x)^{-\alpha - 1/2}
(t - z)^{\alpha + \beta - 1/2} f(z) \, dz \\
& =: J_1 f (x) + J_2 f(x),
\end{align*}
where $J_1$ and $J_2$ correspond to the integration over 
$(t/2,t-x)$ and $(t-x,t)$, respectively. It is straightforward to see that
\begin{align*} 
J_1 f (x) 
& \simeq
\sup_{t \ge 2x} t^{-\beta} \int_{t/2}^{t-x} 
(t - z + x)^{\beta - 1} f(z) \, dz 
\le
\sup_{t \ge 2x} t^{-\beta} \int_{t/2}^{t} 
(t - z + x)^{\beta - 1} f(z) \, dz \\
& \simeq
T_{1-\beta} f(x).
\end{align*}
Therefore, in order to prove \eqref{T5.1} it suffices to show that
\begin{align} \label{T6.1}
J_2 \ind{\mathcal{A}}(x)
\lesssim
\big[T_{(\alpha +1/2)/(\alpha + \beta + 1/2)} \ind{\mathcal{A}}(x)\big]^{\alpha + \beta + 1/2}.
\end{align}
Further, since 
\begin{align*} 
J_2 \ind{\mathcal{A}}(x)
\simeq
\sup_{t \ge 2x} t^{-\beta} \int_{t - x}^{t} 
x^{-\alpha - 1/2}
(t - z)^{\alpha + \beta - 1/2} \ind{\mathcal{A}} (z) \, dz,
\end{align*}
we see that proving \eqref{T6.1} reduces to justifying that
\begin{align} \label{T7.1}
\int_{t - x}^{t}
(t - z)^{\alpha + \beta - 1/2} \ind{\mathcal{A}} (z) \, dz
\lesssim
\bigg( \int_{t - x}^{t} \ind{\mathcal{A}} (z) \, dz \bigg)^{\alpha + \beta + 1/2}, 
\qquad t \ge 2x, \quad \mathcal{A} \subseteq \RR_+.
\end{align}

The last estimate is relatively easily verified. Indeed, $\alpha + \beta - 1/2 < 0$, so the function
$$
z \mapsto (t-z)^{\alpha + \beta - 1/2}, \qquad z \in (0,t),
$$
is increasing. Denoting $l := \abs{(t-x,t) \cap \mathcal{A}}$, we see that 
\begin{align*} 
\texttt{LHS}\eqref{T7.1}
\le 
\int_{t - l}^{t}
(t - z)^{\alpha + \beta - 1/2} \, dz
=
\int_0^{l} z^{\alpha + \beta - 1/2} \, dz
\simeq
l^{\alpha + \beta + 1/2}
=
\texttt{RHS}\eqref{T7.1}.
\end{align*}
This gives \eqref{T7.1}, thus also \eqref{T5.1}.

Finally, we prove \eqref{T8.1}. We have 
\begin{align*} 
\Psi_{*,F_1}^{\alpha,\beta} f (x) 
\simeq
\sup_{x < t < 3x} x^{-2\alpha - \beta - 1} (t-x)^{-\alpha - 1/2} 
\int_0^{t-x} (t-x-z)^{\alpha + \beta - 1/2} z^{2\alpha + 1}
	 f(z) \, dz.
\end{align*}
Therefore, in order to prove \eqref{T8.1} it is enough to check that
\begin{align} \nonumber
& \sup_{x < t < 3x} x^{(\beta \wedge 1) - \beta} (t-x)^{-\alpha - 1/2} 
\int_0^{t-x} (t-x-z)^{\alpha + \beta - 1/2} z^{2\alpha + 1}
\ind{\mathcal{A}} (z) \, dz \\ \label{T9.1}
& \qquad \qquad \lesssim
\bigg( \int_{0}^{2x} 
z^{\frac{2\alpha + (\beta \wedge 1) + 1}{\alpha + \beta + 1/2} -1}
\ind{\mathcal{A}} (z) \, dz \bigg)^{\alpha + \beta + 1/2}, 
\qquad x>0, \quad \mathcal{A} \subseteq \RR_+.
\end{align}
To justify the last estimate, it is convenient to split the set of integration in the integral on the left-hand side
into subintervals $(0, (t-x)/2)$ and $((t-x)/2, t-x)$, denoting the corresponding operators by $P_1$ and $P_2$, respectively.

We first deal with $P_2$. Observe that
\begin{align*} 
P_2 \ind{\mathcal{A}} (x)
\simeq
\sup_{x < t < 3x} x^{(\beta \wedge 1) - \beta} (t-x)^{\alpha + 1/2} 
\int_{(t-x)/2}^{t-x} (t-x-z)^{\alpha + \beta - 1/2} 
\ind{\mathcal{A}} (z) \, dz.
\end{align*}
Denote $l_1 := \abs{((t-x)/2, t-x) \cap \mathcal{A}}$. Since  the function
$$
z \mapsto (t-x-z)^{\alpha + \beta - 1/2}, \qquad z \in (0,t-x),
$$
is increasing, we get
\begin{align*} 
\int_{(t-x)/2}^{t-x} (t-x-z)^{\alpha + \beta - 1/2} 
\ind{\mathcal{A}} (z) \, dz 
\le 
\int_{t-x - l_1}^{t-x} (t-x-z)^{\alpha + \beta - 1/2} \, dz
\simeq 
l_1^{\alpha + \beta + 1/2}.
\end{align*}
Using this we get
\begin{align*} 
P_2 \ind{\mathcal{A}} (x)
\lesssim
\sup_{x < t < 3x} x^{(\beta \wedge 1) - \beta} 
(t-x)^{ \beta - (\beta \wedge 1)} 
\bigg( \int_{(t-x)/2}^{t-x} 
z^{\frac{2\alpha + (\beta \wedge 1) + 1}{\alpha + \beta + 1/2} -1}
\ind{\mathcal{A}} (z) \, dz \bigg)^{\alpha + \beta + 1/2}.
\end{align*}
Consequently, thanks to the fact that $\beta - (\beta \wedge 1) \ge 0$, we see that the right-hand side above is dominated
by the right-hand side of \eqref{T9.1}. Thus the required bound for $P_2$ follows.

Next, we focus on $P_1$. Notice that in this case our problem is to check that
\begin{align} \nonumber
& \sup_{x < t < 3x} x^{(\beta \wedge 1) - \beta} (t-x)^{\beta - 1} 
\int_0^{(t-x)/2} z^{2\alpha + 1}
\ind{\mathcal{A}} (z) \, dz \\ \label{T10.2}
& \qquad \qquad \lesssim
\bigg( \int_{0}^{2x} 
z^{\frac{2\alpha + (\beta \wedge 1) + 1}{\alpha + \beta + 1/2} -1}
\ind{\mathcal{A}} (z) \, dz \bigg)^{\alpha + \beta + 1/2}, 
\qquad x>0, \quad \mathcal{A} \subseteq \RR_+.
\end{align}
If $\beta \le 1$, then 
\begin{align*} 
\texttt{LHS}\eqref{T10.2} 
\lesssim
\int_0^{(t-x)/2} z^{2\alpha + \beta}
\ind{\mathcal{A}} (z) \, dz.
\end{align*}
Using now \eqref{id:72} we get the conclusion.
It remains to consider $\beta > 1$. We have 
$x^{(\beta \wedge 1) - \beta} (t-x)^{\beta - 1} \lesssim 1$ and proceeding as in the proof of \eqref{id:72}
we reduce the task to checking that
\begin{align*} 
\sum_{k\in \ZZ} 2^{k(2\alpha + 2)} b_k 
\lesssim
\Big( \sum_{k\in \ZZ} 2^{k (2\alpha + 2)/(\alpha + \beta + 1/2) } b_k \Big)^{\alpha + \beta + 1/2}, 
\qquad 0 \le b_k \le 1.
\end{align*}
This, however, follows from Lemma~\ref{lem:71} specified to 
$\lambda = (\alpha + \beta + 1/2)^{-1}$ and $\eta = 2\alpha + 2$.
Now \eqref{T10.2} follows and this finishes showing \eqref{T8.1}.

The proof of the bound for $\Psi_{*,F}^{\a,\b}\ind{\mathcal{A}}$ in Lemma~\ref{lem:DRgen} is complete.
\end{proof}

We are now in a position to justify the sufficiency parts in Theorem \ref{thm:weak} in the considered case
$\a + \b < 1/2$. In view of Theorem \ref{thm:strong} (its already proved sufficiency part, to be precise),
we are concerned only with restricted weak type boundedness.
More specifically, it is enough we prove the following.
\begin{propo} \label{prop:wbd}
Assume that $\a > -1$ and $-1/2 < \a + \b < 1/2$ and $1 \le p < \infty$.
If Condition \eqref{cnnn} holds 
and
\begin{equation} \label{dcq}
-1 < \delta, \qquad -\b p \le \delta \le (2\a + (\b \wedge 1) + 1)p-1,
\end{equation}
then $M^{\a,\b}_*$ is of restricted weak type $(p,p)$ with respect to the measure space $(\mathbb{R}_+,x^{\delta}dx)$.
\end{propo}

\begin{proof}
We may assume that $2\a + \b > -1$, since otherwise \eqref{dcq} is not satisfied.
In view of \cite[Theorem 2.2]{CNR2}, one has the control
$$
M^{\a,\b}_{*}f(x) \lesssim \Psi^{\a,\b}_{*}f(x) \le \Psi^{\a,\b}_{*,E}f(x) + \Psi^{\a,\b}_{*,F}f(x), \qquad x > 0.
$$
Combining this with Lemma~\ref{lem:DRgen}
we further obtain
\begin{align*} 
M^{\a,\b}_* \ind{\mathcal{A}} (x) 
& \lesssim 
\big[E_{2,\beta/(\alpha + \beta + 1/2)}\ind{\mathcal{A}}(x)\big]^{\alpha + \beta + 1/2} 
+ \big[R_{\beta/(\alpha + \beta + 1/2)}\ind{\mathcal{A}}(x)\big]^{\alpha + \beta + 1/2} \\
& \quad + 
N_{2\alpha + 2}  \ind{\mathcal{A}} (x) 
+
T_{1 - \beta}  \ind{\mathcal{A}} (x) 
+
\big[T_{(\alpha +1/2)/(\alpha + \beta + 1/2)} \ind{\mathcal{A}}(x)\big]^{\alpha + \beta + 1/2}  \\
& \quad + 
\big[H_{(\alpha +1/2 + (1-\beta)\wedge 0)/(\alpha + \beta + 1/2) + 1}\ind{\mathcal{A}}(2x)\big]^{\alpha + \beta + 1/2},
\end{align*}
uniformly in $x > 0$ and measurable subsets $\mathcal{A}$ of $\mathbb{R}_+$.
Two of the controlling operators appearing here without powers, namely $N_{2\a+2}$ and $T_{1-\b}$,
are immediately seen to be of restricted
weak type $(p,p)$, by virtue of Lemmas \ref{lem:N2} and \ref{lem:Tchar}.
Therefore, it remains to analyze the remaining part of the right-hand side above.

We now argue similarly as in the proof of \cite[Theorem 1.1(b)]{DMO}. 
Let $\Upsilon^{\a,\b}$ stand for any of the operators
$$
E_{2,\beta/(\alpha + \beta + 1/2)}, \quad R_{\beta/(\alpha + \beta + 1/2)}, \quad
T_{(\alpha +1/2)/(\alpha + \beta + 1/2)}, \quad H_{(\alpha +1/2 + (1-\beta)\wedge 0)/(\alpha + \beta + 1/2) + 1}
$$
and let $q = (\alpha + \beta + 1/2)p$ (notice that $q \ge 1$, by \eqref{cnnn}). Then, in order to finish the proof
it suffices to show that
\begin{align} \label{id:337}
\int_{\{x>0\,:\,[\Upsilon^{\a,\b} \ind{\mathcal{A}}(x)]^{\a + \b+1/2} >\lambda \}} z^{\delta}\, dz 
\lesssim
\lambda^{-p} \| \ind{\mathcal{A}} \|^p_{L^p(x^\delta dx)}, 
\end{align}
uniformly in $\lambda>0$ and measurable subsets $\mathcal{A}$ of $\mathbb{R}_+$.

To proceed, we invoke Lemma \ref{lem:H0} and Lemmas \ref{lem:R2}--\ref{lem:Tchar},
and infer that each of the operators $\Upsilon^{\a,\b}$ is of restricted weak type $(q,q)$
with respect to $(\mathbb{R}_{+},x^{\delta} dx)$ for $\delta$ under consideration.
Using this we can write
\begin{align*} 
\int_{\{x>0\,:\,[\Upsilon^{\a,\b} \ind{\mathcal{A}}(x)]^{\a + \b+1/2} >\lambda \}} z^{\delta}\, dz 
& =
\int_{\{x>0\,:\,\Upsilon^{\a,\b} \ind{\mathcal{A}}(x) >\lambda^{1/(\a + \b+1/2)} \}} z^{\delta}\, dz \\
& \lesssim
\lambda^{-q/(\a + \b+1/2)} \| \ind{\mathcal{A}} \|^q_{L^q(x^\delta dx)} \\
& =
\lambda^{-p} \| \ind{\mathcal{A}} \|^p_{L^p(x^\delta dx)}, 
\end{align*}
and \eqref{id:337} follows.
\end{proof}

\section{Proof of Theorems \ref{thm:strong} and \ref{thm:weak}, necessity parts} \label{sec:necpr}

The necessity part of Theorem \ref{thm:strong} is contained in \cite[Proposition 1.6]{CNR2}, except for
the case $p=1$. The following result complements \cite[Proposition 1.6]{CNR2} and also
gives the relevant negative results pertaining to the weak type and restricted weak type boundedness of $M^{\a,\b}_{*}$.
Altogether, this provides the necessity parts in Theorems \ref{thm:strong} and \ref{thm:weak}.

\begin{propo} \label{prop:necBC}
	Assume that $\a > -1$ and $\a + \b > -1/2$. Let $1 \le p < \infty$ and $\delta \in \mathbb{R}$.
	Then $M^{\a,\b}_{*}$ lacks the following mapping properties with respect to the space $(\mathbb{R}_+,x^{\delta}dx)$.
	\begin{itemize}
		\item[(a)]
		$M^{\a,\b}_{*}$ is not of strong type $(1,1)$.
		\item[(b)]
		$M^{\a,\b}_*$ is not of weak type $(p,p)$ when $p > 1$ and
		$$
		\bigg[\;\frac{1}p = \a + \b + \frac{1}2 \quad \textrm{or} \quad \delta = (2\a+\b+1)p-1 \quad \textrm{or} \quad \delta = (2\a+2)p-1
		\;\bigg].
		$$
		\item[(c)]
		$M^{\a,\b}_*$ is not of restricted weak type $(p,p)$ when
		\begin{align*}
			& \frac{1}p > \a + \b + \frac{1}2 \quad \textrm{or} \quad \delta > (2\a+\b+1)p-1 \quad \textrm{or} \quad \delta > (2\a+2)p-1 \\
			& \textrm{or} \quad \delta < -\b p \quad \textrm{or} \quad \delta \le - 1 \quad \textrm{or} \quad
			\big[p=1 \;\; \textrm{and}\;\; \a+\b=1/2 \;\; \textrm{and}\;\; \beta \notin \mathbb{Z}\big].
		\end{align*}
	\end{itemize}
\end{propo}

To prove Proposition \ref{prop:necBC} we will give suitable counterexamples in a similar spirit as it was done in the
proof of \cite[Proposition 1.6]{CNR2}, see \cite[Section 4]{CNR2}.
Thus we first recall some facts and notation from \cite{CNR2}.

Consider the following integral operators, with positive kernels, acting on functions on~$\mathbb{R}_+$:
\begin{align*}
U_{t,1}^{\a,\b}f(x) & = \frac{x^{-2\a-\b}}{t^{2\a+2\b}} \int_{|t-x|}^{t+x} \Big( \big[(x+t)^2-z^2\big]
	\big[ z^2 - (x-t)^2\big] \Big)^{\a+\b-1/2} z^{1-\b} f(z)\, dz,\\
U_{t,2}^{\a,\b}f(x) & = \frac{x^{-\a-1/2}}{t^{2\a+2\b}} \int_{|t-x|}^{t+x} \big[ t^2-(x-z)^2 \big]^{\a+\b-1/2}
	z^{\a+1/2} f(z)\, dz, \\
\widetilde{U}_{t,2}^{\a,\b}f(x) & = \frac{x^{-\a-1/2}}{t^{2\a+2\b}} \int_{|t-x|}^{t+x} \big[ t^2-(x-z)^2 \big]^{\a+\b-1/2}
	\log\bigg( \frac{8xz}{(x+z)^2-t^2} \bigg) z^{\a+1/2} f(z)\, dz, \\
V_{t,1}^{\a,\b}f(x) & = \frac{\chi_{\{x<t\}}}{t^{2\a+2\b}} \int_0^{t-x} \big[ t^2-(x-z)^2\big]^{-\a-1/2}
	\big[t^2-(x+z)^2\big]^{\a+\b-1/2} z^{2\a+1} f(z)\, dz, \\
V_{t,2}^{\a,\b}f(x) & = \frac{\chi_{\{x<t\}}}{t^{2\a+2\b}} \int_0^{t-x} \big[t^2-(x-z)^2\big]^{\b-1} z^{2\a+1}f(z)\, dz, \\
\widetilde{V}_{t,2}^{\a,\b}f(x) & = \frac{\chi_{\{x<t\}}}{t^{2\a+2\b}} \int_0^{t-x} \big[t^2-(x-z)^2\big]^{\b-1}
	\log\bigg(2 \frac{t^2-(x-z)^2}{t^2-(x+z)^2} \bigg)z^{2\a+1}f(z)\, dz.
\end{align*}
Observe that for $\a+\b=1/2$ one has $U_{t,1}^{\a,\b} = U_{t,2}^{\a,\b}$, $V_{t,1}^{\a,\b} = V_{t,2}^{\a,\b}$, and 
\begin{equation} \label{domlog}
U_{t,2}^{\a,\b}f(x) \lesssim \widetilde{U}_{t,2}^{\a,\b}f(x), \qquad
V_{t,2}^{\a,\b}f(x) \lesssim \widetilde{V}_{t,2}^{\a,\b}f(x), \qquad f \ge 0, \quad x,t > 0.
\end{equation}

The six $U$ and $V$ operators correspond to the absolute bounds of the kernel $K_t^{\a,\b}(x,z)$ in \cite[Theorem 2.1]{CNR2}.
This kernel is in general supported in $E \cup F$, but for $-\b \in \mathbb{N}$ only in $E$.
The $U$ operators are related to $E$, while $V$ to $F$.
$U_{t,1}^{\a,\b}$ matches the cases when $-\b \in \mathbb{N}$ or $[\a+\b > 1/2 \; \textrm{and}\; 2\a+\b=0]$ or
$[\a+\b < 1/2 \;\textrm{and}\; -\b \notin \mathbb{N} \; \textrm{and}\; \a+1/2 \notin \mathbb{N}]$.
$U_{t,2}^{\a,\b}$ matches $[\a+\b > 1/2 \;\textrm{and}\; -\b \notin \mathbb{N} \;\textrm{and}\; 2\a+\b \neq 0]$
or $(\a,\b) = (-1/2,1)$ or $[\a+\b < 1/2 \;\textrm{and}\; \a+1/2 \in \mathbb{N}]$.
For $V_{t,i}^{\a,\b}$, $i=1,2$, we take into account $-\b \notin \mathbb{N}$ and then $V_{t,1}^{\a,\b}$ matches
$[\a + \b > 1/2 \;\textrm{and}\; 2\a+\b=0]$ or $[\a+\b < 1/2 \; \textrm{and}\; \b \neq 1]$, while $V_{t,2}^{\a,\b}$
matches $[\a+\b > 1/2 \; \textrm{and}\; 2\a+\b \neq 0]$ or $(\a,\b)=(-1/2,1)$ or $[\a+\b < 1/2 \;\textrm{and}\; \b=1]$.
Finally, $\widetilde{U}_{t,2}^{\a,\b}$ and $\widetilde{V}_{t,2}^{\a,\b}$ match the logarithmic case
$[\a+\b = 1/2 \;\textrm{and}\; \b \notin \mathbb{Z}]$.

Let $\varepsilon > 0$ and denote
\begin{align*}
E_{\varepsilon} & = \big\{ (t,x,z) \in E : t - |x-z| < \varepsilon \sqrt{xz}\;\;
	\textrm{or} \;\; x+z -t < \varepsilon \sqrt{xz} \big\}, \\
F_{\varepsilon} & = \big\{ (t,x,z) \in F : t - (x+z) < \varepsilon \sqrt{xz}\;\;
	\textrm{or} \;\; t >  \varepsilon^{-1} \sqrt{xz} \big\}.
\end{align*}
As explained in \cite[Section 4.2]{CNR2}, for each pair $\a,\b$ under consideration there exists
$\varepsilon = \varepsilon(\a,\b) > 0$ such that the kernel $K_t^{\a,\b}(x,z)$ does not change sign
in $E_{\varepsilon}$ and $F_{\varepsilon}$ and, moreover, $|K_t^{\a,\b}(x,z)|$ is comparable in
$E_{\varepsilon}$ and in $F_{\varepsilon}$ with the kernel of the corresponding operator $U$ and $V$.

Therefore, proving the unboundedness results for $M^{\a,\b}_{*}$ stated in Proposition \ref{prop:necBC} can
be transmitted to showing suitable unboundedness results for the maximal operators corresponding to the $U$ and $V$
operators, stated in Lemma \ref{lem:C} below,
provided that we assure that all triples $(t,x,z)$ involved in the counterexamples are located
either in $E_{\varepsilon}$ or in $F_{\varepsilon}$. This is indeed the case, as commented after the proof of Lemma \ref{lem:C}.
Note that, because of \eqref{domlog}, any unboundedness result in Lemma \ref{lem:C} for $U_{*,2}^{\a,\b}$ implies an analogous
result for its tilded variant, and analogously in case of the $V$ operators.

Denote $U_{*,1}^{\a,\b}f = \sup_{t>0}|U_{t,1}^{\a,\b}f|$ and similarly for the other $U$ and $V$ operators.
\begin{lema} \label{lem:C}
Let $\a > -1$, $\a + \b > -1/2$, $1 \le p < \infty$ and $\delta \in \mathbb{R}$.
Then the $U$ and $V$ operators possess the following unboundedness properties with respect to $(\mathbb{R}_+,x^{\delta}dx)$.
\begin{itemize}
\item[(a)]
$U_{*,i}^{\a,\b}$, $i=1,2$, are not of strong type $(1,1)$.
\item[(b1)]
$U_{*,i}^{\a,\b}$, $i=1,2$, are not of weak type $(p,p)$ when $1/p = \a+\b + 1/2 < 1$.
\item[(b2)]
$U_{*,i}^{\a,\b}$, $i=1,2$, are not of weak type $(p,p)$ when $p > 1$ and $\delta = (2\a+\b+1)p-1$.
\item[(b3)]
$V_{*,i}^{\a,\b}$, $i=1,2$, are not of weak type $(p,p)$ when $p > 1$ and $\delta = (2\a+2)p-1$.
\item[(c1)]
$U_{*,i}^{\a,\b}$, $i=1,2$, are not of restricted weak type $(p,p)$ when $1/p > \a + \b + 1/2$.
\item[(c2)]
$U_{*,i}^{\a,\b}$, $i=1,2$, are not of restricted weak type $(p,p)$ when $\delta < -\b p$.
\item[(c3)]
$U_{*,i}^{\a,\b}$, $i=1,2$, are not of restricted weak type $(p,p)$ when $\delta > (2\a+\b+1)p-1$.
\item[(c4)]
$\widetilde{U}_{*,2}^{\a,\b}$ is not of restricted weak type $(1,1)$ when $\a+\b =1/2$.
\item[(c5)]
$V_{*,i}^{\a,\b}$, $i=1,2$, are not of restricted weak type $(p,p)$ when $\delta \le -1$.
\item[(c6)]
$V_{*,i}^{\a,\b}$, $i=1,2$, are not of restricted weak type $(p,p)$ when $\delta > (2\a+2)p-1$.
\end{itemize}
\end{lema} 
 
\begin{proof}
We shall present a counterexample for each item of the lemma.

\medskip
\noindent{\textbf{Item (a)}.}
Consider $f(z) = \ind{(1,1+\xi)}(z) (z-1)^{-1} ( \log\frac{2}{z-1})^{-2}$
and $x \in (1,1+\xi)$, for a fixed small $\xi > 0$. Notice that $f \in L^1(x^{\delta}dx)$ for any $\delta \in \mathbb{R}$.
Assume for the time being that $t < x/2$. Then we have (see \cite[(3.2)]{CNR2})
\begin{align*}
U_{t,1}^{\a,\b}f(x) & = \frac{x^{-2\a-\b}}{t^{2\a+2\b}} \int_{x-t}^{x+t}
	\Big( \big[ t^2-(x-z)^2\big] \big[ (x+z)^2-t^2\big] \Big)^{\a+\b-1/2} z^{1-\b} f(z)\, dz \\
& \gtrsim \frac{x^{-2\a-\b}}{t^{2\a+2\b}} \int_{x-t/2}^{x+t/2} t^{2\a+2\b-1} x^{2\a + 2\b -1} x^{1-\b} f(z)\, dz \\
& = \frac{1}t \int_{x-t/2}^{x+t/2} f(z)\, dz.
\end{align*}
In an analogous way we get the same lower bound for $U_{t,2}^{\a,\b}f(x)$.
Choosing now $t=2(x-1)$, we see that for $i=1,2$
\begin{align*}
U_{*,i}^{\a,\b}f(x) & \gtrsim \frac{1}{x-1} \int_1^{2x-1} (z-1)^{-1} \Big( \log\frac{2}{z-1}\Big)^{-2}\, dz \\
& \simeq \frac{1}{x-1} \Big( \log\frac{1}{x-1}\Big)^{-1}, \qquad x \in (1,1+\xi).
\end{align*}
But the last expression has a non-integrable singularity at $1^{+}$, so $U_{*,i}^{\a,\b}f \notin L^1(x^{\delta}dx)$
for $i=1,2$ and any $\delta \in \mathbb{R}$. The conclusion follows.

\medskip
\noindent{\textbf{Item (b1)}.}
We pick $f(z) = \ind{(1,2)}(z) (z-1)^{-\a-\b-1/2}\slash \log\frac{2}{z-1}$, see 
the proof of \cite[Lemma 4.1(a2)]{CNR2}. For $\a,\b$ and $p$ in question,
this function belongs to $L^p(x^{\delta}dx)$. However, taking $t=x-1$, for large $x$ we get
$$
U_{x-1,1}^{\a,\b} f(x)
	\gtrsim \frac{1}{x^{2\a+\b+1}} \int_1^{2} {(z-1)^{-1}\Big(\log\frac{2}{z-1}\Big)^{-1}}\, dz = \infty
$$
and similarly $U_{x-1,2}^{\a,\b} f(x) = \infty$.

\medskip
\noindent{\textbf{Item (b2)}.}
Define 
$f(z) = \ind{(0,1)}(z) z^{-2\a-\b-1}\slash \log\frac{2}{z}$ and observe that $f \in L^p(x^{\delta}dx)$
for $\a,\b$ and $p$ under consideration. 
On the other hand, taking $t=x$ and considering $x$ sufficiently large, we obtain
\begin{align*} 
U_{x,1}^{\a,\b} f(x)
\gtrsim
x^{-2\a-\b-1} \int_0^1 z^{-1} \bigg(\log\frac{2}{z} \bigg)^{-1} \, dz
= 
\infty
\end{align*}
and similarly $U_{x,2}^{\a,\b} f(x) = \infty$.

\medskip
\noindent{\textbf{Item (b3)}.}
Take $f(z)=\ind{(0,1)}(z)z^{-(2\alpha+2)}\slash \log\frac{2}{z} $, see the proof of \cite[Lemma 4.1(b1)]{CNR2}.
This function belongs to $L^p(x^{\delta}\,dx)$ for the considered $\a,\b$ and $p$.
Let $t=2x$ and $x$ be large.
We have the bound
$$
V_{2x,1}^{\a,\b}f(x) \gtrsim  \frac{1}{x^{2\a+2}} \int_0^{x/2} z^{2\a+1} f(z)\, dz
\ge \frac{1}{x^{2\a+2}} \int_0^1 z^{-1} \Big( \log\frac{2}z \Big)^{-1}\, dz = \infty, 
$$
and similarly we obtain $V_{2x,2}^{\a,\b}f(x) = \infty$.

\medskip
\noindent{\textbf{Item (c1)}.}
Taking $f_N(z) = \ind{(N-1,N+1)}(z)$ we see that for large $N$, $x \in [N/4,3N/4]$ and either
$t = x+N-1$ or $t = N+1-x$, we have
\begin{align*} 
U_{x+N-1,1}^{\a,\b} f_N(x)
& \gtrsim
N^{-\a-\beta-1/2} \int_{N-1}^{N+1} (z-N+1)^{\alpha + \beta -1/2} \, dz
\simeq
N^{-\a-\b -1/2}, \\
U_{N+1-x,2}^{\a,\b} f_N(x)
& \gtrsim
N^{-\a-\b -1/2} \int_{N-1}^{N+1} (N+1-z)^{\a+\b-1/2}
	\, dz
\simeq
N^{-\a-\b -1/2}.
\end{align*}
Notice that restricted weak type $(p,p)$ of $U_{*,i}^{\a,\b}$, $i=1, 2$, implies that 
\begin{align*} 
\sup_{\lambda > 0} \lambda 
\bigg( \int_{\{ N/4 < x < 3N/4\, :\, N^{-\alpha-\beta - 1/2 } > \lambda\}} 
x^{\delta} \, dx \bigg)^{1/p}
\lesssim
N^{\delta/p},
\end{align*}
uniformly in large $N$.
This forces
\begin{align*} 
N^{-\alpha -\beta - 1/2} N^{(\delta+1)/p}
\lesssim
N^{\delta/p}, \qquad N \; \textrm{large},
\end{align*}
which means that $1/p - (\a+\b+1/2) \le 0$.
The conclusion follows.

\medskip
\noindent{\textbf{Item (c2)}.}
We proceed similarly as in the proof of \cite[Lemma 4.1(a3)]{CNR2}.
Here instead of 
$f_N(z) = \ind{(N-1,N+1)}(z) z^{\b}$ we take 
$f_N(z) = \ind{(N-1,N+1)}(z)$
for $N$ large.
Letting $t=N$, for $x \in (1,2)$ we have $U_{*,i}^{\a,\b}f_N(x) \gtrsim N^{-\b}$, $i=1,2$.
Since $\|f_N\|_{L^p(x^{\delta}dx)} \simeq N^{\delta/p}$, the restricted weak type $(p,p)$ of
$U^{\a,\b}_{*,i}$, $i=1,2$, implies
$$
\sup_{\lambda > 0} \lambda \bigg( \int_{\{1 < x < 2 \,:\, N^{-\b} > \lambda\}} x^{\delta} \, dx\bigg)^{1/p} \lesssim N^{\delta/p},
$$
which gives $N^{-\beta} \lesssim N^{\delta /p}$ for large $N$. This means that necessarily $-\b p \le \delta$.

\medskip
\noindent{\textbf{Item (c3)}.}
Taking $f(z) = \ind{(1,2)}(z)$ and considering $t = x-1$ and large $x$, say $x > C$, we get
(as in the proof of \cite[Lemma 4.1(a1)]{CNR2}) $U_{x-1,1}^{\a,\b}f(x) \gtrsim x^{-2\a-\b-1}$ and the same bound for
$U_{x-1,2}^{\a,\b}f(x)$. Therefore, the restricted weak type $(p,p)$ of $U_{*,i}^{\a,\b}$, $i=1,2$, implies
\begin{equation} \label{id:711}
\sup_{\lambda > 0} \lambda 
\bigg( \int_{\{ x > C \,:\, x^{-2\a-\b-1} > \lambda\}} 
x^{\delta} \, dx \bigg)^{1/p}
\lesssim
1.
\end{equation}
We claim that this forces $\delta \le (2\a+\b+1)p-1$, which is what we need.
Indeed, if $2\a+\b+1 > 0$ then \eqref{id:711} implies
$\lambda^{1 -(\delta + 1)/[(2\alpha + \beta + 1)p]} \lesssim 1$ for small $\lambda$, which gives the conclusion.
When $2\a+\b + 1 < 0$ we consider $\lambda$ large in \eqref{id:711} and obtain
$\lambda^{ 1+ (\delta + 1)/[(-2\a-\b-1)p]} \lesssim 1$, and the conclusion again follows.
Finally, if $2\a + \b +1 = 0$ then \eqref{id:711} with $\lambda = 1/2$ shows that
$\int_C^{\infty} x^{\delta}dx < \infty$, i.e.\ $\delta < -1 = (2\a+\b + 1)p-1$.

\medskip
\noindent{\textbf{Item (c4)}.}
Recall that now $\a+\b = 1/2$.
Fix a small $\xi \in (0,1)$ and let
$f_{\xi}(z) = \ind{(1,1+\xi)}(z) (z-1)^{-1}
\big( \log\frac{2}{z-1} \big)^{-2}$. This function belongs to $L^1(x^{\delta}dx)$ for every $\delta \in \RR$.
However, taking $t=x+1$, for $x \in (1/2,3/4)$ we get
\begin{align*} 
\widetilde{U}_{x+1,2}^{\a,\b}f_{\xi}(x)
& \gtrsim 
\int_{1}^{1+\xi} 
	\log\bigg(\frac{8xz}{(x+z)^2-(x+1)^2} \bigg) 
(z-1)^{-1} \Big( \log\frac{2}{z-1} \Big)^{-2} \, dz \\
& \simeq
\int_{1}^{1+\xi} 
(z-1)^{-1} \Big( \log\frac{2}{z-1} \Big)^{-1} \, dz
= \infty.
\end{align*}
Since weak and restricted weak type $(1,1)$ coincide, we get the desired conclusion.

\medskip
\noindent{\textbf{Item (c5)}.}
Consider $f(z)=\ind{(1,3/2)}(z)$ and note that $f \in L^p(x^{\delta}dx)$ for any $\delta$.
Fix a small $0 < \xi < 1/2$. Choosing $t=2$, for $x \in (0,\xi)$ we get
\begin{align*}
V_{2,1}^{\a,\b}f(x) & \simeq \int_0^{2-x} \big[4-(x-z)^2\big]^{-\a-1/2} \big[ 4 - (x+z)^2\big]^{\a+\b-1/2} z^{2\a+1}
	\ind{(1,3/2)}(z)\, dz \\
	& \gtrsim \int_1^{5/4} \big[4-(x-z)^2\big]^{-\a-1/2} \big[ 4 - (x+z)^2\big]^{\a+\b-1/2} z^{2\a+1}\, dz 
	\simeq 1,\\
V_{2,2}^{\a,\b}f(x) & \simeq \int_0^{2-x} \big[ 4-(x-z)^2\big]^{\b-1} z^{2\a+1} \ind{(1,3/2)}(z)\, dz 
	\gtrsim \int_1^{5/4} z^{2\a+1}\, dz \simeq 1.
\end{align*}
Therefore, the restricted weak type $(p,p)$ of $V_{*,i}^{\a,\b}$, $i=1,2$, implies
\begin{equation} \label{id:V21}
\sup_{\lambda > 0} \lambda 
\bigg( \int_{\{ 0<x <\xi \,:\, 1 > \lambda\}} 
x^{\delta} \, dx \bigg)^{1/p}
\lesssim
1.
\end{equation}
With $\lambda=1/2$, this leads to 
$\int_0^{\xi}x^{\delta}\,dx<\infty$, i.e. $\delta > -1$.

\medskip
\noindent{\textbf{Item (c6)}.}
Take $f(z)=\ind{(1,2)}(z)$. Choosing $t=2x$ and considering $x$ large, say $x>C$, we have the bound
\begin{align*}
V_{2x,1}^{\a,\b}f(x)& \simeq \frac{1}{x^{2\a+2\b}} \int_0^{x} \big[ 4x^2-(x-z)^2\big]^{-\a-1/2}
	\big[ 4x^2- (x+z)^2\big]^{\a+\b-1/2} z^{2\a+1} \ind{(1,2)}(z)\, dz \\
& \gtrsim \frac{1}{x^{2\a+2}} \int_1^{2} z^{2\a+1}\, dz \gtrsim  \frac{1}{x^{2\a+2}}.
\end{align*}
As easily verified, the same bound holds for $V_{2x,2}^{\a,\b}f(x)$.
Thus the restricted weak type $(p,p)$ for $V_{*,i}^{\a,\b}$, $i=1,2$, implies
\begin{equation} 
\label{id:V1b}
\sup_{\lambda > 0} \lambda 
\bigg( \int_{\{  x>C \,:\,  x^{-2\alpha - 2} > \lambda\}} 
x^{\delta} \, dx \bigg)^{1/p}
\lesssim
1.
\end{equation}
Since $2\alpha+2>0$, taking into account small $\lambda$ we see that \eqref{id:V1b} leads to
$\lambda^{1 -(\delta + 1)/[(2\alpha + 2)p]} \lesssim 1$.
Consequently, necessarily $\delta \le (2\a+2)p-1$.

\medskip
The proof of Lemma \ref{lem:C} is complete.
\end{proof}

Observe that the triples $(t,x,z)$ involved in the counterexamples in the proof of Lemma \ref{lem:C}
are located either in $E_{\varepsilon}$ or in $F_{\varepsilon}$. More precisely, the triples from
Items (b1), (b2) and (c1)--(c4) are in $E_{\varepsilon}$ provided that we consider $x$ sufficiently large
(Items (b1), (b2) and (c3)) or $N$ large enough (Items (c1) and (c2)) or $\xi > 0$ sufficiently small (Item(c4)).
The triples from Items (b3), (c5) and (c6) are in $F_{\varepsilon}$ when one considers $x$ sufficiently large
(Items (b3) and (c6)) or $\xi > 0$ small enough (Item (c5)).
Finally, concerning Item (a), with $\xi>0$ chosen sufficiently small, the involved $(t,x,z)$ lie either in
$E_{\varepsilon}$, or in $\{(t,x,z) \in \RR_+^3: t \le |x-z|\}$ where
the kernel $K_t^{\a,\b}(x,z)$ vanishes a.e., so overall in a region where
the kernel does not change its sign, which is also fine for our purpose.

\section{Proofs of Lemmas \ref{lem:N2}--\ref{lem:Rlog}} \label{sec:auxTS}

In this section we give proofs of the mapping properties of the auxiliary maximal operators
$N_{\eta}$, $R_{\eta}$, $E_{k,\eta}$, $T_{\eta}$, $S_{\a,\b}$, $S_{\a,\b}^{\log}$ and $R_{\a,\b}^{\log}$
stated in Lemmas \ref{lem:N2}--\ref{lem:Rlog}.

\subsection{Proof of the mapping properties of $N_{\eta}$} \label{ssec:N}

\begin{proof}[{Proof of Lemma \ref{lem:N2}}]
Observe that for $\eta < 0$ we have $N_\eta f (x) = \infty$, $x>0$, for every $f \ne 0$;
in particular, there is no restricted weak type $(p,p)$.
If $\eta = 0$, then $N_\eta f (x) = \int_0^\infty z^{-1} \abs{f(z)} \, dz$ and
it is straightforward to see that, for any $\gamma \in \mathbb{R}$ and $1 \le p < \infty$,
$N_0$ is not of restricted weak type $(p,p)$ .
Therefore, from now on, we assume that $\eta > 0$.

We have
\begin{align*}
N_{\eta}f(x) \le H_{\eta}|f|(x) + H_0^{\infty}|f|(x),\qquad x > 0,
\end{align*}
since for $t > x$ we can write 
\begin{align*}
\frac{1}{t^{\eta}} \int_0^{t} z^{\eta -1} |f(z)|\, dz
& \le \frac{1}{x^{\eta}} \int_0^{x} z^{\eta-1} |f(z)|\, dz + \frac{1}{t^{\eta}} \int_x^t z^{\eta-1}|f(z)|\, dz \\
& \le \frac{1}{x^{\eta}} \int_0^{x} z^{\eta-1} |f(z)|\, dz + \int_x^{\infty} \frac{|f(z)|}{z}\, dz
	= H_{\eta}|f|(x) + H_0^{\infty}|f|(x).
\end{align*}
On the other hand, letting $t \to x^{+}$,
we see that
$$
N_{\eta}f(x) \ge H_{\eta}|f|(x), \qquad x > 0.
$$
Combining these two estimates with Lemmas~\ref{lem:H0} and \ref{lem:Hinf},
we see that in order to finish the proof of Lemma~\ref{lem:N2} it suffices to show that
for $\gamma \le -1$ and any $1 \le p < \infty$, the operator $N_{\eta}$ is not of restricted weak type $(p,p)$
with respect to the measure $x^\gamma dx$. 

Observe that $\ind{[1,2]} \in L^p(x^\gamma dx)$ for every $\gamma \in \mathbb{R}$.
Further,
\begin{align*}
N_{\eta} \ind{[1,2]} (x) 
\simeq
\sup_{t>x} t^{-\eta} \int_1^{t\wedge 2} z^{\eta -1} \, dz
\simeq 1, \qquad x \in (0,1).
\end{align*}
Therefore, the restricted weak type $(p,p)$ implies
\begin{align*}
\int_0^1 x^{\gamma} \, dx < \infty,
\end{align*}
which means that one must have $\gamma > -1$. This finishes the proof of Lemma~\ref{lem:N2}.
\end{proof}

\subsection{Proof of the mapping properties of $R_\eta$} \label{ssec:R}

We first state a simple but useful observation. Its verification is straightforward and left to the reader.
\begin{lema} \label{lem:chvar}
Let $\lambda \in \mathbb{R}$ be fixed and let $\mathcal{U},\mathcal{V}$ be operators
acting on functions on $\mathbb{R}_+$ and related by the formula
\begin{align} \label{chvar} 
\mathcal{V}f (x) = x^{\lambda} \mathcal{U}\big( (\cdot)^{-\lambda} f \big) (x), \qquad x > 0.
\end{align}
Then, given any $\zeta \in \mathbb{R}$ and $1\le p < \infty$, 
\begin{equation*}
\text{$\mathcal{V}$ is bounded on $L^p(\mathbb{R}_+, x^\zeta dx)$ \; if and only if \;
$\mathcal{U}$ is bounded on $L^p(\mathbb{R}_+, x^{\zeta+\lambda p} dx)$. }
\end{equation*}
\end{lema}

\begin{proof}[Proof of Lemma \ref{lem:R2}]
We split the reasoning into three steps.

\medskip
\noindent \textbf{Step 1.}
We show that for $\gamma < -\eta p$ and any $1 \le p < \infty$, $R_{\eta}$ is not of restricted weak type $(p,p)$.

For $N > 1$ let $f_N = \ind{[N,N+1]}$. Then we have 
\begin{align*} 
\| f_N \|_{L^p(x^\gamma dx)} \simeq N^{\gamma/p}, 
\qquad N > 1.
\end{align*}
Further, observe that taking $t = N+1$ in the expression defining $R_{\eta}$ we get
\begin{align*}
R_{\eta} f_N (x)
\gtrsim
x^{\eta - 1} \int_{N + 1 - x}^{N + 1 } z^{-\eta} \, dz
\simeq 
x^{\eta }  N^{-\eta}, \qquad N > 1, \quad x \in (0,1).
\end{align*}
Therefore, the restricted weak type $(p,p)$ of $R_{\eta}$ implies
\begin{align*} 
\sup_{\lambda > 0} \lambda \bigg( 
\int_{ \{ x < 1 \,:\, x^\eta N^{-\eta} \ge \lambda \} } x^\gamma \, dx
\bigg)^{1/p}
\lesssim
N^{\gamma/p}, \qquad N > 1.
\end{align*}
Observe that the above specified to $\lambda = (2N)^{-\eta}$ forces
(notice that the constraint $x^\eta N^{-\eta} \ge \lambda$ is equivalent to $x\ge 1/2$ if $\eta > 0$ and $x\le 1/2$ if $\eta < 0$) 
\begin{align*} 
N^{-\eta}
\lesssim
N^{\gamma/p}, \qquad N > 1.
\end{align*}
This, in turn, implies $-\eta \le \gamma/p$, which gives the desired conclusion.

\medskip
\noindent \textbf{Step 2.}
We prove part (a) of the lemma.

In view of
Lemma \ref{lem:chvar}, it suffices to consider only the case $\eta = 0$. Therefore, denoting for brevity $R:=R_0$ and
taking into account Step 1, to conclude Step 2 it is enough to prove the following three statements.
\begin{align} \label{id:31}
\text{For $\gamma \ge 0$ and $p>1$ the operator $R$ is bounded on $L^p(x^\gamma dx)$. } \\ \label{id:32}
\text{For $\gamma > 0$ the operator $R$ is bounded on $L^1(x^\gamma dx)$. } \\ \label{id:33}
\text{$R$ is not bounded on $L^1(dx)$. } 
\end{align}

We first deal with \eqref{id:31}. Let $\gamma \ge 0$ and observe that 
\begin{align*} 
Rf (x)
\simeq
x^{-1} \sup_{t > 2x} t^{-\gamma} \int_{t-x}^{t+x} z^{\gamma}\abs{f(z)} \, dz
\lesssim
x^{-\gamma - 1} \| f \|_{L^1 (x^\gamma dx)}.
\end{align*}
It follows that $R$ is of weak type $(1,1)$ with respect to the measure $x^{\gamma}dx$.
Interpolating now with the trivial $L^\infty$-boundedness of $R$ we get \eqref{id:31}.

Next, we prove \eqref{id:32}. Let $0 \le \lambda < \gamma$ be fixed. Then
\begin{align*} 
Rf (x)
\simeq 
x^{-1} \sup_{t > 2x} t^{-\lambda} \int_{t-x}^{t+x} z^{\lambda} \abs{f(z)} \, dz
\lesssim
x^{-\lambda - 1} \int_{x}^{\infty} z^{\lambda} \abs{f(z)} \, dz.
\end{align*}
Consequently, we obtain
\begin{align*} 
\|Rf \|_{L^1(x^\gamma dx)}
\lesssim
\int_0^\infty z^\lambda 
 \abs{f(z)} \int_0^z x^{-\lambda - 1 + \gamma} \, dx \, dz
\simeq
\|f \|_{L^1(x^\gamma dx)}.
\end{align*}
This gives \eqref{id:32}.

Finally, we show \eqref{id:33}.
Invoking $f_N$ from Step 1 and
taking $t=x+N$ in the expression defining $R$ we get
\begin{align*} 
R f_N (x) 
\gtrsim
x^{-1} \int_N^{2x + N} f_N(z) \, dz 
= 
x^{-1}, \qquad 1/2 < x < N, \quad N > 1.
\end{align*}
Therefore,
$L^1(dx)$-boundedness of $R$ would, in particular, imply
\begin{align*} 
\log N \simeq \int_{1/2}^N x^{-1} \, dx 
\lesssim
\int_{N}^{N+1} \, dx = 1, \qquad N > 1,
\end{align*}
which is not true for large $N$.
Thus \eqref{id:33} is proved, and this finishes Step 2.

\medskip
\noindent \textbf{Step 3.}
Finally, we show that for $p=1$ and $\gamma = - \eta p = - \eta$ the operator $R_\eta$ is of weak type $(1,1)$
with respect to $x^\gamma dx$
if and only if $\eta \ne 1$. Recall that weak and restricted weak types $(1,1)$ are equivalent.

We first show the positive part of the above statement.
Observe that
\begin{align*} 
R_\eta f (x)
\lesssim
x^{\eta - 1} \| f \|_{L^1 (x^{-\eta} dx)}.
\end{align*}
If $\eta < 1$, then we further get
\begin{align*} 
\int_{\big\{x \,:\, x^{\eta - 1} \| f \|_{L^1 (x^{-\eta} dx)} > \lambda \big\}} 
x^{-\eta} \, dx
=
\int_{\big\{x \,:\, x< \big(\lambda^{-1} \| f \|_{L^1 (x^{-\eta} dx)} \big)^{1/(1-\eta)} \big\}} 
x^{-\eta} \, dx
\simeq
\frac{ \| f \|_{L^1 (x^{-\eta} dx)}}{\lambda},
\end{align*}
uniformly in $\lambda > 0$. Similarly, if $\eta > 1$, then, again uniformly in $\lambda > 0$,
\begin{align*} 
\int_{\big\{x \,:\, x^{\eta - 1} \| f \|_{L^1 (x^{-\eta} dx)} > \lambda \big\}} 
x^{-\eta} \, dx
=
\int_{\big\{x \,:\, x > \big(\lambda \| f \|_{L^1 (x^{-\eta} dx)}^{-1} \big)^{1/(\eta - 1)} \big\}} 
x^{-\eta} \, dx
\simeq
\frac{ \| f \|_{L^1 (x^{-\eta} dx)}}{\lambda}.
\end{align*}

It remains to show the negative result when $\eta = 1$.
Let $0< b/2 < a < b$
and observe that uniformly in $(b-a)/2 < x < b/3$ we have (take $t= b - x \simeq b$ in the expression for $R_{\eta}$)
\begin{align*}
R_{1} \ind{[a,b]} (x)
\gtrsim
(b-a)/b, \qquad b/2 < a < b.
\end{align*}
Therefore, assuming the weak type $(1,1)$ of $R_1$ with respect to $x^{-1} dx$, we have 
\begin{align*} 
\sup_{\lambda > 0} \lambda \int_{\{ (b-a)/2 < x < b/3 \,:\, (b-a)/b > \lambda\}}
x^{-1} \, dx
\lesssim
\int_a^b x^{-1} \, dx
=
\log \frac{b}a 
\simeq 
\frac{b-a}a \simeq 
\frac{b-a}b,
\end{align*}
uniformly in $b/2 < a < b$.
This implies
\begin{align*} 
\int_{(b-a)/2 }^{b/3}
x^{-1} \, dx
\lesssim
1, 
\qquad b/2 < a < b.
\end{align*}
But taking e.g.\ $b=1$ and letting $a \to 1^{-}$ we get a contradiction. This concludes Step 3.

The proof of Lemma~\ref{lem:R2} is complete.
\end{proof}

\subsection{Proof of the mapping properties of $E_{k, \eta}$} \label{ssec:E}

\begin{proof}[Proof of Lemma \ref{lem:E2}]
As in \cite{DMO} we denote $\delta_k f(x) = f(x^{1/k})$. It is easy to see that 
$$
k E_{k, \eta} f (x)
=
\delta_{1/k} E_{1, \eta/k} \delta_k f (x),
\qquad x > 0.
$$
Then
the proof of Lemma~\ref{lem:E2} is reduced to the special case $k=1$ related to the Hardy-Littlewood maximal operator. 
More precisely, using the above identity one can show that 
$E_{k, \eta}$ is of strong/weak/restricted weak type $(p,p)$ with respect to the measure $x^{\gamma} dx$
if and only if 
$E_{1, \eta/k}$ is of strong/weak/restricted weak type $(p,p)$ with respect to the measure $x^{(\gamma+1)/k -1} dx$.
So from now on we assume that $k=1$. 

Using Lemma \ref{lem:chvar}, we see that proving item (a) of the lemma is reduced to the special case $\eta = 0$.
This, however, is contained in \cite[Lemma 2.1(a)]{DMO}. Unfortunately, there is no analogue of Lemma~\ref{lem:chvar}
for the weak/restricted weak type
(see, for instance, Lemma~\ref{lem:R2} and compare the result for $\eta = 0$ and $\eta = 1$ when $p=1$). 

Considering the case $\eta = 0$, the operator $E_{1, 0}$ is the classical Hardy-Littlewood maximal operator.
Thus the condition characterizing the weak type $(p,p)$ is just the $A_p$ condition, see e.g.\ \cite[Chapter 7]{Duo}.
The restricted weak type inequality from part (c) is also contained in \cite[Lemma 2.1(a)]{DMO}.

Next, we split $E_{1,\eta}$ as follows:
\begin{align*}
E_{1, \eta} f (x)
& \simeq
x^{\eta} 
\sup_{\substack{ 0\le a <x< b \\ a>2b/3}} \frac{1}{b-a} \int_{a}^{b} z^{-\eta} \abs{f(z)} \, dz
+
x^{\eta} \sup_{\substack{ 0\le a <x< b \\ a \le 2b/3}} \frac{1}{b} \int_{a}^{b} z^{-\eta} \abs{f(z)} \, dz \\
& \simeq
\sup_{\substack{ 0\le a <x< b \\ a>2b/3}} \frac{1}{b-a} \int_{a}^{b} \abs{f(z)} \, dz
+
x^{\eta} \sup_{ x< b } \frac{1}{b} \int_{0}^{b} z^{-\eta} \abs{f(z)} \, dz \\
& =:
B_1 f (x) + B_2 f (x).
\end{align*}
The last estimate holds because $a\simeq b \simeq x$ in case of the first operator,
and for the second one the choice of $a=0$ in the supremum is optimal.

Observe that the operator $B_1$ is comparable with $L$.
Therefore, taking into account Lemma \ref{lem:L}, it suffices to consider $B_2$. We have
\begin{align*}
B_2f(x)&\simeq x^{\eta}\bigg(\sup_{ x< b } \frac{1}{b} \int_{0}^{x} z^{-\eta} \abs{f(z)} \, dz +\sup_{ x< b }
	\frac{1}{b} \int_{x}^{b} z^{-\eta} \abs{f(z)} \, dz\bigg)\\
& \lesssim x^{\eta-1} \int_{0}^{x} z^{-\eta} \abs{f(z)} \, dz + x^{\eta}\int_{x}^{\infty} z^{-\eta-1} \abs{f(z)} \, dz
	= H_{1-\eta}|f|(x)+H^{\infty}_{\eta}|f|(x).
\end{align*}
On the other hand, 
letting $b \to x^{+}$ in the expression defining $B_{2}f(x)$ we get
\begin{align*}
B_2 f(x)\gtrsim H_{1-\eta}|f|(x), \qquad x > 0.
\end{align*}
Combining the above estimates with Lemmas \ref{lem:H0} and \ref{lem:Hinf},
we see that what is left to prove are the following two statements.
\begin{align} \label{eq:12} 
& \text{If $\gamma = -\eta p-1$, $\eta \ne 0$ and $p>1$, then $B_{2}$ is of weak type $(p,p)$.} \\ \label{eq:13} 
& \text{If $\gamma < -\eta p-1$, $\eta \ne 0$ and $p \ge 1$, then $B_2$ is not of restricted weak type $(p,p)$.}
\end{align}

Showing \eqref{eq:12} is straightforward. Notice that 
\begin{align*}
\abs{B_2 f(x)} \le x^{\eta} \| f \|_{L^p(x^\gamma dx)} , \qquad x > 0.
\end{align*}
This leads directly to the required property. 

It remains to show \eqref{eq:13}.
Consider $f = \ind{[1,2]} \in L^p(x^{\gamma} dx)$.
Then, taking $b=2$ in the expression for $B_2$, we get
$$
B_2 \ind{[1,2]} (x)
\gtrsim
x^{\eta} \int_{1}^{2} z^{-\eta} \, dz\simeq x^{\eta}, \qquad x \in (0,1).
$$
Therefore, assuming that $B_2$ is of restricted weak type $(p,p)$ for $\gamma<-\eta p-1$, we have 
\begin{align} \label{eq:11}
\sup_{\lambda > 0} \lambda  \bigg( 
\int_{ \{ x\in (0,1) \,:\, x^\eta \ge \lambda \} } x^\gamma \, dx
\bigg)^{1/p}
< \infty.
\end{align}
If $\eta>0$, since $\gamma <-\eta p-1<-1$, one has
$$
\bigg( 
\int_{\lambda^{1/\eta}}^1 x^\gamma \, dx
\bigg)^{1/p}\simeq \lambda^{\frac{\gamma+1}{\eta p}}, \qquad 0<\lambda<1/2,
$$
and thus \eqref{eq:11} implies
\begin{equation*}
 \lambda^{1+\frac{\gamma+1}{\eta p}}\lesssim 1, \qquad 0< \lambda < 1/2.
\end{equation*}
This forces $\gamma \ge -\eta p-1$, a contradiction. 
If $\eta<0$, \eqref{eq:11} implies
\begin{align*} 
\lambda^{1+\frac{\gamma+1}{\eta p}}
\lesssim
\lambda
\bigg( 
\int_{ \lambda^{1/\eta}/2 }^{ \lambda^{1/\eta} } x^\gamma \, dx
\bigg)^{1/p}
\lesssim
\lambda
\bigg( 
\int_{ \{ x\in (0,1) \,:\, x^\eta \ge \lambda \} } x^\gamma \, dx
\bigg)^{1/p}
\lesssim
1, \qquad \lambda > 0.
\end{align*}
Since now $1+\frac{\gamma+1}{\eta p} > 0$, letting $\lambda \to \infty$ we again end up with a contradiction.
Now \eqref{eq:13} follows.

The proof of Lemma~\ref{lem:E2} is finished.
\end{proof}

\subsection{Proof of the mapping properties of $T_{\eta}$} \label{ssec:Teta}

In order to prove Lemma~\ref{lem:Tchar} we need some preparatory results.
To begin with, we state a simple estimate leaving its verification to the reader.

For $\lambda \in \mathbb{R}$ fixed,
\begin{align} \label{id:1}
\int_{a}^b x^{\lambda} \, dx 
\simeq
\begin{cases}
(b-a) b^{\lambda} & \text{ if } \lambda > - 1,\\
\log(b/a) & \text{ if } \lambda = - 1,\\
(b-a) b^{-1} a^{\lambda + 1} & \text{ if } \lambda < - 1, 
\end{cases}
\qquad 0 < a \le b < \infty.
\end{align}

\begin{lem} \label{lem:F2ref2}
Let $\lambda > -1$ and $\xi \in \mathbb{R}$ be fixed.
Then
\begin{align} \label{id:2ref2}
\int_a^{b} (A + x)^{\xi} x^{\lambda} \, dx 
\simeq
\begin{cases}
(b-a) b^{\lambda} (b+ A)^{\xi} & \text{ if } \xi + \lambda > -1,\\
b^{\lambda} (b+ A)^{-\lambda } \log \Big( 1 + \frac{b-a}{a+A} \Big) & \text{ if } \xi + \lambda = -1,\\
(b-a) b^{\lambda} (a + A)^{\xi + \lambda + 1} (b + A)^{- \lambda  - 1} & \text{ if } \xi + \lambda < -1,
\end{cases}
\end{align}
uniformly in $0 \le a < b$ and $A > 0$.
\end{lem}

\begin{proof}
We distinguish three cases.

\noindent \textbf{Case 1:} $b \le 2A$. Using \eqref{id:1} we get
\begin{align*} 
\texttt{LHS}\eqref{id:2ref2} 
\simeq
A^{\xi} \int_a^b x^\lambda \, dx
\simeq 
(b-a) A^{\xi} b^\lambda
\simeq 
\texttt{RHS}\eqref{id:2ref2}. 
\end{align*}

\noindent \textbf{Case 2:} $a \ge A$. Using again \eqref{id:1} we obtain
\begin{align*} 
\texttt{LHS}\eqref{id:2ref2} 
& \simeq
\int_a^b x^{\xi + \lambda} \, dx
\simeq
\begin{cases}
(b-a)  b^{\xi + \lambda} & \text{ if } \xi + \lambda > -1,\\
\log \Big( 1 + \frac{b-a}{a} \Big) & \text{ if } \xi + \lambda = -1,\\
(b-a) b^{-1} a^{\xi + \lambda + 1} & \text{ if } \xi + \lambda < -1
\end{cases} \\
& \simeq 
\texttt{RHS}\eqref{id:2ref2}. 
\end{align*}

\noindent \textbf{Case 3:} $a \le A$ and $b \ge 2A$.
Here we split the integral in question and use \eqref{id:1} to get
\begin{align*} 
\texttt{LHS}\eqref{id:2ref2} 
& \simeq
\bigg( \int_a^{3A/2} + \int_{3A/2}^b \bigg) (A + x)^{\xi} x^{\lambda} \, dx
\simeq
A^{\xi}  \int_a^{3A/2} x^\lambda \, dx 
+
\int_{3A/2}^b x^{\xi + \lambda} \, dx \\
& \simeq 
\begin{cases}
b^{\xi + \lambda + 1} & \text{ if } \xi + \lambda > -1,\\
\log \Big( 1 + \frac{b}{A} \Big) & \text{ if } \xi + \lambda = -1,\\
A^{\xi + \lambda + 1} & \text{ if } \xi + \lambda < -1
\end{cases} \\
& \simeq 
\texttt{RHS}\eqref{id:2ref2}. 
\end{align*}
This finishes the proof of Lemma~\ref{lem:F2ref2}.
\end{proof}

We are now ready to show the mapping properties of $T_{\eta}$.

\begin{proof}[Proof of Lemma~\ref{lem:Tchar}]
We first focus on the negative part of the lemma. Namely, we will show the following statements.
\begin{align} \label{st:1} 
& \text{$T_{\eta}$ is {\bf{not}} of restricted weak type $(p,p)$ with respect to $x^{\gamma} dx$ if $\gamma \le -1$.} \\ \label{st:2} 
& \text{$T_{\eta}$ is {\bf{not}} of restricted weak type $(p,p)$ with respect to $x^{\gamma} dx$
	if $\gamma < p(\eta-1)$.} \\ \label{st:3}  
& \text{$T_{\eta}$ is {\bf{not}} of strong type $(1,1)$ with respect to $x^{\gamma} dx$ if $\gamma = \eta-1$.}  
\end{align}

To proceed, let $a,b > 0$ be such that $0< b/2 < a < b$ and put
\begin{align*} 
f(x) = \ind{(a,b)} (x), \qquad x > 0,
\end{align*}
so that 
\begin{equation}
\label{eq:f}
\|f\|_{L^p(x^{\gamma}\,dx)}^p \simeq b^{\gamma}(b-a).
\end{equation}
Observe that using \eqref{id:1} for $x \le a/4$ we get
\begin{align*} 
T_\eta f (x) 
& \simeq
\sup_{t>2x} b^{\eta - 1} 
\int_{(t/2) \vee a}^{t\wedge b} (t-z+x)^{-\eta} \, dz \\
& \simeq
b^{\eta - 1} \sup_{2b > t > a}
\begin{cases}
[t\wedge b - (t/2)\vee a] [t - (t/2)\vee a + x]^{-\eta} 
& \text{ if } \eta < 1,\\
\log \Big( \frac{t - (t/2)\vee a + x}{t - t\wedge b + x} \Big) 
& \text{ if } \eta = 1,\\
[t\wedge b - (t/2)\vee a]  [t - (t/2)\vee a + x]^{-1}
[t - t\wedge b + x]^{-\eta + 1} 
& \text{ if } \eta > 1.
\end{cases}
\end{align*}
Note that the restriction in $t$ in the supremum in the last expression above comes from the constraint 
$(t/2) \vee a < t\wedge b$.
Evaluating the above expression under supremum at $t = b$ or $t=2a$ we get
\begin{align} \label{id:3}
T_\eta f (x) 
\gtrsim
\begin{cases}
(b-a) b^{-1} 
& \text{ if } \eta \le 0,\\
(b-a) b^{\eta - 1} (b-a+x)^{-\eta} 
& \text{ if } \eta \in (0,1),\\
\log \Big( 1 + \frac{b - a }{x} \Big) 
& \text{ if } \eta = 1,\\
(b-a) b^{\eta - 1}  (b-a+x)^{-1}
x^{-\eta + 1} 
& \text{ if } \eta > 1,
\end{cases}
\end{align}
uniformly in $0< b/2 < a < b$ and $0 < x \le a/4$.

Choosing now $a=2$ and $b=3$, we see that $T_\eta f (x) \gtrsim 1$ for $x \in (0,1/2)$ and \eqref{st:1} follows.

Next, we focus on \eqref{st:2}. Taking into account \eqref{st:1}, we may assume that $-1 < \gamma < p (\eta - 1)$.
Notice that, in particular, we have $\eta > 1/p'$. To proceed, it is convenient to distinguish three cases.

\noindent \textbf{Case 1:} $\eta > 1/p'$, $\eta \in (0,1)$.  Here we may assume that $\gamma \in (-1, p(\eta - 1) )$.
Observe that, in view of \eqref{id:3} and \eqref{eq:f}, the restricted weak type $(p,p)$ of $T_\eta$ would
lead to the estimate
\begin{align*}
\sup_{\lambda > 0} \lambda^p \int_{\{0 < x \le (b-a)/4 \,:\, (b-a)^{1-\eta} b^{\eta - 1} \ge \lambda \}} 
x^\gamma \, dx
\lesssim
b^\gamma (b-a), \qquad 0< b/2 < a < b. 
\end{align*}
Taking $\lambda = (b-a)^{1-\eta} b^{\eta - 1}$ we see that the above forces 
\begin{align*} 
(b-a)^{\gamma - p(\eta - 1)} 
\lesssim
b^{\gamma - p(\eta - 1)}, \qquad 0< b/2 < a < b. 
\end{align*}
Letting $b=1$ and $a\to 1^{-}$,  this leads to a contradiction since $\gamma - p(\eta - 1) < 0$.

\noindent \textbf{Case 2:} $\eta > 1/p'$, $\eta  = 1$. Here we assume that $\gamma \in (-1,0)$. Taking into account
\eqref{id:3}, we see that the restricted weak type $(p,p)$ of $T_\eta$ would lead to the bound
\begin{align*}
\sup_{\lambda > 0} \lambda^p \int_{\big\{0 < x \le (b-a)/4 \,:\, \log \big( 1 + \frac{b - a }{x} \big) \ge \lambda \big\}} 
x^\gamma \, dx
\lesssim
b^\gamma (b-a), \qquad 0< b/2 < a < b. 
\end{align*}
Taking $\lambda = \log 2$, the above implies
\begin{align*}
(b-a)^\gamma
\lesssim
b^\gamma , \qquad 0< b/2 < a < b. 
\end{align*}
Letting $b=1$ and $a\to 1^{-}$, this leads to a contradiction since $\gamma < 0$.

\noindent \textbf{Case 3:} $\eta > 1/p'$, $\eta  > 1$. Here we assume that $\gamma \in (-1,  p(\eta - 1) )$. From
\eqref{id:3} we see that restricted weak type $(p,p)$ for $T_\eta$ would lead to
\begin{align*}
\sup_{\lambda > 0} \lambda^p \int_{\{0 < x \le (b-a)/4 \,:\, b^{\eta - 1} x^{1-\eta} \ge \lambda \}} 
x^\gamma \, dx
\lesssim
b^\gamma (b-a), \qquad 0< b/2 < a < b. 
\end{align*}
This bound is equivalent to 
\begin{align*}
\sup_{\lambda > 0} \lambda^p \big[ (b-a) \wedge (b \lambda^{-1/(\eta - 1)}) \big]^{\gamma + 1} 
\lesssim
b^\gamma (b-a), \qquad 0< b/2 < a < b. 
\end{align*}
Choosing $\lambda \simeq [b/(b-a)]^{\eta - 1}$, the above forces
\begin{align*} 
(b-a)^{\gamma - p(\eta - 1)} 
\lesssim
b^{\gamma - p(\eta - 1)}, \qquad 0< b/2 < a < b. 
\end{align*}
Taking  $b=1$ and $a\to 1^{-}$, this leads to a contradiction since $\gamma - p(\eta - 1) < 0$.
This finishes showing \eqref{st:2}.

Finally, we deal with \eqref{st:3}. Thanks to \eqref{st:1} we may assume that $\eta > 0$,
which is equivalent to $\gamma > -1$ since $\gamma = \eta - 1$. To proceed, we distinguish similar cases as above.

\noindent \textbf{Case 1:} $\eta \in (0,1)$.
Using \eqref{id:3} we see that $L^1(x^\gamma dx)$-boundedness of $T_\eta$ would
lead to the estimate
\begin{align*}
(b-a) b^{\eta - 1} \int_{0}^{a/4} (b-a+x)^{-\eta} x^\gamma \, dx
\lesssim
b^\gamma (b-a), \qquad 0< b/2 < a < b. 
\end{align*}
By Lemma~\ref{lem:F2ref2} this bound is equivalent to
\begin{align*} 
\log \Big( 1 + \frac{a }{b - a} \Big)
\lesssim
1, \qquad 0< b/2 < a < b. 
\end{align*}
Letting $b=1$ and $a\to 1^{-}$ we get a contradiction.

\noindent \textbf{Case 2:} $\eta  = 1$. Here $\gamma = 0$ and, in view of
\eqref{id:3}, the $L^1(dx)$-boundedness of $T_\eta$ would imply the bound
\begin{align*}
\int_{0}^{a/4} 
\log \Big( 1 + \frac{b - a }{x} \Big)  \, dx
\lesssim
b-a, \qquad 0< b/2 < a < b. 
\end{align*}
Changing the variable $x \mapsto (b-a)y$, we obtain
\begin{align*}
\int_{0}^{a/4(b-a)} 
\log \Big( 1 + \frac{1}{y} \Big)  \, dy
\lesssim
1, \qquad 0< b/2 < a < b. 
\end{align*}
Taking  $b=1$ and letting $a\to 1^{-}$ we are led to a contradiction, since 
$\int_{0}^{\infty} \log ( 1 + \frac{1}{y} )  \, dy = \infty$.

\noindent \textbf{Case 3:}  $\eta  > 1$. From
\eqref{id:3} we see that the $L^1(x^\gamma dx)$-boundedness of $T_\eta$ would lead to
\begin{align*}
\log \Big( 1 + \frac{a}{b - a } \Big)
\lesssim
1, \qquad 0< b/2 < a < b. 
\end{align*}
Taking once again $b=1$ and letting $a\to 1^{-}$ we end up with a contradiction.

The proof of \eqref{st:3} is finished.
This completes justifications of the negative results contained in the lemma.

We pass to the positive part of Lemma~\ref{lem:Tchar}. 
As a preparatory observation, note that by H\"older's inequality and \eqref{id:1}
\begin{align} \nonumber
T_{\eta} f (x)
& \lesssim
\sup_{t>2x} t^{\eta - 1} \bigg( \int_{t/2}^t \abs{f(z)}^p \, dz \bigg)^{1/p}
\big\| \ind{(t/2, t)} (t- (\cdot) +x)^{-\eta} \big\|_{L^{p'}(dx)} \\ \label{id:4}
& \simeq 
\sup_{t>2x} t^{\eta - 1} \bigg( \int_{t/2}^t \abs{f(z)}^p \, dz \bigg)^{1/p}
\begin{cases}
t^{1/p' - \eta} & \text{ if } \eta < 1/p',\\
\big( \log \big( 1 + t/x\big) \big)^{1/p'} & \text{ if } \eta = 1/p',\\
x^{1/p' - \eta} & \text{ if } \eta > 1/p'.
\end{cases}
\end{align}

The proof of the positive part of Lemma~\ref{lem:Tchar} splits into showing the following statements
(with all the mapping properties referring to the measure $x^{\gamma}dx$).
\begin{align} \label{st:4} 
& \text{If $\eta < 1/p'$ and $\gamma > -1$, then $T_{\eta}$ is of strong type $(p,p)$.} 
\\ \label{st:5} 
& \text{If $\eta = 1/p'$ and $\gamma > -1$, then $T_{\eta}$ is of strong type $(p,p)$.} 
\\ \label{st:6}
& \text{If $\eta > 1/p'$ and $\gamma > p(\eta-1)$, then $T_{\eta}$ is of strong type $(p,p)$.} 
\\ \label{st:7}
& \text{If $\eta > 1/p'$, $\gamma = p(\eta-1)$ and $p>1$, then $T_{\eta}$ is of strong type $(p,p)$.}
\\ \label{st:8}
& \text{If $\eta > 1/p'$, $\gamma = p(\eta-1)$ and $p=1$, then $T_{\eta}$ is of weak type $(1,1)$.} 
\end{align}

\noindent \textbf{Proof of \eqref{st:4}}. Since $\gamma > -1$ we can fix any $\tau \in [-1,\gamma)$.
Then, using \eqref{id:4} we see that
\begin{align*} 
T_{\eta} f (x) 
\lesssim
\sup_{t>2x} t^{- 1/p - \tau/p} \bigg( \int_{x}^\infty \abs{f(z)}^p z^{\tau} \, dz \bigg)^{1/p}
\lesssim
x^{- 1/p - \tau/p} \bigg( \int_{x}^\infty \abs{f(z)}^p z^{\tau} \, dz \bigg)^{1/p}.
\end{align*}
This leads to
\begin{align*} 
\| T_{\eta} f \|_{L^p(x^\gamma dx)}^p
& \lesssim
\int_0^\infty x^{- \tau - 1 + \gamma } 
\int_x^\infty \abs{f(z)}^p z^{\tau} \, dz \, dx
= 
\int_0^\infty \abs{f(z)}^p z^{\tau} \int_0^z x^{- \tau - 1 + \gamma } \, dx \, dz \\
& \simeq 
\| f \|_{L^p(x^\gamma dx)}^p,
\end{align*}
which proves \eqref{st:4}.

\noindent \textbf{Proof of \eqref{st:5}}. Since $\gamma > -1$ we may fix $\tau$ such that 
$-1 = p(\eta - 1) < \tau < \gamma$. Then, using \eqref{id:4} we get
\begin{align*} 
T_{\eta} f (x) 
& \lesssim
x^{\eta - 1 - \tau/p} 
\bigg( \int_{x}^\infty \abs{f(z)}^p z^{\tau} \, dz \bigg)^{1/p}
\sup_{t>2x} (t/x)^{\eta - 1 - \tau/p} 
\Big( \log \big( 1 + t/x\big) \Big)^{1/p'} \\
& \lesssim
x^{\eta - 1 - \tau/p} 
\bigg( \int_{x}^\infty \abs{f(z)}^p z^{\tau} \, dz \bigg)^{1/p}.
\end{align*}
This gives
\begin{align*} 
\| T_{\eta} f \|_{L^p(x^\gamma dx)}^p
& \lesssim
\int_0^\infty x^{- 1 - \tau + \gamma } 
\int_x^\infty \abs{f(z)}^p z^{\tau} \, dz \, dx
= 
\int_0^\infty \abs{f(z)}^p z^{\tau} \int_0^z x^{- \tau - 1 + \gamma } \, dx \, dz \\
& \simeq 
\| f \|_{L^p(x^\gamma dx)}^p,
\end{align*}
which justifies \eqref{st:5}.

\noindent \textbf{Proof of \eqref{st:6}}.
Here we fix any $\tau \in [ p(\eta - 1),\gamma)$. Then, using \eqref{id:4} we obtain
\begin{align*} 
T_{\eta} f (x) 
\lesssim
\sup_{t>2x} t^{\eta - 1 - \tau/p} \bigg( \int_{x}^\infty \abs{f(z)}^p z^{\tau} \, dz \bigg)^{1/p} x^{-\eta + 1/p'}
\simeq 
x^{- 1/p - \tau/p} \bigg( \int_{x}^\infty \abs{f(z)}^p z^{\tau} \, dz \bigg)^{1/p}.
\end{align*}
This leads to
\begin{align*} 
\| T_{\eta} f \|_{L^p(x^\gamma dx)}^p
& \lesssim
\int_0^\infty x^{- \tau - 1 + \gamma } 
\int_x^\infty \abs{f(z)}^p z^{\tau} \, dz \, dx
= 
\int_0^\infty \abs{f(z)}^p z^{\tau} \int_0^z x^{- \tau - 1 + \gamma } \, dx \, dz \\
& \simeq 
\| f \|_{L^p(x^\gamma dx)}^p,
\end{align*}
which gives \eqref{st:6}.

\noindent \textbf{Proof of \eqref{st:7}}. Here we have $\eta > 0$. Define 
\begin{align*}
\widetilde{T}_{\eta} f (x) 
=
x^{\eta - 1} T_{\eta} \big( (\cdot)^{1-\eta} f  \big) (x)
=
x^{\eta - 1} \sup_{t>2x} \int_{t/2}^t \frac{\abs{f(z)} }{(t-z+x)^{\eta}} \, dz, 
\qquad x > 0.
\end{align*}
By Lemma \ref{lem:chvar}, our task is reduced to showing that
$\widetilde{T}_{\eta}$ is bounded on $L^p(dx)$ for all $1 < p < \infty$ satisfying $\eta > 1/p'$,
i.e.\ for $1< p < \infty$ if $\eta \ge 1$ or for $1 < p < 1/(1-\eta)$ if $\eta \in (0,1)$.
Then, by interpolation, it is enough to verify that
\begin{align} \label{id:5} 
& \text{$\widetilde{T}_{\eta}$ is of weak type $(1,1)$ if $\eta>0$,} \\ \label{id:6} 
& \text{$\widetilde{T}_{\eta}$ is of weak type $(p,p)$ for all $1 < p < 1/(1-\eta)$ if $\eta \in (0,1)$,} \\ \label{id:7}  
& \text{$\widetilde{T}_{\eta}$ is of weak type $(p,p)$ 
for all $1 < p < \infty$ if $\eta = 1$,} \\ \label{id:8} 
& \text{$\widetilde{T}_{\eta}$ is of strong type $(\infty,\infty)$ if $\eta > 1$.}  
\end{align}

Considering \eqref{id:5}, observe that
\begin{align*} 
\widetilde{T}_{\eta} f (x)
\lesssim
x^{-1} \int_0^\infty \abs{f(z)} \, dz,
\end{align*}
which easily implies the weak type $(1,1)$ of $\widetilde{T}_{\eta}$.

Next, we treat \eqref{id:6} and \eqref{id:7} together. Notice that
by H\"older's inequality and \eqref{id:1}
\begin{align*} 
\widetilde{T}_{\eta} f (x) 
& \lesssim
x^{\eta - 1}
\sup_{t>2x} \bigg( \int_{t/2}^t \abs{f(z)}^p \, dz \bigg)^{1/p}
 \bigg( \int_{t/2}^t (t-z+x)^{-\eta p'} \, dz \bigg)^{1/p'} \\ 
& \lesssim
x^{- 1/p} \| f \|_{L^p( dx)}.
\end{align*}
This readily implies the properties asserted in \eqref{id:6} and \eqref{id:7}.

Finally, we prove \eqref{id:8}. With the aid of \eqref{id:1} we obtain  
\begin{align*} 
\widetilde{T}_{\eta} f (x)
\lesssim
\| f \|_{L^\infty}
x^{\eta - 1} \sup_{t>2x} \int_{t/2}^t (t-z+x)^{-\eta} \, dz 
\simeq
\| f \|_{L^\infty}.
\end{align*}
Thus \eqref{id:8} follows and the proof of \eqref{st:7} is finished.

\noindent \textbf{Proof of \eqref{st:8}}.
Here $\eta > 0$ and $\gamma = \eta-1$. Observe that by the very definition of $T_{\eta}$
\begin{align*} 
T_{\eta} f (x) 
\le
x^{-\eta} \int_{x}^{\infty} z^{\eta - 1} \abs{f(z)} \, dz 
\le
x^{-\gamma - 1}
\| f \|_{L^1(x^\gamma dx)},
\end{align*}
which directly leads to \eqref{st:8}.

The proof of Lemma \ref{lem:Tchar} is complete.
\end{proof}

\subsection{Proof of the mapping properties of $S_{\a,\b}$} \label{ssec:Sab}

\begin{proof}[{Proof of Lemma \ref{lem:5}}]
Define, see \eqref{chvar}, 
\begin{align} \label{id:11}
\widetilde{S}_{\alpha, \beta} f (x)
& = 
x^{-\beta} S_{\alpha, \beta} \big(  (\cdot)^{\beta}  f \big) (x) \\ \nonumber
& \simeq
x^{-\beta}
\sup_{t > 2x} \int_{t/2}^{t} (t + x - z)^{-\alpha - 1/2}
(t - z)^{\alpha +\beta - 1/2} \abs{f(z)} \, dz, \qquad x > 0.
\end{align}
In view of Lemma \ref{lem:chvar},
our task reduces to showing that for $1 < p < \infty$ the operator $\widetilde{S}_{\alpha, \beta}$
is bounded on $L^p(dx)$ if and only if $p > (\alpha + \beta + 1/2)^{-1}$.

We first deal with the positive part of the above statement. By interpolation, it suffices to show that
$\widetilde{S}_{\alpha, \beta}$ is of weak type $(p,p)$ with respect to $(\mathbb{R}_{+},dx)$ for all
$p > (\alpha + \beta + 1/2)^{-1}$.
Using H\"older's inequality and Lemma~\ref{lem:F2ref2}
(notice that for $p > (\alpha + \beta + 1/2)^{-1}$ one has $(\alpha + \beta -1/2)p' > -1$), we get
\begin{align} 
& \widetilde{S}_{\alpha, \beta} f (x) \label{bbb} \\
& \quad \lesssim
x^{-\beta}
\sup_{t > 2x} \bigg( \int_{t/2}^{t} \abs{f(z)}^p \, dz \bigg)^{1/p} 
\bigg( \int_{t/2}^{t} (t + x - z)^{(-\alpha - 1/2)p'}
(t - z)^{(\alpha +\beta - 1/2)p'} \, dz \bigg)^{1/p'} \nonumber \\
& \quad \lesssim
x^{-1/p}
\|f\|_{L^p(dx)}. \nonumber
\end{align}
This implies the weak type $(p,p)$ estimate and shows the positive part of the lemma.

Next, we show the negative part of Lemma~\ref{lem:5}, which is equivalent to proving that the maximal operator
$\widetilde{S}_{\alpha, \beta}$ is not bounded on $L^p(dx)$ if $1 < p \le (\alpha + \beta + 1/2)^{-1}$.
To do so, let $a,b > 0$ be such that $0< b/2 < a < b$ and put
\begin{align*} 
f(x) = \ind{(a,b)} (x), \qquad x > 0.
\end{align*}

We first deal with the main case $\b < 0$.
Observe that using Lemma~\ref{lem:F2ref2} we have, uniformly in $x \le a/4$ and $0< b/2 < a < b$,
\begin{align*} 
\widetilde{S}_{\alpha, \beta} f (x) 
& \simeq
\sup_{2b > t > a} 
x^{-\beta}
\int_{(t/2) \vee a}^{t\wedge b} (t - z + x)^{-\alpha - 1/2}
(t - z)^{\alpha +\beta - 1/2} \, dz \\
& \simeq
x^{-\beta} \sup_{2b > t > a}
\int^{t - (t/2) \vee a}_{t - t\wedge b} ( z + x)^{-\alpha - 1/2}
z^{\alpha +\beta - 1/2} \, dz \\
& \simeq
x^{-\beta} \sup_{2b > t > a}
[t\wedge b - (t/2)\vee a] [t - (t/2)\vee a]^{\alpha +\beta - 1/2}
[t - t\wedge b + x]^{\beta} \\
& \qquad \qquad \qquad \times 
[t - (t/2)\vee a + x]^{-(\alpha +\beta + 1/2)}.
\end{align*}
Taking $t=b$ above, we arrive at the estimate
\begin{align*} 
\widetilde{S}_{\alpha, \beta} f (x) 
& \gtrsim
(b-a)^{\alpha +\beta + 1/2} 
(b-a + x)^{-(\alpha +\beta + 1/2)}, \qquad 0 < x \le a/4, \quad 0< b/2 < a < b.
\end{align*}
Therefore, assuming $\widetilde{S}_{\alpha, \beta}$ is bounded on $L^p(dx)$, we have 
\begin{align} \label{id:9}
\int_0^{a/4} (b-a)^{p(\alpha +\beta + 1/2)} 
(b-a + x)^{-p(\alpha +\beta + 1/2)} \, dx
\lesssim
b-a, \qquad 0< b/2 < a < b.
\end{align}

We now show that \eqref{id:9} implies that $\gamma := p(\alpha +\beta + 1/2) > 1$,
which gives the conclusion we need.
Observe that using \eqref{id:1} we get
\begin{align*} 
\texttt{LHS}\eqref{id:9} 
\simeq 
(b-a)^{\gamma} \int_{b-a}^{b-a + a/4} x^{-\gamma} \, dx
\simeq
(b-a)^{\gamma}
\begin{cases}
b^{-\gamma + 1} & \text{ if } \gamma < 1,\\
\log \big( 1 + a/(b-a) \big) & \text{ if } \gamma = 1.
\end{cases}
\end{align*}
Combining this with \eqref{id:9}, taking $b=1$ and letting $a \to 1^{-}$, we see that $\gamma$ cannot be less or equal to $1$.

The case $\b = 0$ is similar, the only difference is that we get a logarithmic lower bound from Lemma \ref{lem:F2ref2}
as follows:
\begin{align*} 
& \widetilde{S}_{\alpha, \beta} f (x) \\
& \quad \simeq
\sup_{2b > t > a}
\int^{t - (t/2) \vee a}_{t - t\wedge b} ( z + x)^{-\alpha - 1/2}
z^{\alpha - 1/2} \, dz \\
& \quad \gtrsim
(b-a)^{\alpha - 1/2} 
(b-a + x)^{-\alpha + 1/2} \log \Big( 1 + \frac{b-a}{x} \Big), 
\qquad 0 < x \le \frac{a}4, \quad 0< \frac{b}2 < a < b.
\end{align*}
Then, it suffices to analyze the integral related to the interval $((b-a)/100, a/4)$, which is straightforward.
\end{proof}

\begin{proof}[{Proof of Lemma \ref{lem:5a}}]
We invoke the operator $\widetilde{S}_{\alpha, \beta}$ from \eqref{id:11} and use Lemma~\ref{lem:chvar}
to reduce our task to showing that  
$\widetilde{S}_{\alpha, \beta}$ is bounded on $L^p(dx)$ if $(\alpha + \beta + 1/2)^{-1} < p < \beta^{-1}$.
By interpolation, it suffices to show that $\widetilde{S}_{\alpha, \beta}$ is of weak type $(p,p)$ with respect to
$(\mathbb{R}_+,dx)$ for all $p$ just indicated.

Applying H\"older's inequality and Lemma~\ref{lem:F2ref2} (notice that for $(\alpha + \beta + 1/2)^{-1} < p < \beta^{-1}$ one has
$(\alpha + \beta -1/2)p' > -1$ and $(\beta-1)p' < -1$) we get
the same estimates as in \eqref{bbb}, which
implies the weak type $(p,p)$ of $\widetilde{S}_{\alpha, \beta}$ and finishes the proof.
\end{proof}

\subsection{Proof of the mapping properties of $S_{\a,\b}^{\log}$} \label{ssec:Slog}

In this subsection we consider $\a$ and $\b$ satisfying $\a+\b=1/2$.
To prove Lemma \ref{lem:5b} we need the following technical result.
\begin{lem} \label{lem:D1}
Let $\xi < -1$ and $\lambda > -1$ be fixed. Then
\begin{align} \label{D1.3}
\int_a^b (A+z)^{\xi} \bigg( \log\Big(2 + \frac{A}{z} \Big) \bigg)^{\lambda} \, dz
\simeq
(b-a) (a + A)^{\xi + 1} (b + A)^{- 1} 
\bigg( \log\Big(2 + \frac{A}{b} \Big) \bigg)^{\lambda}, 
\end{align}
uniformly in $0 \le a \le b$ and $A>0$.
\end{lem}

In the proof of Lemma~\ref{lem:D1} we shall use a relation, which can easily be deduced from e.g.\ \cite[Lemma 2.4]{NSS}.
Namely, for a fixed $\lambda > -1$ we have
\begin{align*} 
\int_a^b s^{\lambda} e^{-s} \, ds 
\simeq
[(b-a) \wedge 1] (a+1)^{\lambda} (b \wedge 1)^{\lambda} e^{-a}, 
\qquad 0 \le a \le b < \infty.
\end{align*}
Changing the variable $s \mapsto -(\zeta+1)\log y$ above,
for any $\zeta < -1$ and $\lambda > -1$ fixed we get
\begin{align} \label{D1.2}
\int_X^Y \big( \log y \big)^{\lambda} y^{\zeta} \, dy
\simeq
X^{\zeta + 1} \bigg[ 1 \wedge \log \frac{Y}{X} \bigg] [1 + \log X]^{\lambda}
[ 1 \wedge \log Y]^{\lambda}, 
\qquad Y \ge X \ge 1.
\end{align}

\begin{proof}[Proof of Lemma~\ref{lem:D1}]
We split the reasoning into three cases.

\noindent \textbf{Case 1:} $b \le 2A$. Changing the variable 
$10A/z \mapsto y$ and using \eqref{D1.2}, we obtain
\begin{align*} 
\texttt{LHS}\eqref{D1.3} 
& \simeq
A^{\xi} \int_a^b \bigg( \log\Big(\frac{10 A}{z} \Big) \bigg)^{\lambda} \, dz
\simeq 
A^{\xi + 1} \int^{10A/a}_{10A/b} \big( \log y \big)^{\lambda} y^{-2} \, dy \\
& \simeq 
A^{\xi} b \bigg[ 1 \wedge \log \frac{b}{a} \bigg] 
\big( \log (1 + A/b) \big)^{\lambda}.
\end{align*}
Since $1 \wedge \log \frac{b}{a} \simeq (b-a)/b$, $0 \le a \le b$, we get the desired estimate.

\noindent \textbf{Case 2:} $a \ge A$. Now $A/z \le A/a \le 1$ for $a \le z \le b$, so using \eqref{id:1} we have 
\begin{align*} 
\texttt{LHS}\eqref{D1.3} 
& \simeq
\int_a^b z^{\xi} \, dz
\simeq
(b-a) b^{-1} a^{\xi + 1} 
\simeq
\texttt{RHS}\eqref{D1.3}. 
\end{align*}

\noindent \textbf{Case 3:} $a \le A$ and $b \ge 2A$. Here we split the integral in question into two parts and use
the already justified estimates from Case 1 and Case 2 to get
\begin{align*} 
\texttt{LHS}\eqref{D1.3}
\simeq
\bigg( \int_a^{3A/2} + \int_{3A/2}^b \bigg) \ldots
\simeq
A^{\xi + 1} 
& \simeq 
\texttt{RHS}\eqref{D1.3}. 
\end{align*}
This finishes the proof of Lemma~\ref{lem:D1}.
\end{proof}

\begin{proof}[Proof of Lemma~\ref{lem:5b}]
Define, see \eqref{chvar}, 
\begin{align} \label{id:13}
\widetilde{S}_{\alpha, \beta}^{\,\log} f (x)
& = 
x^{-\beta} S_{\alpha, \beta}^{\log} \big( (\cdot)^{\beta} f  \big) (x) \\ \nonumber
& \simeq
x^{-\beta}
\sup_{t > 2x} \int_{t/2}^{t} (t + x - z)^{-\alpha - 1/2}
\log \Big(2 + \frac{x}{t - z} \Big) \abs{f(z)} \, dz, \qquad x > 0.
\end{align}
Using Lemma \ref{lem:chvar} reduces the task to showing that for $1 < p < \infty$ the operator
$\widetilde{S}_{\alpha, \beta}^{\, \log}$ is bounded on $L^p(dx)$ if $\beta p < 1$.

Observe that for $\beta \ge 1$ there is nothing to prove, therefore we may assume that $\beta < 1$.
By interpolation, it suffices to show that $\widetilde{S}_{\alpha, \beta}^{\, \log}$ is of weak type $(p,p)$ with respect to
$(\mathbb{R}_+,dx)$ for all $p > 1$ satisfying $\beta p < 1$.
Using H\"older's inequality 
and Lemma~\ref{lem:D1} (notice that for $\beta p < 1$ we have 
$(-\alpha - 1/2)p' < -1$) we get
\begin{align*} 
& \widetilde{S}_{\alpha, \beta}^{\, \log} f (x) \\
& \quad \lesssim
x^{-\beta}
\sup_{t > 2x} \bigg( \int_{t/2}^{t} \abs{f(z)}^p \, dz \bigg)^{1/p} 
\bigg( \int_{t/2}^{t} (t + x - z)^{(-\alpha - 1/2)p'}
\bigg( \log \Big(2 + \frac{x}{t - z} \Big) \bigg)^{p'} \, dz \bigg)^{1/p'} \\
& \quad \lesssim
x^{-\beta}
\bigg( \int_{0}^{\infty} \abs{f(z)}^p \, dz \bigg)^{1/p} 
\sup_{t > 2x}
\bigg( \int_0^{t/2} ( x + z)^{(-\alpha - 1/2)p'}
\bigg( \log \Big(2 + \frac{x}{z} \Big) \bigg)^{p'} \, dz \bigg)^{1/p'} \\
& \quad \lesssim
x^{-1/p} \|f\|_{L^p(dx)}.
\end{align*}
This readily implies the weak type $(p,p)$ of $\widetilde{S}_{\alpha, \beta}^{\, \log}$ in the asserted range of $p$.
\end{proof}

\subsection{Proof of the mapping properties of $R_{\a,\b}^{\log}$} \label{ssec:Rlog}

In this subsection $\a + \b =1/2$.

\begin{proof}[{Proof of Lemma \ref{lem:Rlog}}]
Because of Lemma \ref{lem:chvar}, it is enough we prove that
$$
\widetilde{R}_{\a,\b}^{\, \log}f(x) = \sup_{t > 3x} \frac{1}{x} \int_{t-x}^{t+x} \log\bigg( \frac{4x}{z-(t-x)}\bigg) |f(z)|\, dz,
\qquad x >0,
$$
is bounded on $L^p(dx)$. Using H\"older's inequality we get
$$
\widetilde{R}_{\a,\b}^{\, \log}f(x) \le \sup_{t > 3x}\frac{1}x \bigg( \int_x^{\infty} |f(z)|^p\, dz\bigg)^{1/p}
	\bigg( \int_{t-x}^{t+x} \bigg[\log\bigg(\frac{4x}{z-(t-x)}\bigg)\bigg]^{p'}\, dz \bigg)^{1/p'}.
$$
The second integral here is comparable to $x$, which follows instantly by a simple change of the variable of integration.
Thus
$$
\widetilde{R}_{\a,\b}^{\, \log}f(x) \lesssim x^{-1/p} \|f\|_{L^p(dx)}.
$$
Consequently, $\widetilde{R}_{\a,\b}^{\, \log}$ is of weak type $(p,p)$ for any $1 < p < \infty$.
Now the $L^p$-boundedness follows by interpolation.
\end{proof}

\section*{Appendix}

The purpose of this section is to explain briefly what happens in the limiting case $\a+\b = -1/2$,
in particular to give more credit to the results by Colzani et al.\ \cite{CoCoSt} which perhaps was
not done properly enough in our previous papers \cite{CNR,CNR2}.
To this end we always assume that $\a > -1$.

Recall (cf.\ \cite[Section 1]{CNR}) that the integral operator $M_t^{\a,\b}$, $\a+\b > -1/2$,
originates from the Hankel transform multiplier operator
\begin{equation*} 
\mathcal{M}_t^{\a,\b}f = \mathcal{H}_{\a}\big( m_{\a,\b}(t\, \cdot) \mathcal{H}_{\a}f\big), \qquad t > 0,
\end{equation*}
where
$$
m_{\a,\b}(s) = 2^{\a+\b} \Gamma(\a+\b+1) \frac{J_{\a+\b}(s)}{s^{\a+\b}}, \qquad s > 0.
$$
Clearly, the above makes sense also in the case $\a+\b=-1/2$, since then the multiplier
$m_{\a,\b}(t\cdot)$ is a bounded function. However, a standard integral representation of $\mathcal{M}_t^{\a,\b}$
exists only when $\a+\b > -1/2$, and this is, roughly speaking, thanks to a decay of $m_{\a,b}(t\,\cdot)$ at infinity.
Nevertheless, as observed in \cite{CoCoSt}, $\mathcal{M}_t^{\a,\b}$ in the limiting case $\a+\b=-1/2$ can in general
be represented as a principal value integral operator plus some extra terms.

To proceed, we first define the kernel $K_t^{\a,\b}(x,z)$ when $\a+\b=-1/2$ in a way that shows consistency
with the case $\a+\b > -1/2$. Observe that the triple Bessel function integral defining $K_t^{\a,\b}(x,z)$
for $\a+\b > -1/2$ diverges at infinity (even in the Riemann sense) when $\a+\b=-1/2$.
Therefore we proceed in another way, first compute the integral and then allow $\a+\b=-1/2$ in the resulting expression.
Recall from \cite[Section 3.1]{CNR} that
\begin{align*}
K_t^{\a,\b}(x,z) & = \frac{2^{\a+\b}\Gamma(\a+\b+1)}{\sqrt{2\pi}} \frac{(xz)^{\b-1}}{t^{2\a+2\b}} \\
										& \quad \times
										\begin{cases}
												0, & t < |x-z|, \\												
												(\sin v)^{\a+\b-1/2}\; \mathsf{P}_{\a-1/2}^{1/2-\a-\b}(\cos v), & |x-z| < t < x+z,\\
												\frac{2}{\Gamma(\b)}
												(\sinh u)^{\a+\b-1/2}\; \mathbf{Q}_{\a-1/2}^{1/2-\a-\b}(\cosh u), & x+z < t,
										\end{cases}
\end{align*}
where
$$
v = \arccos\frac{x^2+z^2-t^2}{2xz}, \qquad u = \arcosh\frac{t^2-x^2-z^2}{2xz},
$$
$\mathsf{P}$ is the Ferrers function of the first kind (the associated Legendre function of the first kind on the cut),
and $\mathbf{Q}$ is Olver's function (renormalized associated Legendre function of the second kind); for more information
on $\mathsf{P}$ and $\mathbf{Q}$ we refer to \cite[Chapter 14]{handbook}.
Note that the above formula makes sense also when $\a+\b=-1/2$, and we take it as the definition of $K_t^{\a,\b}(x,z)$
in this case.

Assume now that $\a+\b=-1/2$, i.e.\ $\b=-1/2-\a$. One can verify that $K_t^{\a,-1/2-\a}(x,z)$ coincides,
up to a factor emerging from
the density of $\mu_{\a}$, with the kernel $\frac{\partial}{\partial t}\mathbb{K}(t,x,z)$ defined in \cite{CoCoSt} in terms
of the Gauss hypergeometric function. 
Thus, following \cite[Section 1]{CoCoSt}, we define $M_t^{\a,\b}$, $t>0$, in the limiting
case $\a+\b=-1/2$ by
\begin{align*}
M_t^{\a,-1/2-\a}f(x) & = \pv \int_0^{\infty} K_t^{\a,-1/2-\a}(x,z) f(z)\, d\mu_{\a}(z) \\
	& \quad + \frac{|x-t|^{\a+1/2}f(|x-t|) + (x+t)^{\a+1/2}f(x+t)}{2x^{\a+1/2}} \\
	& \quad	- \sin^2\Big(\frac{\pi}2 \big(\a+1/2\big)\Big)\; \chi_{\{x < t\}} \; \frac{(t-x)^{\a+1/2}f(t-x)}{x^{\a+1/2}}, \qquad x > 0,
\end{align*}
for all $f$ for which the formula makes sense.
Here the principal value pertains to a non-integrable singularity occurring at $z=t-x$ in cases when
$t > x$ and $\a+1/2$ is not integer.
The operator $M_t^{\a,-1/2-\a}$ coincides with $\mathcal{M}_t^{\a,-1/2-\a}$ in $L^2(d\mu_{\a})$.
This was essentially shown in \cite{CoCoSt} under a slight restriction $\a \ge -1/2$ which is not really necessary.

Note that the kernel $K_t^{\a,-1/2-\a}(x,z)$ vanishes completely when $\a=-1/2$.
When $\a+1/2 \in \mathbb{Z}$ this kernel vanishes outside the set $\{(t,x,z) \in \mathbb{R}_+^3 : |x-z| < t < x+z\}$ and,
moreover, has no non-integrable singularities and expresses via elementary functions (see \cite[Corollary 1.2]{CoCoSt}).
In case $\a+1/2 \notin \mathbb{Z}$ the kernel is a transcendental function with a non-integrable singularity at $z=t-x$
of order of magnitude $(t-x-z)^{-1}$.
Some estimates and asymptotics of $K_t^{\a,-1/2-\a}(x,z)$ can be found in \cite{CoCoSt}. Further properties of
the kernel can be concluded from the theory of associated Legendre functions, found e.g.\ in \cite[Chapter 14]{handbook},
in a similar manner as it was done in \cite[Section 3]{CNR}.

Concerning the maximal operator $M_t^{\a,-1/2-\a}$, no results in the spirit of the present paper are possible
(there are simply no strong, weak and restricted weak type $(p,p)$, $p< \infty$, bounds). Nonetheless, as a substitute
there are non-trivial results of this kind for suitable averaging operators related to $M_t^{\a,-1/2-\a}$,
see \cite{CoCoSt} and references therein; see also \cite{MS}.

Finally, we point out that via $M_t^{\a,-1/2-\a}$ (possibly together with $M_t^{\a,1/2-\a}$) one can express
general solutions to several classical Cauchy initial-value problems with radial initial data.
This in particular pertains to the wave, and more generally the Euler-Poisson-Darboux, equation in $\mathbb{R}^n$,
see e.g.\ \cite{CoCoSt} and \cite[Section 7]{CNR}$^{\ddag}$, and also references given there.
{\footnote{$\ddag$ There is a misprint in \cite[Section 7, p.\,4438, l.\,10]{CNR}, there should be no factor $t$ in front of
	$\mathcal{M}_t^{\a,-\a-1/2}$.}


\end{document}